\documentclass[11pt]{amsart}
\usepackage[all]{xy}
\usepackage{amscd,amsthm,amssymb,amsfonts,amsmath}
\usepackage[hidelinks]{hyperref}
\usepackage[margin=0.7in]{geometry}
\usepackage{tikz-cd}
\usepackage{tikz}
\usepackage{bbm}

\usepackage{graphicx}
\theoremstyle{plain}
\newtheorem{thm}{Theorem}[section]
\newtheorem{lemma}[thm]{Lemma}
\newtheorem{prop}[thm]{Proposition}
\newtheorem{cor}[thm]{Corollary}

\theoremstyle{definition}
\newtheorem{defn}[thm]{Definition}

\theoremstyle{remark}
\newtheorem{remark}[thm]{Remark}

\newcommand{\nc}{\newcommand}

\numberwithin{equation}{section}
\def\makeop#1{\expandafter\def\csname#1\endcsname
  {\mathop{\rm #1}\nolimits}\ignorespaces}
\makeop{Hom}   \makeop{End}   \makeop{Aut}   \makeop{Isom}  \makeop{Pic}
\makeop{Gal}   \makeop{ord}   \makeop{Char}  \makeop{Div}   \makeop{Lie}
\makeop{PGL}   \makeop{sym}  \makeop{PSL}   \makeop{sgn}   \makeop{Spf}
\makeop{Spec}  \makeop{Map}    \makeop{Re}    \makeop{Fr}    \makeop{disc}
\makeop{Proj}  \makeop{supp}  \makeop{ker}   \makeop{im}    \makeop{cond}
\makeop{coker} \makeop{Stab}  \makeop{SO}    \makeop{SL}    \makeop{Res}
\makeop{Cl}    \makeop{cond}  \makeop{Br}    \makeop{inv}   \makeop{rank}
\makeop{id}    \makeop{Fil}   \makeop{Frac}  \makeop{GL}    \makeop{SU}
\makeop{Trd}   \makeop{Sp}    \makeop{Tr}    \makeop{Trd}   \makeop{diag}
\makeop{Res}   \makeop{ind}   \makeop{depth} \makeop{Tr}    \makeop{st}
\makeop{Ad}    \makeop{Int}   \makeop{tr}    \makeop{Sym}   \makeop{can}
\makeop{length}\makeop{SO}    \makeop{torsion} \makeop{GSp} \makeop{Ker}
\makeop{rk}   \makeop{Frob}  \makeop{id}    \makeop{Tor}   \makeop{Ind}
\makeop{CoInd} \makeop{Inf}   \makeop{Cor}   \makeop{CoInf} \makeop{coLie}
\makeop{quo}   \makeop{sub}   \makeop{sgn}   \makeop{min} \makeop{lcm}   \makeop{re}
\makeop{ord}   \makeop{val}   \makeop{Lie}   \makeop{res}   \makeop{Ann}  \makeop{Ha}

\def\makebb#1{\expandafter\def
  \csname bb#1\endcsname{{\mathbb{#1}}}\ignorespaces}
\def\makebf#1{\expandafter\def\csname bf#1\endcsname{{\bf
      #1}}\ignorespaces}
\def\makegr#1{\expandafter\def
  \csname gr#1\endcsname{{\mathfrak{#1}}}\ignorespaces}
\def\makescr#1{\expandafter\def
  \csname scr#1\endcsname{{\EuScript{#1}}}\ignorespaces}
\def\makecal#1{\expandafter\def\csname cal#1\endcsname{{\mathcal
      #1}}\ignorespaces}
\def\doLetters#1{#1A #1B #1C #1D #1E #1F #1G #1H #1I #1J #1K #1L #1M
                 #1N #1O #1P #1Q #1R #1S #1T #1U #1V #1W #1X #1Y #1Z}
\def\doletters#1{#1a #1b #1c #1d #1e #1f #1g #1h #1i #1j #1k #1l #1m
                 #1n #1o #1p #1q #1r #1s #1t #1u #1v #1w #1x #1y #1z}
\doLetters\makebb   \doLetters\makecal  \doLetters\makebf
\doLetters\makescr
\doletters\makebf   \doLetters\makegr   \doletters\makegr

\normalsize

\makeop{Bl}

\def\Spec{{\rm Spec}\,}

\def\Qpbar{\overline{{\bbQ}_p}}
\def\Qp{{\bbQ}_p}

\def\Zp{{\bbZ}_p}
\def\Qbar{\overline{\bbQ}}

\newcommand{\Z}{\mathbb Z}
\newcommand{\Q}{\mathbb Q}
\newcommand{\R}{\mathbb R}
\newcommand{\C}{\mathbb C}
\newcommand{\N}{\mathbb N}
\newcommand{\A}{\mathbb A}    % for adele
\renewcommand{\O}{\mathcal O} % for sheaves

\newcommand{\n}{\mathfrak n}
\newcommand{\m}{\mathfrak m}
\newcommand{\q}{\mathfrak q}
\newcommand{\p}{\mathfrak p}

\newcommand{\<}{\langle}   %\< is not defined yet.
\renewcommand{\>}{\rangle} %\> is already defined.

\newcommand{\map}{\rightarrow}

\newcommand{\isom}{\cong}
\nc{\embed}{\hookrightarrow}
\newcommand{\surj}{\twoheadrightarrow}
\newcommand{\ch}{characteristic }

\nc{\ol}{\overline}
\nc{\wt}{\widetilde}
\nc{\wh}{\widehat}
\nc{\opp}{\mathrm{opp}}
\def\ul{\underline}
\def\mfr{\mathfrak}

\makeop{Ram}
\makeop{Rep}
\makeop{lcm}

\title[Congruence modules]
{On congruence modules related to Hilbert Eisenstein series}

\author{Sheng-Chi Shih }
\address{
University of Lille, CNRS, UMR 8524-Laboratoire Paul Painlev\'{e}, 59000 Lille, France.}
\email{sheng-chi.shih@univ-lille.fr}
\date{\today}

\begin{document}
\begin{abstract}
We generalize the work of Ohta on the congruence modules attached to elliptic Eisenstein series to the setting of Hilbert modular forms. Our work involves three parts. 
In the first part, we construct Eisenstein series adelically and compute their constant terms by computing local integrals. 
In the second part, we prove a control theorem for one-variable ordinary $\Lambda$-adic Hilbert modular forms following Hida's work on the space of multivariable ordinary $\Lambda$-adic Hilbert cusp forms.
In part three, we compute congruence modules related to Hilbert Eisenstein series through an analog of Ohta's methods.  
\end{abstract}

\maketitle

\addtocontents{toc}{\setcounter{tocdepth}{1}}% do not list subsections in contents
\tableofcontents

%\newpage
%%%%%%%%%%%%%%%%%%%%%%%%%%%%%%%%%%%%%%%%%%%%%%%%%%%%%%%%%%%%%%%%%%
\section{Introduction}\label{sec:01}
Let $R$ be an integral domain with quotient field $Q(R)$. We
consider a short exact sequence of flat $R$-modules
$$
0\map A \xrightarrow{i} B\xrightarrow{p}  C \map 0.
$$
Suppose that we are given splitting maps after tensoring with $Q(R)$
over $R$, i.e., we have
$$
0 \leftarrow A\otimes_R Q(R) \xleftarrow{t}  B\otimes_R Q(R) \xleftarrow{s} C\otimes_R Q(R) \leftarrow 0.
$$
The congruence module attached to these data is defined by
$$
\mathcal{C}_s:=C/p(B\cap s(C)).
$$
Congruence modules have been studied by many people in different settings. For instance, Ohta \cite{Ohta} computed the congruence module associated with the sequence
$$
0\map S^{\ord}(\Gamma;\Lambda)_E\map M^{\ord}(\Gamma; \Lambda)_E\xrightarrow{\res} \Lambda\map 0,
$$
where $``\res"$ is the residue map, and $M^{\ord}(\Gamma; \Lambda)$ and $S^{\ord}(\Gamma;\Lambda)$ are respectively the spaces of ordinary $\Lambda$-adic modular forms and ordinary $\Lambda$-adic cusp forms. Here $\Lambda=\mfr{o}[[T]]$ for some extension $\mfr{o}$ of $\Zp$. In this paper, we generalize Ohta's work to the setting of Hilbert modular forms. In order to achieve this goal, we review important facts about $p$-adic and $\Lambda$-adic Hilbert modular forms, and prove crucial results about Eisenstein series and cusps through their adelic construction. For the above examples and our main results (Theorem~\ref{11}), we require that splittings are Hecke-equivariant. Moreover, there exist canonical splittings which are considered in the computation of these congruence modules.   

Before we describe our main results, let us mention our motivation, which comes from Sharifi's conjecture \cite{Sha}. Sharifi's conjecture is a refinement of the Iwasawa main conjecture. The main conjecture asserts a relationship between two objects: one is a certain $p$-adic $L$-function and the other is a \ch polynomial associated with the $p$-part of the class group of the cyclotomic $\Zp$-extension of an abelian extension of $\Q$. Roughly, Sharifi's conjecture predicts that one can obtain the information on the second object from the cohomology of modular curves. The main conjecture over $\Q$ was first proved by Mazur\textendash Wiles \cite{MW} using $2$-dimensional Galois representations attached to cusp forms that are congruent to ordinary Eisenstein series. Wiles \cite{W3} generalized the method of Mazur\textendash Wiles to the setting of Hilbert modular forms and proved the main conjecture over totally real fields. Combining his previous works \cite{Ohta2} and \cite{Ohta3}, Ohta gave a refinement of Mazur\textendash Wiles's proof of the main conjecture over $\Q$  examining the action of $\Gal(\Qbar/\Q)$ on the Eisenstein component of the cohomology of modular curves. The cohomology of modular curves provides a canonical choice of a lattice in a Galois representation, which plays an important role in Sharifi's work. Our work is a first step to proving the main conjecture over totally real fields along the lines of Ohta's approach \cite{Ohta3} and to generalizing Sharifi's conjecture to totally real fields.  

To describe our results, we fix some notation first. 
\begin{itemize}
\item $F$ is a totally real field, $\O_F$ is the ring of integers of $F$, $h_F$ is the class number of $F$, and $\mfr{D}$ is the different of $F$ over $\Q$.

\item $p$ is an odd rational prime unramified in $F$.

\item $\mathbbm{1}$ is the trivial character. $\chi_1$ and $\chi_2$ are primitive narrow ray class characters of conductors $\mfr{n}_1$ and $\mfr{n}_2$, respectively, with associated signs $e_{1,\infty},e_{2,\infty} \in (\Z/2\Z)^d$ satisfying 
$e_{1,\infty}+e_{2,\infty}\equiv (0,\ldots,0) (\bmod\; 2\Z^d)$. We assume that $\chi_1$ is not a trivial character, $\mfr{n}_1\mfr{n}_2=\mfr{n}$ or $\mfr{n}p$ for some integral ideal $\mfr{n}$ not divisible by $p$, and $\mfr{n}_2$ is prime to $p$. 

\item Set $\Lambda=\Zp[\chi_1,\chi_2][[T]]$. $M^{\ord}(\mfr{n},\chi_1\chi_2;\Lambda)$ (resp.~$S^{\ord}(\mfr{n},\chi_1\chi_2;\Lambda)$) is the space of $p$-ordinary $\Lambda$-adic modular forms (resp.~cusp forms). $\mathcal{H}^{\ord}:=\mathcal{H}^{\ord}(\mfr{n},\chi,\Lambda)\subset \End_{\Lambda}(M^{\ord}(\mfr{n},\chi_1\chi_2;\Lambda))$ (resp.~$h^{\ord}:=h^{\ord}(\mfr{n},\chi;\Lambda)\subset \End_{\Lambda}(S^{\ord}(\mfr{n},\chi_1\chi_2;\Lambda))$) is the Hecke algebra (resp.~the cuspidal Hecke algebra) generated over $\Lambda$ by Hecke operators $U(\mfr{p})$, $T(\mfr{q})$, and $S(\mfr{q})$ for all prime ideals $\mfr{q}$ not dividing $\mfr{n}p$ and  for all prime ideals $\mfr{p}$ dividing $\mfr{n}p$.

\item $\mathcal{E}(\chi_1,\chi_2)$ is the $\Lambda$-adic Eisenstein series associated to $\chi_1$ and $\chi_2$ (see Proposition~\ref{lambda_adic_eisen} for the definition). 

\item $\mfr{M}$ (resp.~$\mfr{m}$) is the unique maximal ideal of $\mathcal{H}^{\ord}$ (resp.~$h^{\ord}$) containing the Eisenstein ideal $\mathcal{I}(\chi_1,\chi_2)=\Ann_{\mathcal{H}}(\mathcal{E}(\chi_1,\chi_2))$ (resp.~$I(\chi_1,\chi_2)$) associated to $\mathcal{E}(\chi_1,\chi_2)$. We denote by $\mathcal{H}^{\ord}_{\mfr{M}}$ (resp.~$h^{\ord}_{\mfr{m}}$) the localization of $\mathcal{H}^{\ord}$ (resp.~$h^{\ord}$) at $\mfr{M}$ (resp.~at $\mfr{m}$). 

\item $M^{\ord}(\mfr{n},\chi_1\chi_2;\Lambda)_{\mfr{M}}$ (resp.~$S^{\ord}(\mfr{n},\chi_1\chi_2;\Lambda)_{\mfr{M}}$) is the localization
of the space of $p$-ordinary $\Lambda$-adic modular forms (resp.~cusp forms) at $\mfr{M}$.

\item  $\omega:(\Z/p\Z)^{\times}\to \Qbar^{\times}$ is the Teichm\"{u}ller character. Set $\omega(\mfr{a}):=\omega(N_{F/\Q}(\mfr{a}))$ for all integral ideals $\mfr{a}$ of $F$ prime to $p$. This is a narrow ray class character of conductor $p$.

\item $\phi$ is Euler's Phi function.
\end{itemize}

\begin{thm}\label{11}
Assume that $(\chi_1,\chi_2) \neq (\omega^{-2},\mathbbm{1})$ and $\chi_1\chi_2^{-1}\omega(\mfr{p})\neq 1$ for some prime ideal $\mfr{p}|p$. If $p$ does not divide $N_{F/\Q}(\mfr{nD})\phi(N_{F/\Q}(\mfr{n}))h_F$, then the congruence modules attached to the
short exact sequences of flat $\Lambda$-modules
\begin{equation} \label{eq:1}
\begin{cases}
0\map S^{\ord}(\mfr{n},\chi_1\chi_2;\Lambda)_{\mfr{M}}\map M^{\ord}(\mfr{n},\chi_1\chi_2;\Lambda)_{\mfr{M}} \xrightarrow{C_0} \Lambda \map 0\\

0\map \mathcal{I}(\chi_1,\chi_2)\map \mathcal{H}^{\ord}_{\mfr{M}}\map \Lambda\map 0
\end{cases}
\end{equation}
are both $\Lambda/(A(\chi_1,\chi_2))$, where $A(\chi_1,\chi_2)\in \Lambda$ is a formal power series expression of a Deligne\textendash Ribet $p$-adic $L$-function, and the map $C_0$ maps each modular form to a formal sum of its constant terms at cusps. 
\end{thm}

For the congruence module attached to the first short exact sequence in (\ref{eq:1}), we use the splitting map that sends $1\in Q(\Lambda)$ to the $\Lambda$-adic modular form $\mathcal{E}(\chi_1,\chi_2)/A(\chi_1,\chi_2)$ in the space $\mathcal{E}_{Q(\Lambda)}$ generated by $\mathcal{E}(\chi_1,\chi_2)$ over $Q(\Lambda)$ (see p.~39). For the congruence module attached to the second short exact sequence in (\ref{eq:1}), we use the splitting map that sends $1\in Q(\Lambda)$ to an element in $\Hom_{Q(\Lambda)}(\mathcal{E}_{Q(\Lambda)},Q(\Lambda))$ sending $\mathcal{E}(\chi_1,\chi_2)/A(\chi_1,\chi_2)$ to $1$ (see p.~42). The map $C_0$ will be defined in Section~\ref{sec:24}, and the element $A(\chi_1,\chi_2)$ (will be defined in Section ~\ref{sec:61}) is related to the constant terms of the Eisenstein series $\mathcal{E}(\chi_1,\chi_2)$. 

Using the first short exact sequence in (\ref{eq:1}), we obtain a $\Lambda$-adic cusp form $\mathcal{F}_S$ (see Proposition~\ref{721}) which is congruent to the Eisenstein series $\mathcal{E}(\chi_1,\chi_2)$ modulo $A(\chi_1,\chi_2)$. One can associate to $\mathcal{F}_S$ a surjective $\Lambda$-module homomorphism
$$
\Psi:h^{\ord}_{\mfr{m}}/I\surj \Lambda/(A(\chi_1,\chi_2));\; 
T\mapsto C(1,T\cdot F_S),
$$
where $C(1,T\cdot F_S)$ is the first Fourier coefficient of $T\cdot F_S$.

\begin{thm}\label{12}
Let the assumptions be as in Theorem~\ref{11}. Then we have an isomorphism of $\Lambda$-modules
$$
\Psi:h^{\ord}_{\mfr{m}}/I(\chi_1,\chi_2)\isom \Lambda/(A(\chi_1,\chi_2)).
$$
\end{thm}

We will prove a more general version of Theorem~\ref{11} and Theorem~\ref{12} (without assuming $\chi_1\chi_2^{-1}\omega(\mfr{p})\neq 1$ for some prime ideal $\mfr{p}|p$ and $p \nmid N_{F/\Q}(\mfr{nD})\phi(N_{F/\Q}(\mfr{n}))h_F$) in Theorem~\ref{713}, Corollary~\ref{main_thm2}, and Corollary~\ref{68}.

The following corollaries are consequences of Theorem~\ref{11} and Theorem~\ref{12}.

\begin{cor}[Corollary~\ref{68}]
The pairing
$$
(\;,\;):\mathcal{H}^{\ord}_{\mfr{M}} \times M^{\ord}(\mfr{n},\chi_1\chi_2;\Lambda)_{\mfr{M}}\to \Lambda ;\; (T,\mathcal{F})\mapsto C(1,T\cdot \mathcal{F})
$$
is perfect. 
\end{cor}

\begin{cor}[Corollary~\ref{610}]
The $\Lambda$-module $\Ann_{\mathcal{H}^{\ord}_{\mfr{M}}}(\mathcal{I})$ is free of rank $1$.
\end{cor}

When $F=\Q$ and $p\geq 5$, Theorem~\ref{11} and Theorem~\ref{12} were proved by Ohta \cite{Ohta} assuming that the Kubota\textendash Leopoldt $p$-adic $L$-functions do not have a trivial zero and $p\nmid \phi(\mfr{n})$, and by Lafferty \cite{Matt} without assuming those two assumptions in the work of Ohta. When the Kubota\textendash Leopoldt $p$-adic $L$-functions have a trivial zero, in addition to the work of Lafferty, Betina\textendash Dimitrov\textendash Pozzi \cite{BDP} computed the congruence modules attached to the first short exact exact sequence in (\ref{eq:1}) and proved Theorem~\ref{12}  after localizing at the height $1$ prime of $\Lambda$ corresponding to the trivial zero without any assumptions on $p$.

In the work of Ohta, Theorem~\ref{12} was proved by using the Iwasawa main conjecture, while in the work of Lafferty, he did not use the main conjecture. Indeed, Lafferty generalized the work of Emerton \cite{E} in which he proved Theorem~\ref{12} when $F=\Q$ and $\mfr{n}=1$ without using the main conjecture. There are couple of difficulties in the setting of Hilbert modular forms. For example, the residue map does not exist, and the class number $h_F$ of $F$ is not $1$ in general. To overcome those difficulties, we will describe everything adelically, including the space of modular forms and the set of cusps.

\subsection{Ideas of the proof}
There are two main steps in the proof of Theorem~\ref{11}. The first is to show that we have a short exact sequence of flat $\Lambda$-modules (Theorem~\ref{711}), called the \textit{fundamental exact sequence} in \cite{Hsieh} in the setting of unitary automorphic forms,
\begin{equation}\label{eq:14}
0\map S^{\ord}(\mfr{n},\chi_1\chi_2;\Lambda)\map M^{\ord}(\mfr{n},\chi_1\chi_2;\Lambda) \xrightarrow{C_0} \Lambda[C_{\mfr{n}p}^*]^{\ord} \map 0,
\end{equation}
where $C_{\mfr{n}p}^*$ is the set of indicator functions of the set of cusps for the open compact subgroup $K_1(\mfr{n}p)\subset \GL_2(\bbA_{F,f})$. Here $\bbA_{F,f}$ is the finite adele ring of $F$ and the superscript $``\ord"$ in the last term means by applying the ordinary projection.
The fundamental exact sequence is proved by Nakayama's lemma and a control theorem (Corollary~\ref{655}) which states that one has a natural isomorphism
$$
M^{\ord}(\mfr{n},\chi;\Lambda)/(T-\rho(u)u^{k-2}+1)\isom M_k^{\ord}(\mfr{n}p^r, \chi\omega^{2-k}\rho;W).
$$
for all $k\geq 2$. The same statement also holds for $S^{\ord}(\mfr{n},\chi;\Lambda)$. When $F=\Q$, this was proved by Hida \cite[Ch.~7]{Hida2}. When $F\neq \Q$, Wiles \cite[Theorem~3]{W2} proved a control theorem for the space of one-variable ordinary $\Lambda$-adic Hilbert cusp forms, and Hida \cite[Ch.~4.1]{Hida3} proved a theorem for the space of multivariable ordinary $\Lambda$-adic Hilbert cusp forms via a different approach. We follow Hida's argument to prove a theorem for the space of one-variable ordinary $\Lambda$-adic Hilbert modular forms. This seems to be known to experts; however, we have not found any mention of it in the literature. 

Once we have (\ref{eq:14}), to compute the congruence modules attached to the first short exact sequence in (\ref{eq:1}), it remains to compute the constant terms of $\mathcal{E}(\chi_1,\chi_2)$ at all cusps (Proposition~\ref{523}). The congruence modules attached to the second short exact sequence in (\ref{eq:1}) is obtained by Theorem~\ref{12} and Lemma~\ref{731}. When $F=\Q$, Betina\textendash Dimitrov\textendash Pozzi \cite{BDP} used the constant terms of Eisenstein families to study the geometry of eigencurves at weight $1$ Eisenstein points. We hope that our computation in the setting of Hilbert modular forms is useful to generalize the work of Betina-Dimitrov-Pozzi to the setting of Hlbert modular forms. 

When $\chi_2=\mathbbm{1}$, the idea of the proof of Theorem~\ref{12} is to generalize Emerton's argument in the setting of Hilbert modular forms. To show the injectivity of $\Psi$, it suffices to show the existence of a Hecke operator $H\in \mathcal{H}^{\ord}_{\mfr{M}}$ such that for each $\mathcal{F}\in M^{\ord}(\mfr{n},\chi_1\chi_2;\Lambda)_{\mfr{M}}$, we have $C(1,H\cdot \mathcal{F})=C_{\lambda}(0,\mathcal{F})$ for all $\lambda=1,\ldots, h^+_F$, where $C_{\lambda}(0,\mathcal{F})$ are the constant terms of $\mathcal{F}$. We will prove the existence of such Hecke operator in Section~\ref{sec:62}. Here $h_F^+$ is the narrow class number of $F$. Note that for all $\mathcal{F}\in M^{\ord}(\mfr{n},\chi_1\chi_2;\Lambda)_{\mfr{M}}$, we have $C_{1}(0,\mathcal{F})=\ldots=C_{h^+_F}(0,\mathcal{F})$, so the Hecke operator $H$ does not depend on $\lambda$. When $\chi_2\neq \mathbbm{1}$, the proof of Theorem~\ref{12} is much easier than the case of $\chi_2=\mathbbm{1}$. This will be addressed in Corollary~\ref{68}. 

\subsection{Outline}
We now give an outline of the article. In Section~\ref{sec:02}, we review definitions and properties of Hilbert modular forms in both the classical and the adelic settings. Also, we formulate cusps in the adelic language, which plays an important role in stating the main results in Section~\ref{sec:06}.

In Section~\ref{sec:03}, we construct Eisenstein series adelically and compute their constant terms at different cusps. Indeed, one can do this in the classical setting (see \cite[Proposition~2.5.5]{Ohta} and \cite[Proposition~4.7]{BDP} when $F=\Q$ and \cite[Proposition~3.4]{Oz} when $F\neq \Q$). One reason we have to do everything adelically is to show that the map $C_0$ in Theorem~\ref{11} commutes with Hecke operators, which we can only prove in the adelic setting. In addition, it is difficult to write adelic cusps in the classical setting explicitly since to do so, one has to use the strong approximation for $\GL_2$. The way to construct Eisenstein series adelically is to choose certain local induced representations at each place of $F$. We then compute their constant terms by computing local integrals at all places of $F$. In principle, one can obtain the Fourier expansion at all cusps if one can compute all local integrals explicitly. This construction is well-known to experts and has been used to study the arithmetic of Eisenstein series for different algebraic groups such as unitary and symplectic groups (see \cite{Hsieh} for example).   

In Section~\ref{sec:04}, we recall the definition of $\Lambda$-adic modular forms and construct $\Lambda$-adic Eisenstein series as examples. Also, we compute their constant terms at different cusps using results in Section~\ref{sec:03}, which will be used in Section~\ref{sec:06}. 

In Section~\ref{sec:05}, we will review moduli problems of abelian varieties and the definition of geometric modular forms. The goal in this section is to prove a control theorem (Corollary~\ref{655}).
In the last section, we will prove main results and their applications.

\subsection{Notation}
Throughout this paper, we fix a totally real field $F$ with $d=[F:\Q]$, and we let $\O_F$ be the ring of integers of $F$. We write $\wh{\O}_F=\O_F\otimes_{\Z} \wh{\Z}$, where $\wh{\Z}=\prod_{p<\infty} \Zp$.
We denote by $\mathfrak{D}$ the different of $F$ over $\Q$ and $d_F=N(\mfr{D})$ the
discriminant of $F$. Here $N=N_{F/\Q}$ is the norm map from $F$ to $\Q$. We denote by $h_F=|\Cl_F|$ (resp.~$h^+_F=|\Cl_F^+|$) the class number of $F$ (resp.~the narrow class
number of $F$), where $\Cl_F$ (resp.~$\Cl_F^+$) is the ideal class
group of $F$ (resp.~narrow ideal class group).

We fix a set $I=\{\tau_1,\cdots, \tau_d\}$  of distinct real
embeddings of $F$ into $\R$. For any element $f$ in $F$, by $f\gg 0$, we mean that $f$ is totally positive, i.e., $\tau_i(f)>0$ for all $i=1, \ldots, d$. For any subset $A$ of $F$, we denote by $A^+$ the subset of totally positive elements in $A$, i.e., for any $f$ in
$A$, $f \in A^+$ if $f\gg 0$.

For each finite place $v$ of $F$, we denote by $F_{v}$ the completion of $F$ at $v$, $\O_{v}$ its ring of integers, $\mfr{p}_v$ the maximal ideal of $\O_{v}$, and $\varpi_{\mfr{p}_v}$ a fixed uniformizer. We denote by $q_v$ the cardinality of the residue field $\O_{\mfr{p}_v}/\varpi_{\mfr{p}_v}$. Let $\val_v$ be the normalized valuation such that $\val_v(\varpi_{\mfr{p}_v})=1$. Sometimes, we write $\varpi_{\mfr{p}_v}$ as $\varpi_v$ for simplicity. In addition, we will omit $v$ from $\mfr{p}_v$, $\varpi_v$, $q_v$, and $\val_v$ if there is no confusion.   

Finally, we fix, once and for all, embeddings of $\Qbar$ in $\Qpbar$ and in $\C_p$.

\subsection*{Acknowledgments}
The results of this paper are a part of the author's Ph.D.~thesis in University of Arizona.
The author would like to thank his advisor, Romyar Sharifi, for his guidance, support, and suggesting this problem. Also, the author would like to thank Adel Betina, Mladen Dimitrov, Haruzo Hida, Ming-Lun Hsieh, and Hang Xue for helpful suggestions during preparation of this article.

The author was partially supported by National Science Foundation under Grants No.~DMS-1360583 and No.~DMS-1401122, by the Labex CEMPI under Grant No.~ANR-11-LABX-0007-01, and by I-SITE ULNE under Grant No.~ANR-16-IDEX-0004.  

Finally, the author is grateful to the referees for a careful reading and valuable suggestions for improvement.
%%%%%%%%%%%%%%%%%%%%%%%%%%%%%%%%%%%%%%%%%%%%%%%%%%%%%%%%%%%%%%%%%%%%%%%%%%%%%%%%%%%%%%
\section{Hilbert modular forms}\label{sec:02}
In the first two subsections, we review the definitions of classical Hilbert modular forms and adelic Hilbert modular forms of parallel weight following \cite[\S 1 \& \S 2]{Sh}. We refer the reader to \textit{loc.~cit.} for more details, especially, the definition of Hecke operators. %Then we discuss the adelic formulation of cusps and the Hecke action on the set of cusps in the last subsection.

The aim of the last subsection is to compute the ordinary projection of the set of cusps. To do this, we first give an adelic description of the set of cusps and show that this description is equivalent to the classical description (Lemma~\ref{341}). This seems to be known to experts; however, we have not found any mention of it in the literature. We hope that it is useful to the reader to write it down clearly. 
We then define Hecke actions which allow us to compute the ordinary projection of the set of cusps. %(Theorem~\ref{343} and Theorem~\ref{ordcusp}).
In addition, we will show that the map  $C_0$ (defined in (\ref{eq:const_map})) is Hecke-equivariant. Both of these will be used in Section~\ref{sec:06}. 

%%%%%%%%%%%%%%%%%%%%%%%%%%%%%%%%%%%%%%%%%%%%%%%%%%%%%%%%%
\subsection{Classical Hilbert modular forms}\label{sec:21}
Throughout this paper, we denote by $\mathbf{H}=\{z \in
\C \;| \im(z)>0\}$ the complex upper half plane. 
Let $$\GL_2(F)^+=\{\gamma\in \GL_2(F)\mid \det\gamma \gg 0\}$$ be the group of $2\times 2$ matrices with
totally positive determinant. Recall that
$\tau_1,\ldots,\tau_d:F\embed \R$ are fixed distinct real embeddings of $F$. Let $\mathfrak{b}$ be a fractional ideal, and let $\n$ be an integral ideal of $F$. The congruence subgroup $\Gamma_1(\mathfrak{b},\n)$ is defined by
$$
\Gamma_1(\mathfrak{b},\n)=\left\{
\begin{pmatrix}
a & b \\
c & d \\
\end{pmatrix}
\in \GL_2(F)^+ \mid
a\in \O_F, d-1\in \mfr{n}, b\in \mathfrak{b}^{-1}, c \in \mfr{b}\n, ad-bc\in \O_F^{\times}\right\}.
$$

Let $f:\mathbf{H}^d\map \C$ be a function. For $k\in \Z_{\geq 0}$ and
$\gamma\in \GL_2(F)^+$, the slash operator is defined as
$$
f\|_k\gamma(z):=(\det\gamma)^{k/2}j(\gamma,z)^{-k}f(\gamma z),
$$ 
where 
$$
(\det\gamma)^{k/2}=\prod_{i=1}^d\tau_i(\det \gamma)^{k/2} \mbox{ and }
j(\gamma,z)=\prod_{i=1}^d(\tau_i(c)z_i+\tau_i(d)).
$$
Here $\gamma z=(\tau_1(\gamma)z_1,\ldots, \tau_d(\gamma) z_d)$, and
$\tau_i(\gamma)z_i$ is the Mobius action on the upper half plane for
all $i$. A \textit{Hilbert modular form of level $\Gamma_1(\mfr{b},\mfr{n})$ and (parallel) weight $k$} is a holomorphic function $f:\mathbf{H}^d\rightarrow \C$ such that
$$
f\|_{k}\alpha(z)=f(z)
$$
for all $\alpha \in \Gamma_1(\mfr{b},\mfr{n})$. It follows from the definition that every Hilbert modular form
$f$ satisfies $f(z)=f(z+a)$ for $a\in \mfr{b}^{-1}$. Hence one obtains the Fourier expansion
$$
f=\sum_{\mu\in \mfr{b}\mfr{D}^{-1}}c(\mu,f)e^{2\pi i\tr(\mu z)},
$$
where $\tr(\mu z)=\sum_{i=1}^{d} \tau_i(\mu)z_i$. A Hilbert modular form $f$ is called a \textit{cusp form} if the constant term of $f||_{k}\gamma$ vanishes for all
$\gamma\in \GL_2(F)$.

For $R=\C$ or $\Z$, we denote by $M_k(\Gamma_1(\mfr{b},\mfr{n});R)$ (resp.~$S_k(\Gamma_1(\mfr{b},\mfr{n});R)$) the space of Hilbert modular forms (resp.~cusp forms) of level $\Gamma_1(\mfr{b},\mfr{n})$ and weight $k$ whose Fourier coefficients are all in $R$. For any commutative ring $A$, we further define 
$$
M_k(\Gamma_1(\mfr{b},\mfr{n});A):=M_k(\Gamma_1(\mfr{b},\mfr{n});\Z)\otimes_{\Z} A
$$
and define $S_k(\Gamma_1(\mfr{b},\mfr{n});A)$ in the same manner.

%%%%%%%%%%%%%%%%%%%%%%%%%%%%%%%%%%%%%%%%%%%%%%%%%%%%%%%%%%%%%%%%%%%%%%%%%%%%%%%%%%
\subsection{Adelic Hilbert modular forms}\label{sec:22}
Let $\A_F$ be the adele ring of $F$, and let $\A_{F,f}$ be the
finite adele ring. For any finite place $v$ of $F$ and any integral ideal $\n$
of $\O_F$, define
$$
K_{0,v} (\n)=\left\{ \left(\begin{array}{cc}
      a & b \\ c & d \\
      \end{array}\right)\in \GL_2(\O_{F_v})\mid c\in \n_v,\right\},
$$
$$
K_{1,v} (\n)=\left\{ \left(\begin{array}{cc}
      a & b \\ c & d \\
      \end{array}\right)\in K_{0,v}(n)\mid d-1 \in \n_v\right\}, \mbox{ and }      
K_v^1 (\n)=\left\{ \left(\begin{array}{cc}
      a & b \\ c & d \\
      \end{array}\right)\in K_{0,v}(n)\mid a-1 \in \n_v\right\}.
$$
Set $K_1(\n)=\prod_{v<\infty}K_{1,v}(\n)$. The group $K_v^1$ will only be used in the proof of Theorem~\ref{ordcusp}. Let $\GL_2^+(\R)=\{g\in \GL_2(\R) \mid \det g >0\}$, and let $K_{\infty}^+=(\R_+^{\times}\cdot\SO_2(\R))^d$. Note that $K_{\infty}^+$
is the stabilizer of $(i,\ldots,i)\in \mathbf{H}^d$ in $(\GL_2^+(\R))^d$. 

\begin{defn}\label{321}
For $k\in \Z_{\geq 0}$, an \textit{adelic Hilbert modular form of (parallel) weight $k$ and level $K_1(\n)$}  is
a function $f:\GL_2(\mathbb{A}_F)\map \C$ such that the following
properties hold:
\begin{enumerate}
\item $f(\gamma g \kappa)=f(g)$ for all $\gamma\in \GL_2(F)$, $g\in \GL_2(\A_F)$, and $\kappa\in K_1(\n)$.

\item $f(ga)= (\det a)^{\frac{k}{2}}j(a,i)^{-k}f(g)$ for all $a\in (\GL_2^+(\R))^d$ and $g\in \GL_2(\A_F)$.

\item For $x\in \GL_2(\A_{F,f})$, we define a function
$f_x:\mathbf{H}^d\map \C$ by
$$
f_x(z)= (\det g_z)^{-\tfrac{k}{2}}j(g_z,i)^{k}f(g_z x)
$$
for $g_z=(g_{z_j})\in (\GL_2^+(\R))^d$ and $z=(z_j)\in \mathbf{H}^d$ such
that $g_{z_j}\cdot i=z_j$ for $j=1,\ldots,d$. Then $f_x$ is a holomorphic
function for all $x$.
\end{enumerate}
An adelic Hilbert modular form $f$ is called a \textit{cusp form} if we have
$$
\int_{F\backslash\A_F} f\left(\left(\begin{array}{cc}
                    1 & x \\0 & 1 \\
                  \end{array}\right)g\right)dx=0
$$
for all $g\in \GL_2(\A_F)$.
\end{defn}

We denote by $M_{k}(K_1(\n);\C)$ the space of adelic Hilbert modular forms of weight $k$ and level $K_1(\n)$, and denote by $S_{k}(K_1(\n);\C)$ the subspace of cusp forms. The following proposition states the relationship between classical Hilbert modular forms and adelic Hilbert modular forms. See \cite[\S 2]{Sh} for a proof.

\begin{prop}\label{324}
There exist canonical isomorphisms of complex vector spaces
$$
M_{k}(K_1(\n);\C)\simeq \bigoplus_{\lambda=1}^{h^+_F} M_{k}(\Gamma_1(t_{\lambda}\mfr{D},\n);\C)
\mbox{ and }
S_{k}(K_1(\n);\C)\simeq \bigoplus_{i=1}^{h^+_F} S_{k}(\Gamma_1(t_{\lambda}\mfr{D},\n);\C).
$$ 
\end{prop}

Let $\{t_1,\ldots,t_{h_F^+}\}$ be a set of representatives of $\Cl_F^+$. It follows from Proposition~\ref{324} that each $f\in M_k(K_1(\mfr{n});\C)$ can be written as a vector $(f_1,\ldots,f_{h^+_F})$. We saw in the previous subsection that each $f_{\lambda}$ admits a Fourier expansion, namely,
$$
f_{\lambda}(z)=\sum_{\mu\in t_{\lambda}^+\cup\{0\}} c(\mu,f_{\lambda})e^{2\pi\tr(uz)}.
$$
We call $c(\mu,f_{\lambda})$ the \textit{unnormalized
Fourier coefficients} of $f$, and we define the normalized Fourier
coefficients as follows. Each integral ideal $\mfr{m}$ of $F$
must be in one of the narrow ideal classes, say that of $(t_{\lambda}\mfr{D})^{-1}$ for
some $\lambda\in{1,\ldots, h^+_F}$. We choose a totally positive element $u\in
t_{\lambda}\mfr{D}$ such that $\mfr{m}= (u)(t_{\lambda}\mfr{D})^{-1}$. Then the \textit{normalized Fourier coefficient}
of $f$ associated to $\mfr{m}$ is defined by
$$
C(\mfr{m},f):=N(t_{\lambda}\mfr{D})^{-k/2} c(u,f_{\lambda}),
$$
and the \textit{normalized constant terms} are defined by
$$
C_{\lambda}(0,f):=N(t_{\lambda}\mfr{D})^{-k/2} c(0,f_{\lambda})
$$
for $\lambda=1,\ldots, h^+_F$. The normalized Fourier coefficients $C(\mfr{m},f)$ and $C_{\lambda}(0,f)$  are independent of the choice of $u$ and of the choices of the $t_{\lambda}$ (\textit{loc.~cit.}). From now on, we fix a set of representatives $\{t_1,\ldots,t_{h_F^+}\}$ such that
\begin{equation}\label{eq:class_repr}
t_{\lambda}\mfr{D} \mbox{ is an integral ideal and } (t_{\lambda}\mfr{D},\mfr{n})=1.
\end{equation} 
for all $\lambda=1,\ldots,h_F^+$. Such a set of representatives exists by \cite[Lemma~2.9]{Oz}.

For any commutative ring $A$, we denote by 
$M_k(K_1(\mfr{n});A)$ the space of adelic modular forms whose normalized Fourier coefficients are all in $A$ and denote by $S_k(K_1(\mfr{n});A)$ in the same manner.

In p.~648 of \cite{Sh}, Shimura defined Hecke operators $S(\mfr{m})$ and $T(\mfr{m})$ for $\mfr{m}$ relatively prime to $\mfr{n}$ and $U(\mfr{p})$ for prime ideals $\mfr{p}$ dividing $\mfr{n}$. In addition, these Hecke operators commute. For a narrow ray class character $\chi$ with modulus $\mfr{n}$, the space $M_k(\mfr{n},\chi;A)$ is a subspace of $M_k(K_1(\mfr{n});A)$ consisting of modular forms $f$ satisfying
$$
S(\mfr{m})\cdot f=\chi(\mfr{m})\cdot f
$$
for all integral ideals $\mfr{m}$ with $(\mfr{m},\mfr{n})=1$. The space $S_k(\mfr{n},\chi;A)$ is defined in the same manner.
 
\begin{defn}\label{332}\ 
\begin{enumerate}
\item A modular form $f$ is said to be an \textit{eigenform} if $f$ is an eigenvector for all Hecke operators.

\item An eigenform $f$ is \textit{normalized} if $C(1,f)=1$. Then one has $T(\mfr{q})\cdot f=C(\mfr{q},f)\cdot f$ for all prime ideals $\mfr{q}$ prime to $\mfr{n}$, and  $U(\mfr{p})\cdot f=C(\mfr{p},f)\cdot f$ for all $\mfr{p}|\mfr{n}$ (see the remark in \cite[p.~418]{W1}). 

\item Let $\mfr{p}$ be a prime ideal of $F$. A normalized eigenform $f$ is called \textit{$\mfr{p}$-ordinary} if its $\mfr{p}$th Fourier coefficient $C(\mfr{p},f)$ is an unit in $\O_{\mfr{p}}$. A normalized eigenform $f$ is called \textit{$p$-ordinary} if it is $\mfr{p}$-ordinary for all $\mfr{p}|p$.
\end{enumerate}
\end{defn}

Let $e=\lim_{n\map \infty}\prod_{\mfr{p}|p} U(\mfr{p})^{n!}$ be Hida's idempotent element. It was shown by Wiles \cite[p.~537]{W2} that for all $r\in \Z_{>0}$, $e$ acts on $M_k(K_1(\mfr{n}p^r);\Zp)$ under the $p$-adic topology, and $S_k(K_1(\mfr{n}p^r);\Zp)$ is preserved under the action of $e$. We denote by $M^{\ord}_k(K_1(\mfr{n}p^r);\Zp)$ the subspace $e\cdot M_k(K_1(\mfr{n}p^r);\Zp)$ and similarly for $S^{\ord}_k(K_1(\mfr{n}p^r);\Zp)$. In addition, we denote by $\mathcal{H}_k(K_1(\mfr{n}p^r);\Zp)$ (resp.~ $h_k(K_1(\mfr{n}p^r);\Zp)$) the Hecke algebra (resp.~cuspidal Hecke algebra) generated over $\Zp$ by the Hecke operators $U(\mfr{p})$, $T(\mfr{q})$ and $S(\mfr{q})$ for all prime ideals $\mfr{q}$ not dividing $\mfr{n}p$ and for all prime ideals $\mfr{p}|\mfr{n}p$. We write $H^{\ord}_k(K_1(\mfr{n}p^r);\Zp) =e\cdot H_k(K_1(\mfr{n}p^r);\Zp)$ and write $h_k^{\ord}(K_1(\mfr{n}p^r);\Zp)$ in the same manner.  

We now review properties of Eisenstein series attached to pairs of narrow ray class characters
of $F$ following \cite[Proposition 3.4]{Sh}. 

\begin{prop}[Shimura]\label{333}
Let $\chi_1$ and $\chi_2$ be primitive narrow ray class characters of conductors $\mfr{n}_1$ and $\mfr{n_2}$, respectively, with associated signs
$e_{1,\infty},e_{2,\infty}\in (\Z/2\Z)^d$ satisfying
\begin{equation}\label{eq:eisen_cond}
e_{1,\infty}+e_{2,\infty} \equiv (k,\ldots,k)\; (\bmod\; 2\Z^d)
\end{equation}
for some integer $k\geq 2$.
We view the characters $\chi_1\chi_2$ and $\chi_1\chi_2^{-1}$ as characters with
modulus $\mfr{n}=\mfr{n}_1\mfr{n}_2$. Assume that $\chi_1$ is nontrivial. Then there exists an eigenform $E_k(\chi_1,\chi_2)=(E_{\lambda})_{\lambda=1}^{h_F^+}$,
where $E_{\lambda}$ is in $M_k(\Gamma_1(t_{\lambda}\mfr{D},\mfr{n}); \Zp[\chi_1,\chi_2])$ for $\lambda=1,\ldots,h_F^+$, such that
\begin{equation}
C(\mfr{m},E_k(\chi_1,\chi_2))=\sum_{\mfr{a}|\mfr{m}}\chi_1(\mfr{a})\chi_2(\tfrac{\mfr{m}}{\mfr{a}})
N(\mfr{a})^{k-1}
\end{equation}
for all nonzero integral ideals $\mfr{m}$ of $\O_F$ and
\begin{equation}
C_{\lambda}(0,E_k(\chi_1,\chi_2))=
\begin{cases}
2^{-d}\chi_2^{-1}(t_{\lambda}\mfr{D})L(1-k,\chi_1\chi_2^{-1})& \mbox{if } \mfr{n}_2=1,\\
0&\mbox{otherwise}.
\end{cases}
\end{equation}
In particular, $E_k(\chi_1,\chi_2)$ is a $p$-ordinary modular form if $(\mfr{n}_2,p)=1$.
Here $L(s,\chi)$ is the $L$-function associated to the character $\chi$ of conductor $\mfr{n}$, which is defined as the meromorphic continuation of the $L$-series
\begin{equation}\label{eq:L_function}
L(s,\chi):=\sum_{(\mfr{a},\mfr{n})=1}\chi(\mfr{a}) N(\mfr{a})^{-s}
=\prod_{\mfr{p}\nmid \mfr{n}}(1-\chi(\mfr{p})N(\mfr{p})^{-s})^{-1},
\end{equation}
which converges absolutely for $\re(s)>1$. 
\end{prop}

%%%%%%%%%%%%%%%%%%%%%%%%%%%%%%%%%%%%%%%%%%%%%%%%%%%%%%%%%%%%%%%%%%%%%%%%%%%%%%%%%%%%%
\subsection{Cusps of Hilbert modular varieties}\label{sec:24}
We respectively denote by $B$, $T$, and $N$ the algebraic subgroup of upper-triangular matrices, the subgroup of diagonal matrices, and the unipotent upper-triangular subgroup of $\GL_2$. We first review two decompositions of $\GL_2$ for later use. Let $v$ be a finite place of $F$. The first decomposition is the Iwasawa decomposition of $\GL_2(F_{v})$, which is given by 
\begin{equation}\label{eq:ID1}
\GL_2(F_{v})=B(F_{v})\GL_2(\O_{v}).
\end{equation}
Moreover, for each $\left(\begin{smallmatrix} a & b\\ c & d\end{smallmatrix}\right)\in \GL_2(F_{v})$, we have
\begin{equation}\label{eq:ID2}
\left(\begin{array}{cc} a & b\\ c & d\end{array}\right)=
\begin{cases}
\left(\begin{array}{cc} a-\frac{bc}{d} & b\\ 0 & d\end{array}\right)
\left(\begin{array}{cc} 1 & 0\\ \frac{c}{d} & 1\end{array}\right) & \mbox{if } \val_{v}(c)\geq \val_{v}(d))\\
\left(\begin{array}{cc} -\frac{ad}{c}+b & a\\ 0 & c\end{array}\right)
\left(\begin{array}{cc} 0 & 1\\ 1 & \frac{d}{c}\end{array}\right) & \mbox{if } \val_{v}(c)\leq \val_{v}(d)).
\end{cases}
\end{equation}
Let $\mfr{n}$ be an integral ideal of $F$. For simplicity, we set $N_v:=\val_v(\mfr{n})$. The second decomposition is given by
\begin{equation}\label{decom}
\GL_2(\O_v)=\coprod_{i=0}^{N_v} B(\O_v)\gamma_iK_{1,v}(\mfr{n}),
\end{equation}
where $\gamma_i= \left(\begin{smallmatrix} 1 & 0\\ \varpi^i_v & 1 \end{smallmatrix}\right)$ for $0\leq i< N_v$ and $\gamma_{N_v}=I_2$, the identity matrix. Moreover, for $g=\left(\begin{smallmatrix} a & b\\ c & d\end{smallmatrix}\right)\in \GL_2(\O_v)$, if $c$ is a unit, then we have
\begin{equation}\label{eq:22}
\left(\begin{array}{cc} a & b\\ c & d\end{array}\right)=
\left(\begin{array}{cc} \frac{ad-bc}{c} & a+\frac{(bc-ad)(1+\varpi^{N_v})}{c}\\ 0 & c\end{array}\right)
\left(\begin{array}{cc} 1 & 0\\ 1 & 1\end{array}\right)
\left(\begin{array}{cc} 1+\varpi^{N_v} & (1+\varpi^{N_v})c^{-1}d-1\\ -\varpi^{N_v} & 1-\varpi^{N_v} c^{-1}d \end{array}\right),
\end{equation}
and if $\val(c)=j>0$, we have
\begin{equation}\label{eq:23}
\left(\begin{array}{cc} a & b\\ c & d\end{array}\right)=
\left(\begin{array}{cc} (ad-bc)c^{-1}\varpi^j & b\\ 0 & d\end{array}\right)
\left(\begin{array}{cc} 1 & 0\\ \varpi^j & 1\end{array}\right)
\left(\begin{array}{cc} \varpi^{-j}cd^{-1} & 0\\ 0 & 1\end{array}\right).
\end{equation}

We denote by $C_{\mfr{n}}$ the set of cusps for $K_1(\mfr{n})$, defined as the set of double cosets 
$$
C_{\mfr{n}}:=B(F)^+ N(\bbA_{F,f})\backslash T(\bbA_{F,f})\GL_2(\wh{\O}_F)/K_1(\mfr{n}),
$$
where $B(F)^+$ is the subgroup of $B(F)$ consisting of all matrices with totally positive determinant.
We say that two elements $c$ and $c'$ in $T(\bbA_{F,f})\GL_2(\wh{\O}_F)$ are the same in $C_{\mfr{n}}$, denoted by $c\sim c'$, if $c=\gamma c' \kappa$ for some $\gamma\in B(F)^+ N(\bbA_{F,f})$ and $\kappa\in K_1(\mfr{n})$. The classical description of the set of cusps for $K_1(\mfr{n})$ is given by
$$
C'_{\mfr{n}}:= \coprod_{\lambda=1}^{h^+_F}\Gamma_1(t_{\lambda}\mfr{D},\mfr{n})\backslash\mathbb{P}^1(F).
$$ 
Recall that $\{t_1,\ldots,t_{h_F^+}\}$ was fixed to satisfy (\ref{eq:class_repr}) in the previous subsection. We will abuse the notation to denote by $t_{\lambda}\in \mathbb{A}_{F,f}$ generating the ideal $t_{\lambda}$ whose value is 1 at $v$ if $\mfr{p}_v$ does not divide the ideal $t_{\lambda}$. For $\lambda=1,\ldots,h_F^+$, we set
\begin{equation}\label{eq:x_lambda}
x_{\lambda}=\left(\begin{array}{cc}
    t_{\lambda}\delta & 0 \\0 & 1
\end{array}\right)\in \GL_2(\mathbb{A}_{F,f}),
\end{equation}
where $\delta\in \bbA_{F,f}$ such that $\delta\O_F=\mfr{D}$ and $\delta_v=1$ for $v\nmid \mfr{D}$. 

The following lemma shows that the above two definitions of cusps for $K_1(\mfr{n})$ are equivalent, and hence, $C_{\mfr{n}}$ is a finite set as $C_{\mfr{n}}'$ is a finite set.

\begin{lemma}\label{341}
Let the notation be as above. 
\begin{enumerate}
\item we have
\begin{equation}\label{eq:adel_cusp}
\GL_2(F)^+N(\bbA_{F,f})\backslash \GL_2(\bbA_{F,f})\times \mathbb{P}^1(F)/K_1(\mfr{n})=C_{\mfr{n}}.
\end{equation}
Here $N(\bbA_{F,f})$ (resp.~$K_1(\mfr{n})$) and $\GL_2(F)^+$ respectively act on $\GL_2(\bbA_{F,f})\times \bbP^1(F)$ by left multiplying (resp.~right multiplying) on $\GL_2(\bbA_{F,f})$ and left multiplying on $\GL_2(\bbA_{F,f})\times \bbP^1(F)$ diagonally. 

\item The map
\begin{equation}\label{eq:311}
C'_{\mfr{n}}\map \GL_2(F)^+N(\bbA_{F,f})\backslash \GL_2(\bbA_{F,f})\times \mathbb{P}^1(F)/K_1(\mfr{n});
\end{equation}
$$
\Gamma_1(t_{\lambda}\mfr{D},\mfr{n})  \left(\begin{array}{cc} a\\ c\end{array}\right)\mapsto \GL_2(F)^+N(\bbA_{F,f})\left(x_{\lambda}^{-1}, \left(\begin{array}{cc} a \\ c\end{array}\right)\right) K_1(\mfr{n})
$$
is bijective. 
\end{enumerate}
\end{lemma}

\begin{proof} 
Since $\bbP^1(F)=\GL_2(F)^+/B(F)^+$, we know that $\GL_2(F)^+$ acts on $\bbP^1(F)$ transitively, and the stabilizer of $\left(\begin{smallmatrix} 1 \\ 0\end{smallmatrix}\right)$ is $B(F)^+$. Therefore, we have
$$
\GL_2(F)^+N(\bbA_{F,f})\backslash \GL_2(\mathbb{A}_{F,f})\times \mathbb{P}^1(F)/K_1(\mfr{n})
= B(F)^+ N(\bbA_{F,f})\backslash \GL_2(\mathbb{A}_{F,f}) \times \left\{\left(\begin{array}{c} 1 \\ 0\end{array}\right)\right\}/K_1(\mfr{n}).
$$
By the Iwasawa decomposition (\ref{eq:ID1}), one can decompose $\GL_2(\bbA_{F,f})$ as
$$
\GL_2(\bbA_{F,f})=B(\bbA_{F,f})\GL_2(\wh{\O}_F),
$$
and hence, we obtain the following equalities:
\[
\begin{split} 
\GL_2(F)^+ N(\bbA_{F,f})\backslash \GL_2(\mathbb{A}_{F,f})\times \mathbb{P}^1(F)/K_1(\mfr{n})
=& B(F)^+ N(\bbA_{F,f})\backslash \GL_2(\mathbb{A}_{F,f}) \times \left\{\left(\begin{array}{c} 1 \\ 0\end{array}\right)\right\}/K_1(\mfr{n})\\
=& B(F)^+ N(\bbA_{F,f})\backslash \GL_2(\mathbb{A}_{F,f})/K_1(\mfr{n})\\
=& B(F)^+ N(\bbA_{F,f})\backslash B(\mathbb{A}_{F,f}) \GL_2(\wh{\O}_F)/K_1(\mfr{n})\\
=& B(F)^+ N(\bbA_{F,f})\backslash T(\bbA_{F,f}) \GL_2(\wh{\O}_F)/K_1(\mfr{n}).
%=& B(F)^+\backslash \coprod_{\lambda=1}^{h^+_F} B(F)x_{\lambda}B(\wh{\O}_F) \GL_2(\wh{\O}_F)/K_1(\mfr{n})\%\
%=& \coprod_{\lambda=1}^{h^+_F} B(F)^+ \backslash x_{\lambda}\GL_2(\wh{\O}_F)/K_1(\mfr{n})\\
\end{split}
\]
Here the last equality follows from the fact that $B(\bbA_{F,f})=N(\bbA_{F,f})T(\bbA_{F,f})$. This proves (\ref{eq:adel_cusp}).

Next, we show that (\ref{eq:311}) is bijective. In the remaining of the proof, we will write $\Gamma_{\lambda}=\Gamma_1(t_{\lambda}\mfr{D},\mfr{n})$ for simplicity.  
It follows from strong approximation for $\GL_2$ that we have
$$
\GL_2(\A_{F,f})=\coprod_{\lambda=1}^{h^+_F} \GL_2(F)^+ x_{\lambda}^{-1}K_1(\mfr{n}).
$$
From this, one sees that the map (\ref{eq:311}) is surjective. We again write $\bbP^1(F)=\GL_2(F)^+/B(F)^+$. In what follows, we will view elements of $\GL_2(F)^+$ as elements of $\mathbb{P}^1(F)$.  To see the injectivity, assume that we have two elements $g,g'\in \GL_2(F)^+$ mapping to the same coset in $\GL_2(F)^+ N(\bbA_{F,f})\backslash \GL_2(\bbA_{F,f})\times \mathbb{P}^1(F)/K_1(\mfr{n})$. We claim that $g'=\gamma g \beta$ for some $\gamma\in \Gamma_{\lambda}$ and $\beta\in B(F)^+$, which implies that $g$ and $g'$ are the same in $C'_{\mfr{n}}$ and hence, the map (\ref{eq:311}) is injective.  It follows from the assumption on $g$ and $g'$ that there exist $\gamma_1\in \GL_2(F)^+$, $n\in N(\bbA_{F,f})$, $\beta\in B(F)^+$, and $\kappa_1\in K_1(\mfr{n})$ such that  
$$
(x_{\lambda}^{-1},g')=\gamma_1 n (x_{\lambda}^{-1}, g \beta) \kappa_1.
$$  
By strong approximation, we have $N(\mathbb{A}_{F,f})=N(F)(N(\mathbb{A}_{F,f})\cap K_1(\mfr{n}))$.
Thus, we can write $n=\gamma_2 \kappa_2$ for some $\gamma_2\in N(F)$ and $\kappa_2\in N(\mathbb{A}_{F,f})\cap K_1(\mfr{n})$. One can deduce from this that
\begin{equation}\label{eq:312}
(x_{\lambda}^{-1},g')=\gamma_1 \gamma_2 (\kappa_2x_{\lambda}^{-1}\kappa_1, g \beta)
=\gamma (x_{\lambda}^{-1}\kappa, g \beta)
\end{equation}
for some $\gamma\in \GL_2(F)$ and $\kappa\in K_1(\mfr{n})$. The last equality is obtained by (\ref{eq:class_repr}) that $t_{\lambda}\delta\in \wh{\O}_F$.
From (\ref{eq:312}), we see that the finite part $\gamma_f$ of $\gamma$ and the infinite part $\gamma_{\infty}$ of $\gamma$ satisfy
$$
\gamma_f=x_{\lambda}^{-1}\kappa^{-1} x_{\lambda}\mbox{ and } \gamma_{\infty}=g'\beta^{-1}g^{-1}.
$$
The former implies that $\gamma$ satisfies the congruence properties of $\Gamma_{\lambda}$, and the latter implies that $\det \gamma$ is totally positive and $g'=\gamma g \beta$. Hence, $\gamma\in \Gamma_{\lambda}$ which proves the claim. 
\end{proof} 

For any subset $C$ of $C_{\mfr{n}}$, we denote by $C^*=\{I_{[c]}\mid c\in C\}$ the set of indicator functions on $C$. The indicator function $I_{[c]}:C_{\mfr{n}}\map \{0,1\}$ is defined by
$$
I_{[c]}(c')=
\begin{cases}
1 & c\sim c'\\
0 & \mbox{otherwise}.
\end{cases}
$$
The Hecke actions on $C_{\mfr{n}}$ and on $C_{\mfr{n}}^*$ are defined as follows. Recall that for each prime ideal $\mfr{p}_v=\mfr{p}$, one has coset decompositions
$$
K_1(\mfr{n})\left(\begin{array}{cc} \varpi_v & 0\\ 0 & 1\end{array}\right)K_1(\mfr{n})=\coprod_{i} \gamma_i K_1(\mfr{n})=\coprod_j K_1(\mfr{n}) \beta_j
$$
for some $\gamma_i$ and $\beta_j$ in $\GL_2(\bbA_{F,f})$. Here, we abuse the notation to denote by $\varpi_v$ an element of finite adele that is $\varpi_v$ at $v$ and $1$ at all other finite places of $F$. For a prime ideal $\mfr{q}$ of $F$ not dividing $\mfr{n}$, the operator $T(\mfr{q})$ acts on $C_{\mfr{n}}$ by 
$$
T(\mfr{q})\cdot c=\sum_i c\gamma_i,
$$
and acts on $C_{\mfr{n}}^*$ by 
$$
T(\mfr{q})\cdot I_{[c]}=\sum_j I_{[c\beta_j^{-1}]}.
$$  
We also define the operator $T^*(\mfr{q})$ that acts on $C_{\mfr{n}}$ by
$$
T^*(\mfr{q})\cdot c=\sum_{j} c\beta_j^{-1}.
$$
For a prime ideal $\mfr{p}|p$, the Hecke operators $U(\mfr{p})$ and $U^*(\mfr{p})$ are defined in the same manner. For an integral ideal $\mfr{m}$ prime to $\mfr{n}$, one can also define the Hecke operators $T(\mfr{m})$ and $T^*(\mfr{m})$ in the same way, but we will not use them in this paper. It follows from Lemma~\ref{341}(1) that the above definition of Hecke operators is well-defined. It is easy to see that the homomorphism 
\begin{equation}\label{eq:11}
\Zp[C_{\mfr{n}}^*]\map \Zp[C_{\mfr{n}}];\; \sum_{c\in C_{\mfr{n}}} a_{c}\cdot I_{[c]}\mapsto \sum_{c\in C_{\mfr{n}}} a_{c}\cdot c
\end{equation}
is an isomorphism of abelian groups which commutes with the $T(\mfr{q})$-action (resp.~$U(\mfr{p})$-action) on the left and the $T^*(\mfr{q})$-action (resp.~$U^*(\mfr{p})$-action) on the right for all prime ideals $\mfr{q}\nmid \mfr{n}$ (resp.~$\mfr{p}|\mfr{n}$). The Hecke action of $T^*(\mfr{q})$ and $U^*(\mfr{p})$ on $C_{\mfr{n}}$ will only be used in the proof of Corollary~\ref{ordcusp}.

Define a $\Zp$-homomorphism $C_0:M_k(K_1(\mfr{n});\Zp)\map \Zp[C^*_{\mfr{n}}]$ as
\begin{equation}\label{eq:const_map}
C_0(f)=\sum_{[g]\in C_{\mfr{n}}} \left(\int_{\bbA_F/F} f(n(x)g) dx\right)\cdot I_{[g]},
\end{equation}
where $n(x)=\left(\begin{smallmatrix} 1 & x\\ 0 & 1\end{smallmatrix}\right)$ for all $x\in \bbA_F$. This homomorphism is well-defined as $f$ is left $\GL_2(F)$-invariant. Moreover, it does not depend on the choice of the representatives of $C_{\mfr{n}}$. To see this, we observe that for any $\beta=\left(\begin{smallmatrix} a & b \\ 0 & d \end{smallmatrix} \right)\in B(F)$, $n(t)\in N(\bbA_{F,f})$, and $\kappa\in K_1(\mfr{n})$,  by substitution and the fact that $f$ is right $K_1(\mfr{n})$-invariant, we have
$$
\int_{\bbA_F/F} f(n(x)\beta n g\kappa) dx=|d/a|_{\bbA} \int_{\bbA_F/F} f(n(x) g) dx= \int_{\bbA_F/F} f(n(x) g) dx.
$$
Here the last equality is obtained by the assumption that $a,d\in F$ and that  $|d/a|_{\bbA}=\prod_{v}|d/a|_v=1$, where $v$ runs through all places of $F$.

\begin{prop}\label{HeckeEquiv}
Let the notation be as above. Then the map $C_0$ commutes with the Hecke actions.
\end{prop}

\begin{proof}
By the same argument as in \cite[Lemma 5.5.1]{DS}, there exists a set of elements $\{\gamma_i\}$ in $\GL_2(\bbA_{F,f})$ such that  
$$
K_1(\mfr{n})\left(\begin{smallmatrix} \varpi_v & 0\\ 0 & 1\end{smallmatrix}\right)K_1(\mfr{n})=\coprod_i \gamma_{i} K_1(\mfr{n})=\coprod_i K_1(\mfr{n}) \gamma_{i}
$$ 
for all $v$. As the computation for $T(\mfr{q})$ and $U(\mfr{p})$ are the same, we will only prove the assertion for $T(\mfr{q})$. One has
\begin{align*}
C_0(T(\mfr{q})\cdot f)
=& \sum_{[g]\in C_{\mfr{n}}}\sum_i \left(\int_{\bbA_F/F} f(n(x)g\gamma_i) dx\right) \cdot I_{[g]}
= \sum_i\sum_{[g]\in C_{\mfr{n}}} \left(\int_{\bbA_F/F} f(n(x)g\gamma_{i}) dx\right) \cdot I_{[g]}\\
=& \sum_i\sum_{[G]\in C_{\mfr{n}}} \left(\int_{\bbA_F/F} f(n(x)G) dx\right) \cdot I_{[G\gamma_i^{-1}]}
= T(\mfr{q})C_0(f).
\end{align*}
Thus, the assertion follows.
\end{proof}

To close this section, we will give a description of the set of ordinary cusps for $K_1(\mfr{n}p^r)$ with $(\mfr{n},p)=1$ for $r\geq 1$. Let 
$$
D_r=\left\{\left(\begin{array}{cc} b_1 & 0\\ 0 & b_2\end{array}\right)\left(\begin{array}{cc} a & b\\ c & d\end{array}\right)\in T(\bbA_{F,f})\GL_2(\wh{\O}_F)\mid ad-bc=1\in \wh{\O}_F \mbox{ and } \val_{\mfr{p}}(c)>0 \mbox{ for some }  \mfr{p}|p\right\}
$$
and 
$$
\mathbf{D}_{r}=\left\{\left(\begin{array}{cc} b_1 & 0\\ 0 & b_2\end{array}\right)\left(\begin{array}{cc} a & b\\ c & d\end{array}\right)\in T(\bbA_{F,f})\GL_2(\wh{\O}_F)\mid ad-bc=1\in \wh{\O}_F \mbox{ and } \val_{\mfr{p}}(c)< r \mbox{ for some } \mfr{p}|p\right\}
$$
be subsets of $T(\bbA_{F,f})\GL_2(\wh{\O}_F)$, and let 
$$
\ol{D}_r=B(F)^+ N(\bbA_{F,f})\backslash D_r/K_1(\mfr{n}p^r)
$$
and
$$
\ol{\mathbf{D}}_r=B(F)^+ N(\bbA_{F,f})\backslash \mathbf{D}_r/K_1(\mfr{n}p^r)
$$
be subsets of the set of double cosets $C_{\mfr{n}p^r}$. Note that for any $\beta g\in C_{\mfr{a}}$ with $\beta \in T(\bbA_{F,f})$ and $g\in \GL_2(\wh{\O}_F)$, by right multiplying an element in $K_1(\mfr{a})$ if necessary, we may assume $\det g=1$. This is the reason to put the condition $ad-bc=1$ in the definition of $D_r$ and $\mathbf{D}_r$.

In what follows, we will treat $e\cdot \Zp[C_{\mfr{n}p^r}]$ as the quotient of $\Zp[C_{\mfr{n}p^r}]$ by $(1-e)\cdot \Zp[C_{\mfr{n}p^r}]$ unless otherwise noted.

\begin{thm}\label{343}
Let the notation be as above. Assume that $p$ is unramified in $F$. Then 
\begin{enumerate}
\item $U(\mfr{p})^n\cdot \Zp[\ol{D}_r] \subset \Zp[\ol{D}_r]$ for all $\mfr{p}|p$ and for all $n\in \N$,

\item $e\cdot \Zp[\ol{D}_r]=0$,

\item $e\cdot \Zp[C_{\mfr{n}p^r}]\isom\Zp[C_{\mfr{n}p^r}]/\Zp[\ol{D}_r]$,
\end{enumerate}
\end{thm}

\begin{proof}
Given any $\delta=\beta g=\left(\begin{smallmatrix} b_1 & 0\\ 0 & b_2\end{smallmatrix}\right)\left(\begin{smallmatrix} a & b\\ c & d\end{smallmatrix}\right)\in D_r$, we are going to compute $U(\mfr{p})^n\cdot \delta$ in $\Zp[C_{\mfr{n}p^r}]$ for $n\in \N$ and for some $\mfr{p}|p$. We first recall that for each $\mfr{p}|p$, we have
$$
K_1(\mfr{n}p^r)\left(\begin{array}{cc} \varpi_v & 0\\ 0 & 1\end{array}\right)K_1(\mfr{n}p^r)=\coprod_{u\in \O_v/\varpi_v} 
\left(\begin{array}{cc} \varpi_v & u\\ 0 & 1\end{array}\right)
K_1(\mfr{n}p^r).
$$
Here, for $u\in \O_v/\varpi_v$, we will arbitrary choose its representative in $\O_{v}$, also denoted by $u$. The same notation will be used in this proof and the proof of Theorem~\ref{ordcusp}. This will not cause any confusion as all matrices in the proof are over $\O_{v}$.  

We fix a prime ideal $\mfr{p}=\mfr{p}_v$ of $F$ such that $\mfr{p}|p$ and $\val_{v}(c)>0$.
The following computation is at the place $v$ which is sufficient since the action of the Hecke operator $U(\mfr{p})$ is trivial at all places other than $v$. For simplicity, we write $\varpi_v$ as $\varpi$ and write $\val_{v}$ as $\val$. Then we have
\[
\begin{split}
U(\mfr{p})^n\cdot \delta
&=\sum_{u\in \O_v/\varpi^n\O_v}
\left(\begin{array}{cc} b_1 & 0\\ 0 & b_2\end{array}\right)\left(\begin{array}{cc} a & b\\ c & d\end{array}\right)\left(\begin{array}{cc} \varpi^n & u\\ 0 & 1\end{array}\right)\\
&= \sum_{u\in \O_v/\varpi^n\O_v}
\left(\begin{array}{cc} b_1\varpi^n & 0\\ 0 & b_2\end{array}\right)\left(\begin{array}{cc} a & \varpi^{-n}(au+b)\\ c\varpi^n & cu+d\end{array}\right).
\end{split}
\]
Since $\val(c)>0$, we know that $d$ and $cu+d$ are in $\O_v^{\times}$. Note that $\det \left(\begin{smallmatrix} a & \varpi^{-n}(au+b)\\ c\varpi^n & cu+d\end{smallmatrix}\right)=1$. By (\ref{eq:ID2}), our formula for $U(\mfr{p})^n\cdot \delta$ is the same as 
$$
\sum_{u\in \O_v/\varpi^n\O_v}
\left(\begin{array}{cc} b_1\varpi^n & 0\\ 0 & b_2\end{array}\right)\left(\begin{array}{cc} (cu+d)^{-1} & \varpi^{-n}(au+b)\\ 0 & cu+d \end{array}\right)\left(\begin{array}{cc} 1 & 0\\ c\varpi^n(cu+d)^{-1}  & 1\end{array}\right).
$$
One can write this as 
$$
\sum_{u\in \O_v/\varpi^n\O_v}
\gamma_u\left(\begin{array}{cc} b_1\varpi^n & 0\\ 0 & b_2\end{array}\right)\left(\begin{array}{cc} (cu+d)^{-1} & 0 \\ 0 & cu+d \end{array}\right)\left(\begin{array}{cc} 1 & 0\\ c\varpi^n(cu+d)^{-1}  & 1\end{array}\right)
$$ 
for some $\gamma_u\in N(\bbA_{F,f})$ for all $u\in \O_v/\varpi^n\O_v$. As elements in $\Zp[C_{\mfr{n}p^r}]$, the above formal sum equals
\begin{equation}\label{eq:key_comp}
\sum_{u\in \O_v/\varpi^n\O_v}
\left(\begin{array}{cc} b_1\varpi^n & 0\\ 0 & b_2\end{array}\right)\left(\begin{array}{cc} (cu+d)^{-1} & 0 \\ 0 & cu+d \end{array}\right)\left(\begin{array}{cc} 1 & 0\\ c\varpi^n(cu+d)^{-1}  & 1\end{array}\right).
\end{equation}
This yields 
\begin{equation}\label{eq:12}
U(\mfr{p})^n\cdot \delta=\sum_{u\in \O_v/\varpi^n\O_v}
\left(\begin{array}{cc} b_1\varpi^n & 0\\ 0 & b_2\end{array}\right)\left(\begin{array}{cc} (cu+d)^{-1} & 0\\ c\varpi^n  & cu+d \end{array}\right).
\end{equation}
Hence the element $U(\mfr{p})^n\cdot \delta$ is in $\Zp[\ol{D}_r]$.
This proves the first assertion.

For the second assertion, we claim that the element (\ref{eq:12}) is equivalent to 
$$
N(\varpi)^{n-r} \cdot \sum_{j\in \O_v/\varpi^r\O_v}
\left(\begin{array}{cc} b_1\varpi^n & 0\\ 0 & b_2\end{array}\right)\left(\begin{array}{cc} (cj+d)^{-1} & 0\\ 0  & cj+d \end{array}\right),
$$
for all $n\geq r$ and note that the sequence of such elements converges to $0$ under the $p$-adic topology as $n\map \infty$. We write $u=j+\varpi^r s$, where $j=\sum_{i=0}^{r-1} \alpha_i\varpi^i$ for $\alpha_i\in \O_v^{\times}$ and $s\in \O_v$. To prove the claim, it suffices to show that for each $u\in \O_v/\varpi^n\O_v$, there exist $X,Y\in \O_v$ such that 
$$
\left(\begin{array}{cc} (cu+d)^{-1} & 0\\ c\varpi^n  & cu+d \end{array}\right)
\left(\begin{array}{cc} (1+\varpi^rY)^{-1} & 0\\ \varpi^r X  & 1+\varpi^r Y \end{array}\right)=
\left(\begin{array}{cc} (cj+d)^{-1} & 0\\ 0  & cj+d \end{array}\right).
$$
To find $Y$, it suffices to solve the equation
$$
(cu+d)(1+\varpi^r Y)=cj+d
$$
integrally, which is possible since $cu+d\in \O_v^{\times}$ and since if two units are congruent modulo $\varpi^r$ then they differ by multiplication by an element of $1+\varpi^r\O_v$. To find $X$, we solve the equation
$$
c\varpi^n (1+\varpi^r Y)^{-1}+\varpi^r X(cu+d)=0.
$$
It is easy to see that the solution is 
$$
X=-c\varpi^{n-r}(1+\varpi^r Y)^{-1}(cu+d)^{-1}\in \O_v.
$$

For the last assertion, we consider the sequence
$$
0\map \ker\map \Zp[C_{\mfr{n}p^r}]\map e\cdot \Zp[C_{\mfr{n}p^r}]\map 0.
$$
It is enough to show that $\ker=\Zp[\ol{D}_r]$. It follows from Theorem~\ref{343}(2) that $\Zp[\ol{D}_r] \subset \ker$. Thus, it remains to show that for any $\gamma=\sum_{\delta\in C_{\mfr{n}p^r}-\ol{D}_r} a_{\delta} \delta \notin\Zp[\ol{D}_r]$, we have
$$
e\cdot \gamma\neq 0\in e\cdot \Zp[C_{\mfr{n}p^r}].
$$
Since $U(\mfr{p})\cdot \gamma = \sum_{\delta\in C_{\mfr{n}p^r}-\ol{D}_r} a_{\delta} (U(\mfr{p})\cdot \delta)$, it suffices to show two things: one is that $U(\mfr{p})^n\cdot \delta$ is a constant times a single nonzero cusp in $e\cdot \Zp[C_{\mfr{n}p^r}]$ and is not in $\ol{D}_r$ for all $n$ big enough. The other is that for all $\delta_1,\delta_2\notin \ol{D}_r$ with $\delta_1\nsim \delta_2$, $U(\mfr{p})^n \cdot \delta_1 \nsim U(\mfr{p})^n\cdot \delta_2$ for all $n$ big enough, for some $\mfr{p}|p$. 

Let $\beta g=\left(\begin{smallmatrix} b_1 & 0\\ 0 & b_2\end{smallmatrix}\right)\left(\begin{smallmatrix} a & b\\ c & d\end{smallmatrix}\right)\notin D_r$, i.e., $\val_{v}(c)=0$ for all $\mfr{p}_v|p$. We fix a prime $\mfr{p}_v|p$ and write $\mfr{p}_v=\mfr{p}$ and $\val_{v}=\val$ for simplicity. We first show that $e\cdot \beta g\neq 0$ by explicitly computing $U(\mfr{p})^n\cdot \beta g$ for all $n\geq r$. For each $n\geq r$, right multiplying the matrix $g$ by the matrix $\left(\begin{smallmatrix} 1 & \varpi^n-c^{-1}d\\ 0 & 1\end{smallmatrix}\right)\in K_{1,v}(\mfr{n}p^r)$ if necessary, we may assume $\val(d)\geq n$. It is easy to  see that $U(\mfr{p})^n \cdot \left(\begin{smallmatrix} b_1 & 0\\ 0 & b_2\end{smallmatrix}\right)\left(\begin{smallmatrix} a & b\\ c & d\end{smallmatrix}\right)$ is
\begin{equation}\label{eq:313} 
\sum_{u\in \O_v/\varpi^n\O_v} \left(\begin{array}{cc} b_1 & 0\\ 0 & b_2\end{array}\right)\left(\begin{array}{cc} \varpi^n a & au+b\\ \varpi^n c & cu+d\end{array}\right).
\end{equation}
If $u\neq 0$, then $\val(cu+d)<n$. By (\ref{eq:ID2}) and by the same argument of proving (\ref{eq:key_comp}), we obtain 
\begin{equation}\label{eq:314}
\left(\begin{array}{cc} b_1 & 0\\ 0 & b_2\end{array}\right)\left(\begin{array}{cc} \varpi^n a & au+b\\ \varpi^n c & cu+d\end{array}\right)\sim
\left(\begin{array}{cc} b_1\varpi^n & 0\\ 0 & b_2\end{array}\right)
\left(\begin{array}{cc} (cu+d)^{-1} & 0\\ 0 & cu+d\end{array}\right)
\left(\begin{array}{cc} 1 & 0\\ \varpi^n c(cu+d)^{-1} & 1 \end{array}\right)\in D_r,
\end{equation}
which is $0$ in $e\cdot \Zp[C_{\mfr{n}p^r}]$. If $u=0$, then we have
\begin{equation}\label{eq:315}
\left(\begin{array}{cc} b_1 & 0\\ 0 & b_2\end{array}\right)\left(\begin{array}{cc} \varpi^n a & b\\ \varpi^n c & d\end{array}\right)=
\left(\begin{array}{cc} b_1 & 0\\ 0 & b_2\varpi^n\end{array}\right)\left(\begin{array}{cc} \varpi^n a & b\\  c & d\varpi^{-n}\end{array}\right)
\end{equation}
Since $\val(c)=0$ and $\val(d)\geq n$, this element is in $B(F_{\mfr{p}})\GL_2(\O_v)$ and is not in $\ol{D}_r$.
%in $e\cdot \Zp[C_{\mfr{n}p^r}]$.

Next, we claim that if $\gamma_1,\gamma_2\in C_{\mfr{n}p^r}-\ol{D}_r$ are not equivalent, then $U(\mfr{p})^n\cdot \gamma_1,U(\mfr{p})^n\cdot \gamma_2\in e\cdot \Zp[C_{\mfr{n}p^r}]$ are not equivalent for some $\mfr{p}|p$ and for all $n$ big enough. We now write $\gamma_i=\beta_ig_i=\left(\begin{smallmatrix} \alpha_i & 0\\ 0 & \alpha_i'\end{smallmatrix}\right)\left(\begin{smallmatrix} a_i & b_i\\ c_i & d_i\end{smallmatrix}\right)$ with $\det g_i=1$ for $i=1,2$. Since $\gamma_i\notin \ol{D}_r$, there exists a prime ideal $\mfr{p}$ such that $\val(c_1)=0=\val(c_2)$. It follows from the above computations (\ref{eq:313}), (\ref{eq:314}), and (\ref{eq:315}) that for all $n\geq r$ and for $i=1,2$, we have
$$
U(\mfr{p})^n\cdot \gamma_i \sim \left(\begin{array}{cc} \alpha_i & 0\\ 0 & \alpha_i'\varpi^n\end{array}\right)\left(\begin{array}{cc} \varpi^n a_i & b_i\\  c_i & d_i\varpi^{-n}\end{array}\right)
\sim \left(\begin{array}{cc} \alpha_i & 0\\ 0 & \alpha_i'\varpi^n\end{array}\right)\left(\begin{array}{cc} c_i^{-1} & 0 \\ 0 & c_i\end{array}\right)\left(\begin{array}{cc} 1 & 0 \\ 1 & 1 \end{array}\right)\in e\cdot \Zp[C_{\mfr{n}p^r}].
$$
Here the second equivalence is obtained by (\ref{eq:22}) and by a similar argument of proving (\ref{eq:key_comp}). By (\ref{eq:22}) and the same argument again, we see that for $i=1,2$, we have
$$
\gamma_i \sim \left(\begin{array}{cc} \alpha_i & 0\\ 0 & \alpha_i'\end{array}\right)\left(\begin{array}{cc} c_i^{-1} & 0 \\ 0 & c_i\end{array}\right)\left(\begin{array}{cc} 1 & 0 \\ 1 & 1 \end{array}\right).
$$
Since $\gamma_1$ and $\gamma_2$ are not equivalent, $U(\mfr{p})^n\cdot \gamma_1$ and $U(\mfr{p})^n\cdot \gamma_2$ are also not equivalent for all $n\geq r$. 
\end{proof}

\begin{thm}\label{ordcusp}
Assume that $p$ is unramified in $F$. Then we have
$$
e\cdot \Zp[C_{\mfr{n}p^r}^*]\isom \Zp[C_{\mfr{n}p^r}^*]/\Zp[\ol{\mathbf{D}}_r^*].
$$
\end{thm}

\begin{proof}
By (\ref{eq:11}), to show the assertion, it is equivalent to show that
\begin{equation}\label{eq:ord_cusp}
e^*\cdot \Zp[C_{\mfr{n}p^r}]\isom \Zp[C_{\mfr{n}p^r}]/\Zp[\ol{\mathbf{D}}_r],
\end{equation}
where $e^*=\lim_{n\to \infty} \prod_{\mfr{p}|p} U^*(\mfr{p})^{n!}$. We first observe the relationship between the action of $e$ and the action of $e^*$ on $\Zp[C_{\mfr{n}p^r}]$.  Note that for $v|p$, we have
$$
K_1(\mfr{n}p^r)\left(\begin{array}{cc} \varpi & 0\\ 0 & 1\end{array}\right) K_1(\mfr{n}p^r)=\coprod_{u\in \O_v/\varpi} K_1(\mfr{n}p^r) \tau_{\mfr{p}}^{-1}\left(\begin{array}{cc} \varpi & u\\ 0 & 1\end{array}\right)^{\iota} \tau_{\mfr{p}},
$$
where $\tau_{\mfr{p}}=\left(\begin{smallmatrix} 0 & -1\\ \varpi^{\val_{v}(p)r} & 0 \end{smallmatrix}\right)\in \GL_2(F_{v})$. In fact, we have $\tau_{\mfr{p}}=\left(\begin{smallmatrix} 0 & -1\\ \varpi^r & 0\end{smallmatrix}\right)$, since $p$ is unramified in $F$. We will also view $\tau_{\mfr{p}}$ as an element in $\GL_2(\bbA_{F,f})$ whose entries at finite places of $F$ other than $\mfr{p}$ are identity matrices. The operator $\iota$ is the standard involution on $\GL_2$ defined as $\left(\begin{smallmatrix} a & b\\ c & d\end{smallmatrix}\right)^{\iota}=\left(\begin{smallmatrix} d & -b\\ -c & a \end{smallmatrix}\right)$. Note that one has $AA^{\iota}=\det A\cdot I_2$ for all $A\in \GL_2$, where $I_2$ is the identity matrix in $\GL_2$. It follows from this applied to $A=\left(\begin{smallmatrix} \varpi & u \\ 0 & 1 \end{smallmatrix}\right)$ for some $u\in \O_v$ and the definition of $U^*(\mfr{p})$ that
$$
U^*(\mfr{p}) \cdot c =\tau_{\mfr{p}}^{-1}U(\mfr{p})\tau_{\mfr{p}} \left(\begin{array}{cc} \varpi^{-1} & 0\\ 0 & \varpi^{-1}\end{array}\right)\cdot c
$$ 
for all $c\in C_{\mfr{n}p^r}$. Here for an element $\gamma \in \GL_2(\bbA_{F,f})$, $\gamma\cdot c$ is defined as $c\gamma$ for all $c\in C_{\mfr{n}p^r}$. It is easy to see that $\left(\begin{smallmatrix} \varpi^{-1} & 0\\ 0 & \varpi^{-1}\end{smallmatrix}\right)\cdot C_{\mfr{n}p^r}=C_{\mfr{n}p^r}$, so one obtains $U^*(\mfr{p}) \cdot \Zp[C_{\mfr{n}p^r}]=\tau_{\mfr{p}}^{-1}U(\mfr{p}) \tau_{\mfr{p}}\cdot \Zp[C_{\mfr{n}p^r}]$ and hence, $U^*(\mfr{p})^n \cdot \Zp[C_{\mfr{n}p^r}]=\tau_{\mfr{p}}^{-1}U(\mfr{p})^n\tau_{\mfr{p}}\cdot \Zp[C_{\mfr{n}p^r}]$ for all positive integers $n$. This yields 
\begin{equation}
e^*\cdot \Zp[C_{\mfr{n}p^r}]=\tau_{p}^{-1} e \tau_{p}\cdot \Zp[C_{\mfr{n}p^r}],
\end{equation}
where $\tau_p=\prod_{\mfr{p}|p} \tau_{\mfr{p}}$.

Next, we set $K=\prod_{v\nmid p} K_{1,v}(\mfr{n}p^r)\times \prod_{v|p} K_v^1(\mfr{n}p^r)$ and 
$$
C_{\mfr{n}p^r}^1=B(F)^+ N(\bbA_{F,f})\backslash T(\bbA_{F,f})\GL_2(\wh{\O}_F)/K.
$$
Recall that the group $K_v^1(\mfr{n}p^r)$ was defined in Section~\ref{sec:22}. Then we have $\tau_p^{-1}\cdot \Zp[C_{\mfr{n}p^r}]=\Zp[C_{\mfr{n}p^r}^1]$ as $\tau_{p}^{-1} \cdot c$ is a cusp  for $\tau_p K_1(\mfr{n}p^r)\tau_p^{-1}=K$. Note that in the proof of Theorem~\ref{343}, when we showed two cusps are equivalent by multiplying by matrices in $K_1(\mfr{n})$, those matrices are always in $K_1^1(\mfr{n})$. Therefore, by the same argument as in Theorem~\ref{343}, we have
$$
e\cdot \Zp[C_{\mfr{n}p^r}^1]\isom \Zp[C_{\mfr{n}p^r}^1]/\Zp[B(F)^+ N(\bbA_{F,f})\backslash D_r/K].
$$

Finally, we observe that $\tau_p\cdot D_r=\mathbf{D}_r$. To see this, we suppose that $\delta=\beta g=\left(\begin{smallmatrix} b_1 & 0 \\ 0 & b_2 \end{smallmatrix}\right)\left(\begin{smallmatrix} a & b\\ c & d \end{smallmatrix}\right)\in D_r$ is given. We fix a prime ideal $\mfr{p}=\mfr{p}_v$ of $F$ dividing $p$ such that $\val_{v}(c)=s>0$. If $s\geq r$, then we have
$$
\tau_{\mfr{p}}\cdot \delta=\beta\left(\begin{array}{cc} b\varpi^r & -a \\ d\varpi^r & -c \end{array}\right)=\left(\begin{array}{cc} b_1 & 0 \\ 0 & b_2\varpi^r \end{array}\right)\left(\begin{array}{cc} b\varpi^r & -a \\ d & -c\varpi^{-r} \end{array}\right)\in \mathbf{D}_r.
$$
If $0<s<r$, by a similar computation we again have $\tau_{\mfr{p}}\cdot \delta\in \mathbf{D}_r$. 
Thus we see that
$$
\tau_p \cdot \Zp[C_{\mfr{n}p^r}^1]/\Zp[B(F)^+ N(\bbA_{F,f})\backslash D_r/K]\isom\Zp[C_{\mfr{n}p^r}]/\Zp[\ol{\mathbf{D}}_r].
$$  
To sum up, we have shown that 
\[
\begin{split}
e^*\cdot\Zp[C_{\mfr{n}p^r}]
=  \tau_{p}^{-1} e\tau_{p}\cdot \Zp[C_{\mfr{n}p^r}]
%=  \tau_p e\cdot\Zp[C_{\mfr{n}p^r}^1]
\isom  \tau_p \cdot \Zp[C_{\mfr{n}p^r}^1]/\Zp[B(F)^+ N(\bbA_{F,f})\backslash D_r/K]
\isom  \Zp[C_{\mfr{n}p^r}]/\Zp[\ol{\mathbf{D}}_r].
\end{split}
\]
This proves (\ref{eq:ord_cusp}) and hence, the assertion follows.
\end{proof}

%%%%%%%%%%%%%%%%%%%%%%%%%%%%%%%%%%%%%%%%%%%%%%%%%%%%%%%%%%%%%%%%%%
\section{Automorphic forms}\label{sec:03}
The main goal of this section is to construct the Eisenstein series in Proposition~\ref{333} as automorphic forms and to compute their constant terms at different cusps. To do this, we have to compute local integrals. Some results regarding these computations can be found in the literature, except for Lemma~\ref{local_wittaker} and Lemma~\ref{415}. 
%These lemmas will be used in the last subsection to compute the constant terms of Eisenstein series at cusps other than $\infty$. 

Throughout this section, we fix an additive character $\psi=\otimes_{v} \psi_v$ on $\bbA_F$ defined as $\psi_v(x)=e^{2\pi i x}$ if $v|\infty$ and $\psi_v(x)=e^{-2\pi i [\Tr_{F_v/\Qp}(x)]_p}$ if $v|p$, where $[x]_p$ is the decimal part of $x$ for $x\in \Qp$, i.e., if $x=\sum_{i=n}^{\infty} a_i p^i$ for some $n\in \Z_{<0}$, then $[x]_p=\sum_{i=n}^{-1} a_i p^i$. Here $v$ runs through all places of $F$. For $v<\infty$, we say that the additive character $\psi_v$ is of conductor $\varpi_v^{-r}$ if $r$ is the smallest positive integer such that $\psi_v(\varpi_v^{-r}\O_v)=1$ or equivalently, $r$ is the smallest positive integer such that $\Tr_{F_v/\Qp}(\varpi_v^{-r})\in \Zp$. Indeed, $\mfr{p}_v^{-r}$ is the inverse different of $F_v$ over $\Qp$.

In addition, we fix a self-dual Haar measure $dn=\otimes_v dn_v$ defined as follows. For each finite place $v$, $dn_{v}$ is the normalized Haar measure such that the volume of $\O_{v}$ is $1$, and for each infinite place the Haar measure $dn_{\infty}$ is normalized such that the volume of $\R/\Z$ is $1$.

Finally, we denote by $\chi=(\chi_1,\chi_2)$ a pair of primitive narrow ray class characters  of conductors $\mfr{n}_1$ and $\mfr{n}_2$, respectively, with associated signs $e_{1,\infty},e_{2,\infty} \in (\Z/2\Z)^d$ satisfying (\ref{eq:eisen_cond}). For $i=1,2$, we write $\mfr{n}_i=\prod_{\mfr{p}|\mfr{n}_i}\mfr{p}^{e_{\mfr{p},i}}$, and the finite part $\chi_{i,f}$ of $\chi_i$ can be decomposed as
$$
\chi_{i,f}=\prod_{\mfr{p}|\mfr{n}_i}\chi_i^{(\mfr{p})}:\prod_{\mfr{p}|\mfr{n}_i} (\O/\mfr{p}^{e_{\mfr{p},i}})^{\times}\map \C^{\times}.
$$
We will denote by $\chi_i=\prod_{i=1}^d \chi_{i,\tau_i}\times \prod_{v<\infty}\chi_{i,v}$ their corresponding Hecke characters, where $\prod_{i=1}^d \chi_{i,\tau_i}$ is the infinity part of $\chi_i$.  
%%%%%%%%%%%%%%%%%%%%%%%%%%%%%%%%%%%%%%%%%%%%%%%%%%%%%%%%%%%%%%%%%%%
\subsection{Preliminaries}\label{sec:31}
In this subsection, we set up some notation for character sums, which will simplify the notation and computation in Section~\ref{sec:32}. Throughout this subsection, we fix a finite place $v$ of $F$. For simplicity, we will omit $v$ from $\psi_v$, $\mfr{p}_v$, $\varpi_v$, $q_v$, and $\val_v$.

For a positive integer $n$, we put $U^{(n)}:=\O_v^{\times}/(1+\varpi^n\O_v)$. For $x\in U^{(n)}$, we abuse the notation to denote by $x$ a lifting in $\O_v^{\times}$. For a primitive ramified  character $\theta$ of $F_v^{\times}$ of conductor $\varpi^e$, the Gauss sum $\tau(\theta)$ associated with $\theta$ is defined as
$$
\tau(\theta):=\sum_{x\in U^{(e)}} \theta(x) \psi(\varpi^{-e-r}x).
$$
This definition does not depend on the choice of the lifting of $x\in U^{(e)}$ in $\O_v^{\times}$ as $\theta$ is trivial on $1+\varpi^e\O_v$ and 
$$
\psi(\varpi^{-e-r}(x(1+\varpi^e\O_v))=\psi(\varpi^{-e-r}x)\psi(\varpi^{-r}\O_v)=\psi(\varpi^{-e-r}x).
$$
Moreover, it is known \cite[Proposition~7.5]{Neu} that \begin{equation}\label{eq:gauss_sum}
\tau(\theta)\ol{\tau(\theta)}=N_{F_v/\Qp}(\mfr{p}^{e}).
\end{equation}

For two primitive ramified characters $\theta_1,\theta_2$ of $F_v^{\times}$ of conductors $\varpi^{e_1}$ and $\varpi^{e_2}$, respectively, the Jacobi sum associated with $\theta_1$ and $\theta_2$ is defined as
$$
J_a(\theta_{1},\theta_{2}, \varpi^k):= \sum_{x\in U^{(k)}} \theta_{1}(x)\theta_{2}(a-x)
$$
for $a\in \O_v$ and $k\geq \max\{e_1,e_2\} \in \Z_{> 0}$. It is easy to see that $J_a(\theta_{1},\theta_{2},\varpi^k)=J_a(\theta_{2},\theta_{1}, \varpi^k)$, and $J_a(\theta_{1}\theta, \theta_{2},\varpi^k)=J_a(\theta_{1},\theta_{2},\varpi^k)$ if $\theta$ is an unramified character of $F_v^{\times}$. 
The following lemma is obtained by a direct computation (or see \cite[Theorem 2.5]{Jun}).

\begin{lemma}\label{251}
Let the notation be as above.
\begin{enumerate}

\item Assume that both $\theta_1$ and $\theta_2$ are ramified, and  assume that $e_{1}=e_{2}$. Let $\varpi^s$ be the conductor of $\theta_{1}^{-1}\theta_{2}$. If $1\leq s\leq e_{1}$, then 
$$
J_{\varpi^{e_{1}-s}}(\theta_{1}^{-1},\theta_{2},\varpi^{e_{1}})=\tau(\theta_{1}^{-1})\tau(\theta_{2})/\tau(\theta_{1}^{-1}\theta_{2}).
$$

\item Assume that both $\theta_1$ and $\theta_2$ are ramified. If $e_{1}\neq e_{2}$, then 
we have
$$
\tau(\theta_{1}^{-1})\tau(\theta_{2})=\tau(\theta^{-1}_{1}\theta_{2})\times
\begin{cases}
\sum_{x \in U^{(e_2)}} \theta_{2}(x)\theta_{1}^{-1}(1-\varpi^{e_{1}-e_{2}}x) &\mbox{if } e_{1}>e_{2}\\
\sum_{x \in U^{(e_1)}} \theta_{1}^{-1}(x)\theta_{2}(1-\varpi^{e_{2}-e_{1}}x) & \mbox{if } e_{1}<e_{2}.
\end{cases}
$$

\end{enumerate}
\end{lemma}

The following lemma will be used in the proof of Lemma~\ref{415}.

\begin{lemma}
Let the notation be as above. 
\begin{enumerate}\label{253}
\item If $e_1>e_2$, then we have
$$
\int_{\O_v^{\times}} \theta_1(1+\varpi^{e_1-e_2}x)\theta_2(x)dx
= |\varpi|_v^{e_2} \sum_{x\in U^{(e_2)}}\theta_2(x)\theta_1(1+\varpi^{e_1-e_2}x).
$$

\item If $e_1=e_2$, then we have
$$
\int_{\O_v^{\times}} \theta_1(\varpi^{k-e_1}+x)\theta_2(x) dx
= |\varpi|_v^{e_1} \sum_{x\in U^{(e_1)}} \theta_2(x)\theta_1(\varpi^{k-e_1}+x)
$$
for all positive integers $k\geq e_1$.
\end{enumerate}
\end{lemma}

\begin{proof}
Since the computations for both assertions are essentially the same, we will only prove the first assertion. Assume that $e_1>e_2>0$. Then we have
\[
\begin{split}
\int_{\O_v^{\times}} \theta_{1}(1+\varpi^{e_{1}-e_{2}}x) \theta_{2}(x) dx 
&=  \sum_{x\in U^{(e_{2})}} \theta_{2}(x) \int_{1+\varpi^{e_{2}}\O_v} \theta_{1}(1+\varpi^{e_{1}-e_{2}}xy) dy \\
&=  \sum_{x\in U^{(e_{2})}} \theta_{2}(x) \theta_{1}(1+\varpi^{e_{1}-e_{2}}x) \int_{\O_v} 1 d(\varpi^{e_{2}}z) \\
&=  |\varpi^{e_2}|_v \sum_{x\in U^{(e_{2})}} \theta_{2}(x) \theta_{1}(1+\varpi^{e_{1}-e_{2}}x).
\end{split}
\] 
Note that the second equality is obtained by letting $y=1+\varpi^{e_{2}}z$ and by the observation that  
$$
\theta_{1}(1+\varpi^{e_{1}-e_{2}}xy)= \theta_{1}(1+\varpi^{e_1-e_{2}}x+\varpi^{e_{1}}z)=\theta_{1}(1+\varpi^{e_{1}-e_{2}}x)
$$
for all $x\in \O_v^{\times}$ and $z\in \O_v$.
\end{proof}
%%%%%%%%%%%%%%%%%%%%%%%%%%%%%%%%%%%%%%%%%%%%%%%%%%%%%%%%%%%%%%%%%%%%%%%%%%%%%%%%
\subsection{Non-archimedean sections}\label{sec:32}
In this subsection, a finite place $v$ of $F$ is still fixed. Let $\theta_1$ and $\theta_2$ be primitive characters of $F_v^{\times}$ of conductor $\varpi^{e_1}$ and $\varpi^{e_2}$, respectively, and set $\theta=(\theta_1,\theta_2)$. Let other notation be as in the previous subsection. The goal of this subsection is to compute non-archimedean integrals that will be used, by taking $\theta_i=\chi_{i,v}$ for each finite place $v$, to compute the constant terms of Eisenstein series in the last subsection. 

Let $I(\theta_1|\cdot|_v^s,\theta_2|\cdot|_v^{-s})$ be the set of all functions $f_s:\GL_2(F_v)\map \C$ satisfying
\begin{equation}\label{eq:13}
f_s\left(\left(\begin{array}{cc}
    a & b \\ 0   & d
\end{array}\right)g\right)=\theta_1(a)\theta_2(d)\left|\frac{a}{d}\right|_v^{{s+\tfrac{1}{2}}}f_s(g)
\end{equation}
for $a,d\in F_v^{\times}$, $b\in F_v$, $g\in \GL_2(F_v)$, and $s\in \C$.
The following proposition was proved in \cite[Theorem 1.1]{Cas}.
\begin{prop}
There exists a unique section $f_{\theta,s,v} \in I(\theta_1|\cdot|_v^s,\theta_2|\cdot|_v^{-s})$ distinguished by the following properties:
\begin{enumerate} 
\item for $g\in \GL_2(F_v)$, we have
$
f_{\theta,s,v}\left(g\left(\begin{smallmatrix}
    a & b \\
    c   & d
\end{smallmatrix}\right)\right)=f_{\theta,s,v}(g)\theta_1\theta_2(d)
$
for $\left(\begin{smallmatrix}
    a & b \\c   & d
\end{smallmatrix}\right)\in K_{0,v}(\mfr{n})$, 

\item  we have $f_{\theta,s,v}(\left(\begin{smallmatrix}
    1 & 0 \\ \varpi^i   & 1
\end{smallmatrix}\right))=
\begin{cases}
\theta_1(\varpi^{-e_2}) & \mbox{ if } i=e_2  \\ 
0 & \mbox{ otherwise}. 
\end{cases}
$
\end{enumerate}
Moreover, the definition of $f_{\theta,s,v}$ does not dependent on the choice of a uniformizer. 
\end{prop}

The following lemma is obtained by a direct computation (see \cite[Proposition 2.1.2]{Sch}).

\begin{lemma}\label{412}
Let $f_{\theta,s,v}$ be defined as above. 
\begin{enumerate}
\item If $\theta_1$ and $\theta_2$ are ramified, then 
      \begin{equation}\label{eq:42}
      f_{\theta,s,v}\left(\left(\begin{array}{cc}
      1 & 0 \\ x   & 1
      \end{array}\right)\right)=
      \begin{cases}
      \theta^{-1}_{1,v}(x) & \mbox{if } \val(x)=e_2\\
      0 & \mbox{if } \val(x)\neq e_2.
      \end{cases}
      \end{equation}

\item If $\theta_1$ is unramified and $\theta_2$ is ramified, then 
      \begin{equation}\label{eq:43}
      f_{\theta,s,v}\left(\left(\begin{array}{cc}
      1 & 0 \\ x   & 1
      \end{array}\right)\right)=
      \begin{cases}
      \theta^{-1}_{1,v}(\varpi^{e_2}) & \mbox{if } \val(x)\geq e_2\\
      0 & \mbox{if } \val(x)< e_2.
      \end{cases}
      \end{equation}
      
\item If $\theta_1$ is ramified and $\theta_2$ is unramified, then 
      \begin{equation}\label{eq:44}
      f_{\theta,s,v}\left(\left(\begin{array}{cc}
      1 & 0 \\ x   & 1
      \end{array}\right)\right)=
      \begin{cases}
      \theta_1^{-1}\theta_2(x) |x|_v^{-(2s+1)} & \mbox{if } \val(x)\leq 0\\
      0 & \mbox{if } \val(x)>0.
      \end{cases}
      \end{equation}
    
\item If $\theta_1$ and $\theta_2$ are unramified, then 
      \begin{equation}\label{eq:45}
      f_{\theta,s,v}\left(\left(\begin{array}{cc}
      1 & 0 \\ x   & 1
      \end{array}\right)\right)=
      \begin{cases}
      \theta_1^{-1}\theta_2(x)|x|_v^{-(2s+1)} & \mbox{if } \val(x) < 0\\
      1 & \mbox{if } \val(x)\geq 0.
      \end{cases}
      \end{equation}
\end{enumerate}
\end{lemma}

We now define the intertwining operator and review its properties following \cite[\S 4.5]{B}. The intertwining operator ${M_v}{f_{\theta,s,v}}:\GL_2(F_v)\map \C$ is defined as the integral
$$
{M_v}{f_{\theta,s,v}}(g)=\int_{F_v} f_{\theta,s,v}\left(\left(\begin{array}{cc}
      0 & -1 \\ 1   & 0\end{array}\right)
      \left(\begin{array}{cc}
      1 & n \\ 0   & 1\end{array}\right)g\right) dn
$$
for all $g\in \GL_2(F_v)$. This integral converges absolutely when $\re(s)\gg 0$ and one can prove that ${M_v}{f_{\theta,s,v}}\in I(\theta_2|\cdot|_v^{-s}, \theta_1|\cdot|_v^s)$ (see Proposition~4.5.6 in \textit{loc.~cit.}). Moreover, it has analytic continuation to all $s$ (see Proposition~4.5.7 in \textit{loc.~cit.}). 
%aside from a pole at $s=t$ if $\theta_1\theta_2^{-1}(x)=|x|_v^t$ for all $x\in F_v^{\times}$. 

For $\beta\in F^{\times}$, the integral
$$
\int_{F_v} f_{\theta,s,v}\left(\left(\begin{array}{cc}
      0 & -1 \\ 1   & 0\end{array}\right)
      \left(\begin{array}{cc}
      1 & n \\ 0   & 1\end{array}\right)g\right) \psi(-\beta n)dn
$$
converges absolutely when $\re(s)\gg 0$ and has analytic continuation to all of $\C$ (see p.~498 in  \textit{loc.~cit.}). In what follows, we first assume that $\re(s)$ is big enough so that the above integrals converge absolutely for all $\beta\in F$, and the lemmas then follow by meromorphic continuation.

\begin{lemma}\label{local_wittaker}
For any $\left(\begin{smallmatrix} a & b\\0 &   
d\end{smallmatrix}\right)\in B(F_v)$ and for any $\beta\in F$, the integral
$$
\int_{F_v}f_{\chi,s,v}\left(\left(\begin{array}{cc}
    0 & -1 \\1 & 0
    \end{array}\right)\left(\begin{array}{cc}
      1 & n \\0 & 1
    \end{array}\right)
    \left(\begin{array}{cc}
    a & b \\ 0 & d
    \end{array}\right)g\right) \psi_v(-\beta n) dn
$$ 
equals 
$$
\chi_{1,v}(d)\chi_{2,v}(a)|a/d|_v^{\frac{1}{2}-s} \int_{F_{v}}f_{\chi,s,v}\left(\left(\begin{array}{cc}
    0 & -1 \\1 & 0
    \end{array}\right)\left(\begin{array}{cc}
      1 & n \\0 & 1
    \end{array}\right)g\right) \psi_v(-\beta a d^{-1} n) dn
$$
for all $g\in \GL_2(F_v)$. 
\end{lemma}

\begin{proof}
For $\beta\in F$, we have 
\begin{align*}
&\int_{F_{v}}f_{\chi,s,v}\left(\left(\begin{array}{cc}
    0 & -1 \\1 & 0
    \end{array}\right)\left(\begin{array}{cc}
      1 & n \\0 & 1
    \end{array}\right)
    \left(\begin{array}{cc}
    a & b \\ 0 & d
    \end{array}\right)\right) \psi_v(-\beta n) dn\\
=&\  \int_{F_{v}}f_{\chi,s,v}\left(\left(\begin{array}{cc}
    d & 0 \\ 0 & a
    \end{array}\right)\left(\begin{array}{cc}
    0 & -1 \\1 & 0
    \end{array}\right)\left(\begin{array}{cc}
      1 & (b+nd)a^{-1} \\0 & 1
    \end{array}\right)\right) \psi_v(-\beta n)dn\\
=&\  \chi_{1,v}(d)\chi_{2,v}(a)|d/a|^{s+\frac{1}{2}}\int_{F_{v}}f_{\chi,s,v}\left(\left(\begin{array}{cc}
    0 & -1 \\1 & 0
    \end{array}\right)\left(\begin{array}{cc}
      1 & (b+nd)a^{-1} \\0 & 1
    \end{array}\right)\right) \psi_v(-\beta n)dn\\
=&\  \chi_{1,v}(d)\chi_{2,v}(a)|d/a|^{s-\frac{1}{2}}\int_{F_{v}}f_{\chi,s,v}\left(\left(\begin{array}{cc}
    0 & -1 \\1 & 0
    \end{array}\right)\left(\begin{array}{cc}
      1 & n \\0 & 1
    \end{array}\right)\right) \psi_v(-\beta a d^{-1} n) dn.\qedhere
\end{align*}
\end{proof}

\begin{lemma}\label{413}
Let the notation be as above.
\begin{enumerate}
\item If $\theta_1$ and $\theta_2$ are unramified, then
$$
{M_v}{f_{\theta,s,v}}\left(\left(\begin{array}{cc}
    1 & 0 \\0 & 1
\end{array}\right)\right)=
\frac{1-\theta_1\theta_2^{-1}(\varpi)q^{-(2s+1)}}{1-\theta_1\theta_2^{-1}(\varpi)q^{-2s}}. 
$$

\item If $\theta_1$ is ramified and $\theta_2$ is unramified, then 
$$
{M_v}{f_{\theta,s,v}}\left(\left(\begin{array}{cc}
    1 & 0 \\0 & 1
\end{array}\right)\right)= 1.
$$

\item If $\theta_2$ is ramified, then 
$$
{M_v}{f_{\theta,s,v}}\left(\left(\begin{array}{cc}
    1 & 0 \\0 & 1
\end{array}\right)\right)= 0.
$$
\end{enumerate}
\end{lemma}

\begin{proof}
See \cite[Proposition~4.6.7]{B} for a proof.
\end{proof}

Recall that at the beginning of this section, we fixed local additive characters $\psi_v$ of conductor $\varpi^{-r}$ for each finite place $v$ of $F$. Following \cite[\S 3]{Tate}, the local epsilon factor is defined as 
$$
\varepsilon_v(s,\theta_2^{-1},\psi)= \int_{\varpi_v^{-e_2-r}\O_v^{\times}}\theta_2(n) |n|^{-s}\psi_v(-n) dn.
$$
For a number $\beta\in F^{\times}$, we set $\gamma_v(\beta):=\int_{\O_v} \psi_v(-\beta n) dn$. By a simple computation, one sees that $\gamma_v(\beta)=1$ if $\Tr_{F_v/\Qp}(\beta)\in \Zp$; otherwise, it is $0$. The following lemma is obtained by a direct computation (see \cite[Lemma 2.2.1]{Sch}). 

\begin{lemma}\label{414}
Let $\beta\in F^{\times}$. 
\begin{enumerate}
\item Assume that the conductor of $\psi_v$ is $\varpi^{-r}$ for some $r\in \Z_{\geq 0}$, and assume that  $\theta_1$ and $\theta_2$ are unramified. Then
\[
\begin{split} 
\int_{F_v}f_{\theta,s,v}\left(\left(\begin{array}{cc}
    0 & -1 \\1 & 0
\end{array}\right)\left(\begin{array}{cc}
    1 & n \\0 & 1
\end{array}\right)\right)
 \psi_v(-\beta n) dn
=  (1-\theta_1\theta_2^{-1}(\varpi_v) q^{-(2s+1)})
 \sum_{n=0}^{r+\val(\beta)} \theta_1\theta_2^{-1}(\varpi_v^n)q^{-2sn}
\end{split}
\]
if $\val(\beta)+r\geq 0$. Otherwise, the integral is $0$.
    
\item If $\theta_1$ is ramified and $\theta_2$ is unramified, then 
    $$
    \int_{F_v}f_{\theta,s,v}\left(\left(\begin{array}{cc}
    0 & -1 \\1 & 0
    \end{array}\right)\left(\begin{array}{cc}
      1 & n \\0 & 1
    \end{array}\right)\right)
    \psi_v(-\beta n) dn= \gamma_v(\beta).
    $$

\item If $\theta_2$ is ramified and $\theta_1$ is unramified, then
    \[
    \begin{split} 
    \int_{F_v}f_{\theta,s,v}\left(\left(\begin{array}{cc}
    0 & -1 \\1 & 0
    \end{array}\right)\left(\begin{array}{cc}
    1 & n \\0 & 1
    \end{array}\right)\right) 
    \psi_v(-\beta n) dn
     = \theta_1\theta_2^{-1}(\beta)|\beta|_v^{2s}\varepsilon_v(2s+1, \theta_2^{-1},\psi)
    \end{split}
    \]
    if $\val(\beta)\geq 0$; otherwise, the integral is $0$.
   
\item If $\theta_1$ and $\theta_2$ are ramified, then
    \[
    \begin{split} 
     \int_{F_v}f_{\theta,s,v}\left(\left(\begin{array}{cc}
    0 & -1 \\1 & 0
    \end{array}\right)\left(\begin{array}{cc}
    1 & n \\0 & 1
    \end{array}\right)\right)
    \psi_v(-\beta n) dn
     =\theta_2^{-1}(\beta)\varepsilon_v(2s+1,\theta_2^{-1},\psi)
    \end{split}
    \]
    if $\val(\beta)=0$; otherwise, the integral is $0$. 
    \end{enumerate}
\end{lemma}

We now compute $f_{\theta,s,v}\left(\left(\begin{smallmatrix}
    0 & -1 \\1 & 0
\end{smallmatrix}\right)\left(\begin{smallmatrix}
    1 & n \\0 & 1
\end{smallmatrix}\right)
\left(\begin{smallmatrix}
    1 & 0 \\ \varpi^i & 1
\end{smallmatrix}\right)\right)$
for $i\in \Z_{\geq 0}$ and $n\in F_v$, which will be used in the proof of Lemma~\ref{415}. First of all, we have
\begin{equation}\label{eq:412}
\left(\begin{array}{cc} 
0 & -1 \\1 & 0
\end{array}\right)
\left(\begin{array}{cc}
    1 & n \\0 & 1
\end{array}\right)
\left(\begin{array}{cc}
    1 & 0 \\ \varpi^i & 1
\end{array}\right)
= \left(\begin{array}{cc} 
-\varpi^i & -1 \\ 1+\varpi^i n & n
\end{array}\right).
\end{equation}
When $i\neq 0$, by (\ref{eq:22}) for $n\in \O_v$ and by (\ref{eq:ID2}) for $n\in F_v-\O_v$, we have
\begin{equation}\label{eq:413}
\left(\begin{array}{cc} 
-\varpi^i & -1 \\ 1+\varpi^i n & n
\end{array}\right)=
\begin{cases}
\renewcommand*{\arraystretch}{0.5}
\left(\begin{array}{cc} 
(1+n\varpi^i)^{-1} & * \\ 0 & 1+n\varpi^i
\end{array}\right)
\left(\begin{array}{cc} 
1 & 0 \\ 1 & 1
\end{array}\right)
\kappa &\mbox{if }n\in \O_v\\
\renewcommand*{\arraystretch}{0.5}
\left(\begin{array}{cc} 
n^{-1} & * \\ 0 & n
\end{array}\right)
\left(\begin{array}{cc} 
1 & 0 \\ (1+n\varpi^i)n^{-1} & 1
\end{array}\right) &
\mbox{if } n\in F_v-\O_v
\end{cases}
\end{equation}
for some $\kappa\in K_{1,v}(\mfr{n})$. Thus, for $n\in \O_v$, by Lemma~\ref{412}, we have
\begin{equation}\label{eq:414}
f_{\theta,s,v}\left(\left(\begin{array}{cc}
    0 & -1 \\1 & 0
\end{array}\right)
\left(\begin{array}{cc}
    1 & n \\0 & 1
\end{array}\right)
\left(\begin{array}{cc}
    1 & 0 \\ \varpi^i & 1
\end{array}\right)\right)=
\begin{cases}
0 & \mbox{if } e_2\neq 0\\
\theta_1^{-1}(1+n\varpi^i) & \mbox{if } e_2=0.
\end{cases}
\end{equation}
Moreover, for $n\in F_v-\O_v$, we have
\begin{equation}\label{eq:415}
f_{\theta,s,v}\left(\left(\begin{array}{cc}
    0 & -1 \\1 & 0
\end{array}\right)
\left(\begin{array}{cc}
    1 & n \\0 & 1
\end{array}\right)
\left(\begin{array}{cc}
    1 & 0 \\ \varpi^i & 1
\end{array}\right)\right)=
\theta_1^{-1}\theta_2(n)|n|^{-(2s+1)}
f_{\theta,s,v}\left(\left(\begin{array}{cc}
    1 & 0 \\ (1+n\varpi^i)n^{-1} & 1
\end{array}\right)\right).
\end{equation}

\begin{lemma}\label{415}
Let $i<e_1+e_2$ be a non-negative integer, and let $M_{v,s,i}:={M_v}{f_{\theta,s,v}}
    \left(\left(\begin{smallmatrix}
    1 & 0 \\ \varpi^i & 1
    \end{smallmatrix}\right)\right)$.
\begin{enumerate}  
\item Suppose one of the following conditions holds
\begin{enumerate}
\item $\theta_1$ is unramified, $\theta_2$ is ramified, and $0<i<e_2$.

\item $\theta_1$ is ramified, $\theta_2$ is unramified, and $0\leq i<e_1$.

\item $\theta_1$ and $\theta_2$ are ramified, $0\leq i<e_1+e_2$, and $i\neq e_1$. 
\end{enumerate}
Then $ M_{v,s,i}=0$.
\item If $\theta_1$ is unramified and $\theta_2$ is ramified, then 
    $$
     M_{v,s,0}= \theta_1^{-1}(\varpi^{e_2})\theta_2(-1)q^{-e_2}.
    $$

\item If both $\theta_1$ and $\theta_2$ are ramified, let $\varpi^t$ be the conductor of $\theta_1^{-1}\theta_2$. 
\begin{enumerate}
\item[(3.i)] If $ e_1>e_2$, then  
    $$
M_{v,s,e_1}
=  \theta_2(-\varpi^{-e_2}) |\varpi^{e_2}|_v^{2s+1} \sum_{x \in U_{\mfr{p}}^{(e_2)}} \theta_2(x)\theta_1^{-1}(1-\varpi^{e_1-e_2}x). 
   %\sum_{x\in U_{\mfr{p}}^{(e_2)}} \theta_2(-x) \theta_1^{-1}(\varpi^{e_{\mfr{p}_2}-e_1}+x)
    $$

\item[(3.ii)] If $e_1< e_2$, then 
    $$
M_{v,s,e_1}
= \theta_1^{-1}(\varpi^{e_2-e_1}) \theta_2(-\varpi^{-e_1})  |\varpi|_v^{2se_1+e_2} 
   \sum_{x \in U_{\mfr{p}}^{(e_1)}} \theta_1^{-1}(x)\theta_2(1-\varpi^{e_2-e_1}x).
    %\sum_{x\in U_{\mfr{p}}^{(e_1)}}\O_v \theta_1^{-1}(x) \theta_2(\varpi^{e_1-{e_2}}-x)
    $$

\item[(3.iii)]  If $e_1=e_2$ and if $t=e_1$, then 
    $$
M_{v,s,e_1}
= \theta_2(-\varpi^{-e_1}) |\varpi|_v^{e_1(2s+1)} J_1(\theta_2,\theta_1^{-1}, \varpi^{e_1}).
    $$

\item[(3.iv)] If $e_1=e_2$ and if $1\leq t< e_1$, then 
    $$
M_{v,s,e_1}
= \theta_1^{-1}(\varpi^{t-e_1})\theta_2(-\varpi^{t-2e_1})|\varpi|_v^{2s(2e_1-t)+e_1} J_{\varpi^{e_1-t}}(\theta_2,\theta_1^{-1}, \varpi^{e_1}).
    %\sum_{x\in U_{\mfr{p}}^{(e_1)}} \theta_2(-x)\theta_1(x+\varpi^{e_1-s}),
    $$
    
\item[(3.v)] If $e_1=e_2$ and if $t=0$, then 
    \[
    \begin{split}
M_{v,s,e_1}=&\  \theta_1^{-1}(\varpi^{e_1}) \theta_2(-1) |\varpi|_v^{2s(2e_1-1)+e_1} J_{\varpi^{e_{\mfr{{p},1}}-1}}(\theta_2,\theta_1^{-1}, \varpi^{e_1})+\\
    %\sum_{x\in U_{\mfr{p}}^{(e_1)}} \theta_2(-x)\theta_1(x+\varpi^{e_1-1})+    
    & \sum_{k=2e_1}^{\infty} \theta_1^{-1}(\varpi^{e_1})\theta_2(-1) |\varpi|_v^{2sk+e_1} J_0(\theta_1^{-1},\theta_2, \varpi^{e_1}).
    %\sum_{x\in U_{\mfr{p}}^{(e_1)}} \theta_2(-x)\theta_1^{-1}(x),
    \end{split}
    \]
\end{enumerate}
\end{enumerate}
\end{lemma}

\begin{proof}
If follows from the definition of $f_{\theta,s,v}$ that one has
$$
{M_v}f_{\theta,s,v}\left(g\left(\begin{array}{cc}
    a & b \\ c & d
\end{array}\right)\right)={M_v}f_{\theta,s,v}(g)\theta_1\theta_2(d)
$$
for $\left(\begin{smallmatrix}
    a & b \\ c & d
\end{smallmatrix}\right) \in K_{0,v}(\mfr{n})$ and $g\in \GL_2(F_v)$.
By this and the fact that ${M_v}f_{\theta,s,v}\in I(\theta_2|\cdot|^{-1},\theta_1|\cdot|)$, one sees that $ M_{v,s,i}=0$ if one of (1)(a), (1)(b) and (1)(c) holds.

Since the computations for (2) and (3) are similar, we will only prove (3). Assume that both $\theta_1$ and $\theta_2$ are ramified.  We will compute $M_{v,s,e_1}$ by computing two integrals. The first integral is 
$$
\int_{\O_v}f_{\theta,s,v}\left(\left(\begin{array}{cc}
0 & -1 \\1 & 0
\end{array}\right)
\left(\begin{array}{cc}
1 & n \\0 & 1
\end{array}\right)
\left(\begin{array}{cc}
1 & 0 \\\ \varpi^{e_1} & 1
\end{array}\right)\right) dn
$$
which is zero by (\ref{eq:414}). The second integral is
$$
\int_{F_v-\O_v}f_{\theta,s,v}\left(\left(\begin{array}{cc}
0 & -1 \\1 & 0
\end{array}\right)
\left(\begin{array}{cc}
1 & n \\0 & 1
\end{array}\right)
\left(\begin{array}{cc}
1 & 0 \\\ \varpi^{e_1} & 1
\end{array}\right)\right) dn.
$$
It follows from (\ref{eq:415}) that the integral equals
\begin{equation}\label{eq:416}
\int_{F_v-\O_v} \theta_1^{-1}\theta_2(n)|n|^{-(2s+1)} f_{\theta,s,v}\left(\left(\begin{array}{cc}1 & 0 \\ (1+n\varpi^{e_1})n^{-1} & 1 \end{array}\right)\right) dn.
\end{equation}
Note that under our assumption of $\theta_1$ and $\theta_2$, we know from Lemma~\ref{412} that
\begin{equation}\label{eq:417}
f_{\theta,s,v}\left(\left(\begin{array}{cc}1 & 0 \\ (1+n\varpi^{e_1})n^{-1} & 1 \end{array}\right)\right)\neq 0 \mbox{ if and only if } \val((1+n\varpi^{e_1})n^{-1})=e_{\mfr{p,2}}.
\end{equation}
When $e_1>e_2$, it follows from (\ref{eq:417}) that (\ref{eq:416}) equals 
\[
\begin{split}
& \int_{\varpi^{-e_2}\O_v^{\times}} \theta_1^{-1}\theta_2(n)|n|^{-(2s+1)} f_{\theta,s,v}\left(\left(\begin{array}{cc}1 & 0 \\ (1+n\varpi^{e_1})n^{-1} & 1 \end{array}\right)\right) dn \\
=& \int_{\varpi^{-e_2}\O_{v}^{\times}} \theta_1^{-1}\theta_2(n)|n|^{-(2s+1)} \theta_1^{-1}((1+n\varpi^{e_1})n^{-1}) dn \\
=&\  \theta_2(\varpi^{-e_2}) |\varpi^{e_2}|_v^{2s}  \int_{\O_{v}^{\times}} \theta_1^{-1}(1+\varpi^{e_1-e_2}x) \theta_2(x) dx.
\end{split}
\]
By Lemma~\ref{253}, we have
\[
\begin{split}
&\theta_2(\varpi^{-e_2})|\varpi^{e_2}|_v^{2s}  \int_{\O_{v}^{\times}} \theta_1^{-1}(1+\varpi^{e_1-e_2}x) \theta_2(x) dx\\
=&\  \theta_2(\varpi^{-e_2}) |\varpi^{e_2}|_v^{2s+2} \sum_{x\in U^{e_2}} \theta_2(x)\theta_1^{-1}(1+\varpi^{e_1-e_2}+x)\\
=&\  \theta_2(-\varpi^{-e_2}) |\varpi^{e_2}|_v^{2s+2} \sum_{x\in U^{e_2}} \theta_2(x)\theta_1^{-1}(1-\varpi^{e_1-e_2}+x).
\end{split}
\]
This proves the assertion (3.i).

When $e_2 > e_1$, it follows from (\ref{eq:417}) that (\ref{eq:416}) equals
\[
\begin{split}
& \int_{\varpi^{-e_1}(-1+\varpi^{e_2-e_1}\O_{v}^{\times})}  \theta_1^{-1}\theta_2(n)|n|^{-(2s+1)} f_{\theta,s,v}\left(\left(\begin{array}{cc}1 & 0 \\ (1+n\varpi^{e_1})n^{-1} & 1 \end{array}\right)\right) dn\\
=&\  \int_{\varpi^{-e_1}(-1+\varpi^{e_2-e_1}\O_{v}^{\times})}  \theta_1^{-1}\theta_2(n)|n|^{-(2s+1)} \theta_1^{-1}((1+n\varpi^{e_1})n^{-1}) dn\\
=&\  \theta_1^{-1}(\varpi^{e_2-e_1}) \theta_2(\varpi^{-e_1}) |\varpi^{e_1}|_v^{2s} |\varpi^{e_2-e_1}|_v \int_{\O_{v}^{\times}} \theta_2(-1+\varpi^{e_2-e_1}x) \theta_1^{-1}(x) dx\\
=&\  \theta_1^{-1}(\varpi^{e_2-e_1}) \theta_2(\varpi^{-e_1}) |\varpi|_v^{2se_1+e_2} \sum_{x\in U_{\mfr{p}}^{(e_1)}} \theta_1^{-1}(x) \theta_2(-1+\varpi^{e_2-{e_1}}x).
\end{split}
\]
Note that the first equality is obtained by Lemma~\ref{412} and the last equality is obtained by Lemma~\ref{253}. This proves the assertion (3.ii).

Now we assume that $e_1=e_2$. It follows from (\ref{eq:417}) that (\ref{eq:416}) equals 
$$
\sum_{k=e_1}^{\infty} \int_{\varpi^{-k}\O_{v}^{\times}} \theta_1^{-1}\theta_2(n)|n|^{-(2s+1)} f_{\theta,s,v}\left(\left(\begin{array}{cc}1 & 0 \\ (1+n\varpi^{e_1})n^{-1} & 1 \end{array}\right)\right) dn .
$$
We will compute this integral by separating it into two parts. One part is when $k=e_1$ and the other one is when $k\geq e_1+1$. When $k=e_1$, we have
\[
\begin{split}
& \int_{\varpi^{-e_1}\O_{v}^{\times}} \theta_1^{-1}\theta_2(n)|n|^{-(2s+1)} f_{\theta,s,v}\left(\left(\begin{array}{cc}1 & 0 \\ (1+n\varpi^{e_1})n^{-1} & 1 \end{array}\right)\right) dn \\
=& \sum_{\substack{j\in (\O_{v}/\varpi)^{\times} \\ j\neq -1}} \int_{\varpi^{-e_1}(j+\varpi\O_v} \theta_1^{-1}(1+n\varpi^{e_1}) \theta_2(n) |n|^{-(2s+1)} dn\\
=&\  \theta_2(\varpi^{-e_1}) |\varpi|_v^{2se_1} \int_{\O_{v}^{\times}-\{-1+\varpi\O_{v}\}} \theta_1^{-1}(1+x) \theta_2(x)  dx.
\end{split}
\]
By Lemma~\ref{253}, we know that 
\[
\begin{split}
&\theta_2(\varpi^{-e_1}) |\varpi|_v^{2se_1} \int_{\O_{v}^{\times}-\{-1+\varpi \O_v \}} \theta_1^{-1}(1+x) \theta_2(x)  dx\\
=&\  \theta_2(\varpi^{-e_1}) |\varpi|_v^{2se_1+e_1} \sum_{\substack{x\in U_{\mfr{p}}^{(e_1)} \\ x\notin -1+\varpi \O_{v}}} \theta_2(x) \theta_1^{-1}(1+x)\\
=&\  \theta_2(-\varpi^{-e_1}) |\varpi|_v^{e_1(2s+1)} J_1(\theta_2,\theta_1^{-1},\varpi^{e_1}).
\end{split}
\]
When $k\geq e_1+1$, we have
\[
\begin{split}
& \sum_{k=e_{\mfr{p},1+1}}^{\infty} \int_{\varpi^{-k}\O_{v}^{\times}} \theta_1^{-1}\theta_2(n)|n|^{-(2s+1)} f_{\theta,s,v}\left(\left(\begin{array}{cc}1 & 0 \\ (1+n\varpi^{e_1})n^{-1} & 1 \end{array}\right)\right) dn \\
=& \sum_{k=e_1+1}^{\infty} \int_{\varpi^{-k}\O_{v}^{\times}} \theta_1^{-1}(1+\varpi^{e_1}n)\theta_2(n)|n|^{-(2s+1)} dn \\
=& \sum_{k=e_1+1}^{\infty} \theta_2(\varpi^{-k}) |\varpi|_v^{2sk} \int_{\O_{v}^{\times}} \theta_1^{-1}(1+\varpi^{e_1-k}x)\theta_2(x) dx \\
=& \sum_{k=e_1+1}^{\infty} \theta_1^{-1}(\varpi^{e_1-k})\theta_2(\varpi^{-k}) |\varpi|_v^{2sk} \int_{\O_{v}^{\times}} \theta_1^{-1}(\varpi^{k-e_1}+x)\theta_2(x) dx. 
\end{split}
\]
By Lemma~\ref{253}, we have
\[
\begin{split}
&\sum_{k=e_1+1}^{\infty} \theta_1^{-1}(\varpi^{e_1-k})\theta_2(\varpi^{-k}) |\varpi|_v^{2sk} \int_{\O_{v}^{\times}} \theta_1^{-1}(\varpi^{k-e_1}+x)\theta_2(x) dx\\
=& \sum_{k=e_1+1}^{\infty} \theta_1^{-1}(\varpi^{e_1-k})\theta_2(\varpi^{-k}) |\varpi|_v^{2sk+e_1} \sum_{x\in U_{\mfr{p}}^{(e_1)}} \theta_2(x) \theta_1^{-1}(\varpi^{k-e_1}+x)  \\
=& \sum_{k=e_1+1}^{2e_1-1} \theta_1^{-1}(\varpi^{e_1-k})\theta_2(\varpi^{-k}) |\varpi|_v^{2sk+e_1} \sum_{x\in U_{\mfr{p}}^{(e_1)}} \theta_2(x) \theta_1^{-1}(\varpi^{k-e_1}+x)  \\
&+ \sum_{k=2e_1}^{\infty} \theta_1^{-1}(\varpi^{e_1-k})\theta_2(\varpi^{-k}) |\varpi|_v^{2sk+e_1} \sum_{x\in U_{\mfr{p}}^{(e_1)}} \theta_2(x) \theta_1^{-1}(x).
\end{split}
\]
To sum up, we have shown that if $e_1=e_2$, then 
\[
\begin{split}
& \int_{F_v-\O_{v}} \theta_1^{-1}\theta_2(n)|n|^{-(2s+1)} f_{\theta,s,v}\left(\left(\begin{array}{cc} 1 & 0 \\ (1+n\varpi^{e_1})n^{-1} & 1 \end{array}\right)\right) dn \\
=&\  \theta_2(-\varpi^{-e_1}) |\varpi|_v^{e_1(2s+1)}J_1(\theta_2, \theta_1, \varpi^{e_1})\\
&+ \sum_{k=e_1+1}^{2e_1-1} \theta_1^{-1}(\varpi^{e_1-k})\theta_2(\varpi^{-k}) |\varpi|_v^{2sk+e_1} \sum_{x\in U_{\mfr{p}}^{(e_1)}} \theta_2(x) \theta_1^{-1}(\varpi^{k-e_1}+x)  \\
&+ \sum_{k=2e_1}^{\infty} \theta_1^{-1}(\varpi^{e_1-k})\theta_2(\varpi^{-k}) |\varpi|_v^{2sk+e_1} \sum_{x\in U_{\mfr{p}}^{(e_1)}} \theta_2(x) \theta_1^{-1}(x).
\end{split}
\]
If $\theta_1^{-1}\theta_2$ is a primitive character, then by \cite[Lemma 2.3]{Jun}, we see that 
$$
\sum_{x\in U_{\mfr{p}}^{(e_1)}} \theta_2(x) \theta_1^{-1}(\varpi^{k-e_1}+x)=0
$$
for all $k\geq e_1$ and
$$
\sum_{x\in U_{\mfr{p}}^{(e_1)}} \theta_2(x) \theta_1^{-1}(x)=0.
$$
Thus the assertion (3.iii) follows. Similarly, the assertions (3.iv) and (3.v) follow from Lemma 2.3 and Lemma 2.4 in \cite{Jun}. This completes the proof.
\end{proof}
%%%%%%%%%%%%%%%%%%%%%%%%%%%%%%%%%%%%%%%%%%%%%%%%%%%%%%%%%%%%%%%%%%%%%%%%%%%%%%%%%%%%%%%%%%%%%%%%%%%%%%%%
\subsection{Archimedean sections}\label{sec:33}
Let $k$ be a positive integer. This integer will correspond to the weight of Eisenstein series in the next subsection. For $i=1,\ldots,d$, define the archimedean section $f_{\chi,s,k,\tau_i}:\GL_2(\R)\map \C^{\times}$ as
$$
f_{\chi,s,k,\tau_i}\left(\left(\begin{array}{cc}
    a_1 & b \\
    0   & a_2
\end{array}\right)\kappa_{\theta}\right)=\chi_{1,\tau_i}(a_1)\chi_{2,\tau_i}(a_2)\left|\frac{a_1}{a_2}\right|^{s+\tfrac{1}{2}}\cdot j(\kappa_{\theta},i)^{-k}
$$
for $a_1,a_2\in \R$ and $\kappa_{\theta}\in \SO_2(\R)$. Here $j(\kappa_{\theta}, i)$ is the automorphic factor defined in Section~\ref{sec:21}. For $\beta\in F$ and $g\in \GL_2(\R)$, the integral
$$
\int_{\R} f_{\chi,s,k,\tau_i}\left(\left(\begin{array}{cc}
    0 & -1 \\1 & 0
\end{array}\right)\left(\begin{array}{cc}
    1 & n \\0 & 1
\end{array}\right)g\right) \psi_{\infty}(-\beta n)dn
$$
converges absolutely when $\re(s)$ is big enough and has analytic continuation for all $s$ (see the proof of \cite[Theorem 3.7.1]{B}. 

Let $g_z=\left(\begin{smallmatrix} y & x \\ 0 & 1\end{smallmatrix}\right)$ for $z=x+iy\in \mathbf{H}$.
Since the measure $dn$ is additive, by the Iwasawa decomposition for $\GL_2(\R)$, we have
$$
\int_{\R} f_{\chi,s,k,\tau_i}\left(\left(\begin{array}{cc}
    0 & -1 \\1 & 0
\end{array}\right)\left(\begin{array}{cc}
    1 & n \\0 & 1
\end{array}\right)g_z\right) \psi_{\infty}(-\beta n) dn=\int_{\R} f_{\chi,s,k,\tau_i}\left(\left(\begin{array}{cc}
    0 & -1 \\y & n
\end{array}\right)\right) \psi_{\infty}(-\beta (n-x)) dn,
$$
which, by the definition of $f_{\chi,s,k,\tau_i}$, equals 
$$
\psi_{\infty}(\beta x)y^{s+\frac{1}{2}}\int_{\R} (n+iy)^{-k} |n+iy|^{-2(s-\frac{k-1}{2})} \psi_{\infty}(-\beta n) dn.
$$
From the discussion in \cite[\S 9.2]{Hida2}, one further has
\begin{equation}\label{arch-sec}
\left[\int_{\R} f_{\chi,s,k,\tau_i}\left(\left(\begin{array}{cc}
    0 & -1 \\1 & 0
\end{array}\right)\left(\begin{array}{cc}
    1 & n \\0 & 1
\end{array}\right)g_z\right) \psi_{\infty}(-\beta n)dn\right]_{s=\tfrac{1-k}{2}}\\
=\begin{cases}
\tfrac{1}{2}C_{\infty}(k)& \mbox{ if }\beta=0\\
C_{\infty}(k) e^{2\pi i \beta z}& \mbox{ if }\beta>0\\
0 & \mbox{ if }\beta<0,
\end{cases}
\end{equation}
where $C_{\infty}(k)=i^{-k} 2^k \pi y^{k/2}$.
%%%%%%%%%%%%%%%%%%%%%%%%%%%%%%%%%%%%%%%%%%%%%%%%%%%%%%%%%%%%%%%%%%%%%%%%%%%%%%%%%%%%%%%%%%%%%%%%%%%%%%%%%
\subsection{More on Eisenstein series}\label{sec:34}
Recall that we fix a pair of characters $\chi=(\chi_1,\chi_2)$ at the beginning of this section satisfying (\ref{eq:eisen_cond}). For simplicity, we will write $\mfr{n}=\mfr{n}_1\mfr{n}_2$, where $\mfr{n}_1$ and $\mfr{n}_2$ are respectively the conductors of $\chi_1$ and $\chi_2$. Following (7.8) in \cite[\S3.7]{B}, \textit{the Eisenstein series associated to the section $f_{\chi,s}=\bigotimes_{i=1}^d f_{\chi,s,k,\tau_i} \otimes\bigotimes_{v<\infty} f_{\chi,s,v}$} is defined by
$$
E(f_{\chi,s},g)=\sum_{\gamma\in B(F)\backslash \GL_2(F)} f_{\chi,s}(\gamma g)
$$
for all $g=(g_v)_v\in \GL_2(\A_F)$, which converges absolutely when $\re(s)>1/2$. Recall that for a narrow ray class character $\psi$, the $L$-function $L(s,\psi)$ was defined by (\ref{eq:L_function}). The partial $L$-function $L^{\mfr{n}}(1-k,\chi_1\chi_2^{-1})$ is defined as 
\begin{equation}\label{eq:partial_L_function}
L^{\mfr{n}}(1-k,\chi_1\chi_2^{-1})= L(1-k,\chi_1\chi_2^{-1})\cdot \prod_{\substack{\mfr{q}|\mfr{n},\\ \mfr{q}\nmid \cond(\chi_1\chi_2^{-1})}} (1-\chi_1^{-1}\chi_2(\mfr{q})N(\mfr{q})^{-k}).
\end{equation}
The normalized Eisenstein series $L^{\mfr{n}}(2s+1,\chi_1\chi_2^{-1})E(f_{\chi,s},g)$ has meromorphic continuation to all $s$ except that it has a pole at $s=\frac{1}{2}$ if $\chi_{1}=\chi_2$ (see Theorem~3.7.1 in \textit{loc.~cit.}). The \textit{adelic Eisenstein series $E_k(\chi_1,\chi_2)(z,g)\in M_k(\mfr{n},\chi_1\chi_2;\C)$} is defined by
$$
E_k(\chi_1,\chi_2)(z,g):=\left.\frac{C_{\infty}(k)^{-d}L^{\mfr{n}}(2s+1,\chi_1\chi_2^{-1})}{\prod_{v|\mfr{n}_2}\varepsilon_v(2s+1,\chi_2^{-1},\psi_v)}E(f_{\chi,s},g_z g)\right|_{s=\tfrac{1-k}{2}}
$$
for $z=(x_j+iy_j)_j\in \mathbf{H}^d$, $g_z=\left(\left(\begin{smallmatrix} y_j & x_j \\ 0 & 1\end{smallmatrix}\right)\right)_j\in (\GL_2(\R))^d$ and $g\in \GL_2(\bbA_{F,f})$, where $C_{\infty}(k)$ was defined in Section~\ref{sec:33}. 

The Fourier expansion of $E(f_{\chi,s},g_zg)$ (see (7.11) in \textit{loc.~cit.}) is given by
$$
E(f_{\chi,s},g_zg)=\sum_{\beta\in F} c_{\beta}(E(f_{\chi,s},g_zg)),
$$
where $c_{\beta}(E(f_{\chi,s}(g_z\gamma)))$ is defined as
\begin{equation}\label{eq:20}
c_{\beta}(E(f_{\chi,s},g_zg)):=\int_{F\backslash \A_F}E\left(f_{\chi,s},\left(\begin{array}{cc}
    1 & n \\0 & 1
\end{array}\right)g_zg\right) \psi(-\beta n)dn.
\end{equation}
Here $dn=\otimes_v dn_p$ is the self-dual Haar measure defined at the beginning of this section. The number $c_0(E(f_{\chi,s},g_zg))$ is called the constant term of $E(f_{\chi,s},g_zg)$ at the cusp associated to $g$. By (7.14) and (7.15) in \textit{loc.~cit.}, one has  
$$
c_{0}(E(f_{\chi,s},g_zg))=f_{\chi,s}(g_zg)+\int_{\A_F}f_{\chi,s}\left(\left(\begin{array}{cc}
    0 & -1 \\1 & 0
\end{array}\right)\left(\begin{array}{cc}
    1 & n \\0 & 1
\end{array}\right)g_zg\right) dn
$$
and 
$$
c_{\beta}(E(f_{\chi,s},g_zg))=\int_{\A_F}f_{\chi,s}\left(\left(\begin{array}{cc}
    0 & -1 \\1 & 0
\end{array}\right)\left(\begin{array}{cc}
    1 & n \\0 & 1
\end{array}\right)g_zg\right) \psi(-\beta n) dn
$$
for $\beta\in F^{\times}$. It follows from (\ref{arch-sec}) that the integral is $0$ if $\beta$ is not totally positive. Moreover, it was proved by Tate \cite[Ch.~XV, \S 3.3]{Cas-Fro} (or see (7.18) in \textit{loc.~cit.}) that one has
\[
\begin{split}
&\int_{\A_F}f_{\chi,s}\left(\left(\begin{array}{cc}
    0 & -1 \\1 & 0
\end{array}\right)\left(\begin{array}{cc}
    1 & n \\0 & 1
\end{array}\right)g_z g\right) \psi(-\beta n) dn\\
=& \prod_{v<\infty} \int_{F_{v}}f_{\chi,s,v}\left(\left(\begin{array}{cc}
    0 & -1 \\1 & 0
\end{array}\right)\left(\begin{array}{cc}
    1 & n_v \\0 & 1
\end{array}\right)g_v\right) \psi_v(-\beta n_v) dn_v \;\times \\
& \prod_{i=1}^d \int_{\R}f_{\chi,s,k,\tau_i}\left(\left(\begin{array}{cc}
    0 & -1 \\1 & 0
\end{array}\right)\left(\begin{array}{cc}
    1 & n_i \\0 & 1
\end{array}\right)g_z\right) \psi_{\infty}(-\beta n_i) dn_i.
\end{split}
\]
Here $v$ runs through all places of $F$. Again by (\ref{arch-sec}), one obtains
$$
\prod_{i=1}^d\int_{\R}f_{\chi,s,k,\tau_i}\left(\left(\begin{array}{cc}
    0 & -1 \\1 & 0
\end{array}\right)\left(\begin{array}{cc}
    1 & n_i \\0 & 1
\end{array}\right)g_z\right) \psi_{\infty}(-\beta n_i)  dn_i= C_{\infty}(k)^d e^{2\pi i \Tr(\beta z)}.
$$
Since the characters $\chi_1$ and $\chi_2$ satisfies the the condition (\ref{eq:eisen_cond}), one has
$$
L^{\mfr{n}}(2s+1,\chi_1\chi_2^{-1})f_{\chi,s}(g_zg)|_{\frac{1-k}{2}}=0.
$$
From the above discussion, one knows that the $\beta$-th Fourier coefficient $c_{\beta}(E_k(\chi_1,\chi_2)(z,g))$ of $E_k(\chi_1,\chi_2)(z,g)$ is given by 
\begin{equation}\label{eq:419}
\left[\frac{L^{\mfr{n}}(2s+1,\chi_1\chi_2^{-1})}{\prod_{v|\mfr{n}_2}\varepsilon_v(2s+1,\chi_2^{-1},\psi_v)}\prod_{v<\infty}\int_{F_{v}}f_{\chi,s,v}\left(\left(\begin{array}{cc}
    0 & -1 \\1 & 0
\end{array}\right)\left(\begin{array}{cc}
    1 & n_v \\0 & 1
\end{array}\right)g_v\right) \psi_v(-\beta n_v)  dn_v\right]_{s=\tfrac{1-k}{2}}
\end{equation}
and the constant term $c_0(E_k(\chi_1,\chi_2)(z,g))$ at the cusp associated to $g$ is given by
\begin{equation}\label{eq:420}
\left[\frac{L^{\mfr{n}}(2s+1,\chi_1\chi_2^{-1})}{2^d\prod_{v|\mfr{n}_2}\varepsilon_v(2s+1,\chi_2^{-1},\psi_v)}\left( \prod_{v<\infty} \int_{F_v}f_{\chi,s,v}\left(\left(\begin{array}{cc}
    0 & -1 \\1 & 0
\end{array}\right)\left(\begin{array}{cc}
    1 & n_v \\0 & 1
\end{array}\right)g_v\right) dn \right)\right]_{s=\tfrac{1-k}{2}}
\end{equation} 

In the proof of the following two propositions, we will apply results in Section~\ref{sec:32} by taking $\theta_i=\chi_{i,v}$ for each finite place $v$ of $F$ and $i=1,2$. Recall that for $\lambda=1,\ldots,h_F^+$, $t_{\lambda}$ was fixed in Section~\ref{sec:22} satisfying (\ref{eq:class_repr}) and the matrix $x_{\lambda}\in \GL_2(\mathbb{A}_{F,f})$ is defined by (\ref{eq:x_lambda}). Now we compute the Fourier coefficient at the cusp associated to $x_{\lambda}^{-1}$. Note that all of the Fourier coefficients in Proposition~\ref{431} are unnormalized. One can multiply $N(t_{\lambda}\mfr{D})^{-k/2}$ to obtain the normalized Fourier coefficients as in Section~\ref{sec:02}.

\begin{prop}\label{431}
Suppose that the infinity parts of $\chi_1$ and $\chi_2$ satisfy the condition (\ref{eq:eisen_cond}). Then we have
$$
c_0(E_k(\chi_1,\chi_2)(z, x_{\lambda}^{-1}))=
\begin{cases}
2^{-d}N(t_{\lambda}\mfr{D})^{k/2}\chi_2^{-1}(t_{\lambda}\mfr{D})L(1-k,\chi_1\chi_2^{-1})& \mbox{if } \mfr{n}_2=1,\\
0&\mbox{otherwise}
\end{cases}
$$
for $\lambda=1,\ldots, h^+$. Moreover, for any integral ideal $\mfr{m}$ of $F$, we have
\begin{equation}\label{eq:421}
c_{\beta}(E_k(\chi_1,\chi_2)(z,x_{\lambda}^{-1}))=N(t_{\lambda}\mfr{D})^{k/2}\sum_{\mfr{a}|\mfr{m}} \chi_1(\mfr{a}) \chi_2(\tfrac{\mfr{m}}{\mfr{a}}) N(\mfr{a})^{k-1},
\end{equation}
where $\beta\in {F^+}$ satisfies $\mfr{m}=(t_{\lambda}\delta)^{-1}\beta\O_F$. 
In particular, $E_k(\chi_1,\chi_2)(z,g)$ is the same as the Eisenstein series $E_k(\chi_1,\chi_2)=(E_{\lambda})_{\lambda=1}^{h_F^+}$ in Proposition~\ref{333}.
\end{prop}

\begin{proof}
We first compute $c_0(E_k(\chi_1,\chi_2)(z,x_{\lambda}^{-1}))$. If $\mfr{n}_2\neq 1$, it is $0$ by Lemma~\ref{413}(3). Now we suppose that $\mfr{n}_2=1$. By Lemma~\ref{412}, Lemma~\ref{413}, and the uniqueness of meromorphic continuation of $L$-functions, we obtain 
\[
\begin{split}
&\left[2^{-d}L(2s+1,\chi_1\chi_2^{-1})\prod_{v<\infty} 
\int_{F_{v}}f_{\chi,s,v}\left(\left(\begin{array}{cc}
    0 & -1 \\1 & 0
\end{array}\right)\left(\begin{array}{cc}
    1 & n_v \\0 & 1
\end{array}\right)x_{\lambda}^{-1}\right)  dn_v\right]_{s=\frac{1-k}{2}}\\
=& \left[ 2^{-d}\left( \prod_{v<\infty} \chi_{2,v}(t_{\lambda}\delta) N_{F_v/\Qp}(t_{\lambda}\delta)^{k/2}\right) L(2s,\chi_1\chi_2^{-1})\right]_{s=\tfrac{1-k}{2}}\\
=&\  2^{-d}\chi_2(t_{\lambda}\mfr{D}) N(t_{\lambda}\mfr{D})^{k/2} L(1-k,\chi_1\chi_2^{-1}).
\end{split}
\]
Thus, we see that the constant term of $E_k(\chi_1,\chi_2)(z,x_{\lambda}^{-1})$ coincides with $C_{\lambda}(E_k(\chi_1,\chi_2))$ in Proposition~\ref{333}.

Next, we claim that $E_k(\chi_1,\chi_2)(z,g)$ coincides with the Eisenstein series in Proposition~\ref{333}. To do so, it remains to show
$$
c_{\beta}(E_k(\chi_1,\chi_2)(z,x_{\lambda}^{-1}))=N(t_{\lambda}\mfr{D})^{k/2}\sum_{\mfr{a}|\mfr{p}^{\alpha}} \chi_1(\mfr{a}) \chi_2(\tfrac{\mfr{p}^{\alpha}}{\mfr{a}}) N(\mfr{a})^{k-1},
$$
where $\beta\in F^+$ satisfies $\mfr{p}^{\alpha}=(t_{\lambda}\delta)^{-1} \beta\O_F$
for all prime ideals $\mfr{p}$ and $\alpha\in \Z_{>0}$, since Eisenstein series are eigenforms and the Hecke algebra is generated by $T(\mfr{p}^{\alpha})$ for all prime ideals $\mfr{p}$ and for all positive integers $\alpha$. Note that once we prove the claim, we obtain (\ref{eq:421}) by Proposition~\ref{333}. 

We will only compute the case $\mfr{p}|\mfr{n}_1\mfr{n}_2$ as the computation is similar when $\mfr{p}\nmid \mfr{n}_1\mfr{n}_2$. We first observe that for $\mfr{p}|\mfr{n}_1\mfr{n}_2$, we have
\begin{equation}\label{eq:422}
\val_v(\beta)=
\begin{cases}
\alpha & \mbox{if } v=\mfr{p} \\
\val_v(t_{\lambda}\delta) & \mbox{if } v\nmid \mfr{n}_1\mfr{n}_2\\
0 & \mbox{otherwise}
\end{cases}
\end{equation}
because $\mfr{p}^{\alpha}=(t_{\lambda}\delta)^{-1} \beta\O_F$ and $t_{\lambda}\delta$ is prime to $\mfr{n}_1\mfr{n}_2$ by (\ref{eq:class_repr}). In the following computation, we will denote by $w$ the finite place corresponding to $\mfr{p}$.

If $\mfr{p}|\mfr{n}_1$ and $\mfr{p}|\mfr{n}_2$, then for the finite place $w$, we have $\val_{w}(\beta)> 0$ by (\ref{eq:422}). It follows from Lemma~\ref{414}(4) that 
$$
\int_{F_w} f_{\chi,s,w}\left(\left(\begin{array}{cc}
    0 & -1 \\1 & 0
\end{array}\right)\left(\begin{array}{cc}
    1 & n \\0 & 1
\end{array}\right)\right) \psi_w(-\beta n) dn=0,
$$
and hence, $c_{\beta}(E_k(\chi_1,\chi_2)(z,x_{\lambda}^{-1}))=0$ by (\ref{eq:419}). This proves the assertion since $\chi_1(\mfr{p})=\chi_2(\mfr{p})=0$ in this case. 

If $\mfr{p}\nmid \mfr{n}_1$ and $\mfr{p}|\mfr{n}_2$, 
By Lemma~\ref{414} and (\ref{eq:422}), we have
\[
\begin{split}
& \prod_{v<\infty} \int_{F_v}f_{\chi,s,v}
\left(\left(\begin{array}{cc}
    0 & -1 \\ 1 & 0
\end{array}\right)
\left(\begin{array}{cc}
    1 & n_v \\0 & 1
\end{array}\right)x_{\lambda}^{-1}\right) \psi_v(-\beta n_v) dn_v\\
=& \prod_{v|\mfr{n}_1,v|\mfr{n}_2} \chi_{2,v}^{-1}(\beta) \varepsilon_v(2s+1,\chi_{2,v}^{-1},\psi_v) \times 
 \prod_{v\nmid \mfr{n}_1, v|\mfr{n}_2} \chi_{1,v}\chi_{2,v}^{-1}(\beta) |\beta|_v^{2s} \varepsilon_v(2s+1,\chi_{2,v}^{-1},\psi_v) \times \prod_{v|\mfr{n}_1,v\nmid \mfr{n}_2} 1 \times \\
& \prod_{v\nmid \mfr{n}_1\mfr{n}_2} \chi_{2,v}^{-1}(t_{\lambda}\delta) |t_{\lambda}\delta|_v^{s-\tfrac{1}{2}} (1-\chi_{1,v}^{-1}\chi_{2,v}(\varpi_v) q_v^{-(2s+1)})\\
=&\ L^{\mfr{n}}(2s+1,\chi_1\chi_2^{-1})^{-1} \times \prod_{v|\mfr{n}_2} \varepsilon_v(2s+1,\chi_{2,v}^{-1},\psi_v) \times 
\prod_{v|\mfr{n}_1,v|\mfr{n}_2} \chi_{2,v}^{-1}(\beta) \times
\prod_{v\nmid \mfr{n}_1,v|\mfr{n}_2} \chi_{1,v}\chi_{2,v}^{-1}(\beta) |\beta|_v^{2s} \times\\ 
&\prod_{v\nmid \mfr{n}_1\mfr{n}_2} \chi_{2,v}^{-1}(t_{\lambda}\delta)|t_{\lambda}\delta|_v^{s-\tfrac{1}{2}}, 
\end{split}
\]
which, by (\ref{eq:419}), implies that we have 
\[
\begin{split}
c_{\beta}(E_k(\chi_1,\chi_2)(z,x_{\lambda}^{-1}))
&= \left[\prod_{v|\mfr{n}_2} \chi_{2,v}^{-1}(\beta) \times  \chi_{1,w}(-\beta) |\beta|_w^{2s} \times \prod_{v\nmid \mfr{n}_1\mfr{n}_2} |t_{\lambda}\delta|_v^{s-\tfrac{1}{2}} \chi_{2,v}^{-1}(t_{\lambda}\delta)\right]_{s=\tfrac{1-k}{2}}\\
&= \prod_{v\nmid \mfr{n}_1\mfr{n}_2} \chi_{2,v}^{-1}(\beta^{-1} t_{\lambda}\delta) \times N(t_{\lambda}\mfr{D})^{k/2}\chi_1(\mfr{p}^{\alpha}) N(\mfr{p}^{\alpha})^{k-1}\\
&= N(t_{\lambda}\mfr{D})^{k/2} \chi_1(\mfr{p}^{\alpha}) N(\mfr{p}^{\alpha})^{k-1}.
\end{split}
\]
Note that the last equality is obtained by the assumption that $\mfr{p}\nmid \mfr{n}_1$ and $\mfr{p}|\mfr{n}_2$. 

If $\mfr{p}|\mfr{n}_1$ and $\mfr{p}\nmid \mfr{n}_2$, then by a similar computation, one obtains
$$
c_{\beta}(E_k(\chi_1,\chi_2)(z,x_{\lambda}^{-1}))=N(t_{\lambda}\mfr{D})^{k/2} \chi_2(\mfr{p}^{\alpha}).\qedhere
$$
\end{proof}

Recall that we denote by $\chi_{i,f}$ the finite part of $\chi_i$ for $i=1,2$. We will write $\mfr{n}_i=\mfr{n}'_i\times \mfr{m}_i$ with $\mfr{n}'_i=\prod_{\mfr{p}\nmid \gcd(\mfr{n}_1,\mfr{n}_2)} \mfr{p}^{e_{\mfr{p},i}}$ and $\mfr{m}_i= \prod_{\mfr{p}|\gcd(\mfr{n}_1,\mfr{n}_2)} \mfr{p}^{e_{\mfr{p},i}}$ for $i=1,2$. 
We will compute the constant term  of $E_k(\chi_1,\chi_2)(z,x_{\lambda}^{-1}g)$ for some $g=(g_v)=\left(\left(\begin{smallmatrix}
a_v & b_v \\ c_v & d_v\end{smallmatrix}\right)\right)_v \in \GL_2(\wh{\O}_F)$. By right multiplying some element in $K_1(\mfr{n})$, we may assume $\det g=1$. Moreover, since $K_{1,v}(\mfr{n})=\GL_2(\O_v)$ for $v\nmid \mfr{n}_1\mfr{n}_2$, we may assume further that $g_v$ is the identity matrix for $v\nmid \mfr{n}_1\mfr{n}_2$.

\begin{prop}\label{432}
Let the notation be as above, and let the assumptions be as in Proposition~\ref{431}. We set $c=(c_v)_v$ and $d=(d_v)_v$. Also, we set $c_{\mfr{n}_2}$ (resp.~$d_{\mfr{n}_1}$) be the $\mfr{n}_2$-part of $c$ (resp.~$\mfr{n}_1$-part of $d$). Assume further that the following conditions hold
\begin{enumerate}
\item $\val_v(c_v)\geq e_{\mfr{p},1}$ for all $v|\mfr{n}'_1$,

\item $\val_v(c_v)=0$ for all $v|\mfr{n}'_2$,

\item $\val_v(c_v)=e_{\mfr{p},1}$ for all $v|\gcd(\mfr{n}_1,\mfr{n}_2)$.
\end{enumerate}
Then the constant term $c_0(E_k(\chi_1,\chi_2)(z,x_{\lambda}^{-1}g))$ is 
\[
\begin{split}
&  \frac{1}{2^d} L^{\mfr{n}}(1-k,\chi_1\chi_2^{-1}) \chi_1^{-1}(\mfr{n}'_2) N(\mfr{n}'_2)^{-1}\times
\prod_{v|\gcd(\mfr{n}_1,\mfr{n}_2)} M_{v,\frac{1-k}{2},e_{\mfr{p},1}} \times 
\prod_{v|\mfr{n}_2} \varepsilon_v(2-k,\chi_{2,v}^{-1},\psi_v)^{-1}\times \\
& \chi_2^{-1}(t_{\lambda}\mfr{D}) N(t_{\lambda}\mfr{D})^{k/2} \chi_{1,f}^{-1}(d_{\mfr{n}_1}) \chi_{2,f}(-c_{\mfr{n}_2}\mfr{m}_1^{-1})
\end{split}
\]
where $M_{v,s,e_{\mfr{p},1}}$ was defined in Lemma~\ref{415}. Here $\chi_{1,f}(d_{\mfr{n}_1})$ and $ \chi_{2,f}(c_{\mfr{n}_2}\mfr{m}_1^{-1})$ are respectively defined by the isomorphisms $
\prod_{v|\mfr{n}_i} \O_v^{\times}/(1+\varpi^{e_{\mfr{p},i}}\O_v)\isom (\O/\mfr{n}_i\O)^{\times}$ for $i=1,2$.

Otherwise, if any of the above conditions (1)-(3) does not hold, then the constant term is $0$. 
\end{prop}

\begin{proof}
Recall that $c_0(E_k(\chi_1,\chi_2)(z,x_{\lambda}^{-1}g))$ is given by (\ref{eq:420}). It follows from Lemma~\ref{415} that if one of the three conditions does not hold, then the integral $\int_{F_v} f_{\chi,s,v}\left(\left(\begin{smallmatrix}
    0 & -1 \\1 & 0
\end{smallmatrix}\right)\left(\begin{smallmatrix}
    1 & n_v \\0 & 1
\end{smallmatrix}\right)\gamma\right) dn_v$ is $0$ for some $v|\mfr{n}_1\mfr{n}_2$. 

Now we assume all of the three conditions hold. For finite places $v\nmid \mfr{n}_1\mfr{n}_2$, by Lemma~\ref{413}(1), we have
$$
\left[ \prod_{v\nmid \mfr{n}_1\mfr{n}_2} \int_{F_{v}}f_{\chi,s,v}\left(\left(\begin{array}{cc}
    0 & -1 \\1 & 0
\end{array}\right)\left(\begin{array}{cc}
    1 & n_v \\0 & 1
\end{array}\right)x_{\lambda}^{-1}g_v\right)  dn_v\right]_{s=\frac{1-k}{2}}
=N(t_{\lambda}\mfr{D})^{k/2} \chi_2^{-1}(t_{\lambda}\mfr{D}) \frac{L^{\mfr{n}}(1-k,\chi_1\chi_2^{-1})}{L^{\mfr{n}}(2-k,\chi_1\chi_2^{-1})}.
$$
For finite places $v|\mfr{n}'_1$, by the above assumption (1) on $c_v$, we have
$$
\left[ \prod_{v|\mfr{n}_1'} \int_{F_{v}}f_{\chi,s,v}\left(\left(\begin{array}{cc}
    0 & -1 \\1 & 0
\end{array}\right)\left(\begin{array}{cc}
    1 & n_v \\0 & 1
\end{array}\right)g_v\right)  dn_v\right]_{s=\frac{1-k}{2}}
=\prod_{v|\mfr{n}_1'} \chi_{1,v}\chi_{2,v}(d_v)=\prod_{v|\mfr{n}_1'} \chi_{1,v}(d_v).
$$
For finite places $v|\mfr{n}'_2$, by Lemma~\ref{415}(2), we have
$$
\left[ \prod_{v|\mfr{n}_2'}\int_{F_{v}}f_{\chi,s,v}\left(\left(\begin{array}{cc}
    0 & -1 \\1 & 0
\end{array}\right)\left(\begin{array}{cc}
    1 & n \\0 & 1
\end{array}\right)g_v\right)  dn_v\right]_{s=\frac{1-k}{2}}
=\prod_{v|\mfr{n}_2'} \chi_{2,v}^{-1}(-c_v) \chi_{1,v}(\varpi^{-e_{\mfr{p},2}}) q_v^{-e_{\mfr{p},2}}
$$
which equals $\chi_1(\mfr{n}'_2)^{-1} N(\mfr{n}'_2)^{-1}\prod_{v|\mfr{n}'_2} \chi_{2,v}^{-1}(-c_v)$. 
For finite places $v|\gcd(\mfr{n}_1,\mfr{n}_2)$, since $\val_c(c_v)=e_{\mfr{p},1}$,  by (\ref{eq:23}), one can decompose $g_v$ as
$$
g_v= \left(\begin{array}{cc}
    a_v & b_v \\ c_v & d_v
\end{array}\right)=
\left(\begin{array}{cc}
    c_v^{-1}\varpi^{e_{\mfr{p},1}} & b_v \\ 0 & d_v
\end{array}\right)
\left(\begin{array}{cc}
    1 & 0 \\ \varpi^{e_{\mfr{p},1}} & 1
\end{array}\right)\kappa
$$
for some $\kappa\in K_{1,v}(\mfr{n})$. Then by Lemma~\ref{local_wittaker} and Lemma~\ref{415}(3), one has
\[
\begin{split}
\prod_{v|\gcd(\mfr{n}_1,\mfr{n}_2)} \int_{F_v} f_{\chi,s,v}\left(\left(\begin{array}{cc}
    0 & -1 \\1 & 0
\end{array}\right)\left(\begin{array}{cc}
    1 & n_v \\0 & 1
\end{array}\right)g_v\right)  dn_v
=  \prod_{v|\gcd(\mfr{n}_1,\mfr{n}_2)} \chi_{1,v}(d_v) \chi_{2,v}(c^{-1}_v\varpi^{e_{\mfr{p},1}}) M_{v,s,e_{\mfr{p},1}}.
\end{split}
\]
Finally, to finish the proof, we note that $\prod_{v|\mfr{n}_1} \chi_{1,v}(d_v)=\chi_{1,f}^{-1}(d_{\mfr{n}_1})$ and similarly, $\prod_{v|\mfr{n}_2} \chi_{2,f}(-c_v^{-1}\varpi^{e_{\mfr{p},2}})=\chi_{2,f}(-c_{\mfr{n}_2}\mfr{m}_1^{-1})$.
\end{proof}

%%%%%%%%%%%%%%%%%%%%%%%%%%%%%%%%%%%%%%%%%%%%%%%%%%%%%%%%%%%%%%%%%%%%%%%%%%%%%%%%%%%%%%%%%%%%%%%%%%%%%%%%%
\section{$\Lambda$-adic modular forms}\label{sec:04}
In this section, we first recall the definition of $\Lambda$-adic modular forms and $\Lambda$-adic Eisenstein series following \cite[\S 1.2]{W2}. The main goal is to compute the constant terms of $\Lambda$-adic Eisenstein series (Proposition~\ref{523}) using the results in Section~\ref{sec:34}. 

We now fix some notation that will be used throughout this section. Let $p$ be an odd prime unramified in $F$, and let $u=(1+p)\in \Zp^{\times}$. Let $\gamma$ be a topological generator of $\Gal(F_{\infty}/F)\isom \Zp$ such that $\gamma\cdot \zeta=\zeta^u$ for all $p$-power roots of unity $\zeta$. Let $\omega$ be the Teichm\"{u}ller character and $\<\cdot\>$ be the projection $\Zp^{\times}\surj 1+p\Zp$. Then, we have a canonical isomorphism
$$
\omega\otimes\<\cdot\>:\Zp^{\times}\isom (\Z/p\Z)^{\times}\times 1+p\Zp.
$$
Moreover, for $a\in \Zp^{\times}$, one can write $\<a\>=u^{s(a)}$ for some $s(a)\in \Zp$. For an integral ideal $\mfr{a}$ of $F$ prime to $p$, we set $s(\mfr{a}):=s(N(\mfr{a}))$. 
For simplicity, we put $\omega(\mfr{a})=\omega(N(\mfr{a}))$ for all ideals $\mfr{a}$ of $F$ prime to $p$.

For an integral ideal $\mfr{m}$, we denote by $I_{\mathfrak{m}}$ the set of fractional ideals of $F$ relatively prime to $\mathfrak{m}$. For a narrow ray class character $\chi$ with modulus $\n$ or $\mfr{n}p$, we associate a homomorphism,
\begin{equation}\label{eq:wt_chi}
\wt{\chi}:\varprojlim_r I_{\n p^r}\map \Zp[\chi][[T]];\; \mathfrak{a}\mapsto\chi(\mathfrak{a})(1+T)^{s(\mfr{a})}.
\end{equation}

Let $\mu_{p^{\infty}}$ be the group of all $p$-power roots of unity, and let $\wh{\mu}_{p^{\infty}}$ be the group of all characters of $\mu_{p^{\infty}}$ with values in $\C$. Let $\O_{\infty}\subset \C_p$ be a local complete valuation ring, whose valuation is compatible with the valuation of $\C_p$, containing $\mu_{p^{\infty}}$ and all values of narrow ray class characters with modulus $\mfr{n}p$ for a fixed integral ideal $\mfr{n}$ not divisible by $p$. We put $\Lambda=\O_{\infty}[[T]]$.
%%%%%%%%%%%%%%%%%%%%%%%%%%%%%%%%%%%%%%%%%%%%%%%%%%%%%%%%%%%%%%%%%%%%%%%%%%%%%%%%%%%%%%%%%%%%%%%%%%%%%%%%%
\subsection{$\Lambda$-adic modular forms}\label{sec:41}
 For each integer $k$ and for
each $\zeta \in \mu_{p^{\infty}}$, we define two evaluation maps
$$
v_{k,\zeta}, v'_{k,\zeta}:\Lambda\map \mathcal{O}_{\infty}
$$
by $v_{k,\zeta}(T)=\zeta u^{k-2}-1$ (resp.~$v'_{k,\zeta}(T)=\zeta u^{k}-1$). Note that for each $\zeta\in \mu_{p^{\infty}}$, there exists $\rho=\rho_{\zeta}\in \wh{\mu}_{p^{\infty}}$ such that $\rho(\gamma)=\zeta$. We will also write $v_{k,\zeta}$ (resp.~$v'_{k,\zeta}$) as $v_{k,\rho}$ (resp.~$v'_{k,\rho}$). For simplicity, we will write $\rho(\<N(\mfr{a})\>)$ as $\rho(\mfr{a})$ for all $\mfr{a}\in I_p$. Let $$\mathfrak{X}:=\{(k,\zeta) \mid k\geq 2,
\zeta^{p^r}=1\; \mbox{for some}\; r\geq 0\}
$$
be the set of classical weights.

\begin{defn}\label{511}\ 
\begin{enumerate}
\item Let $\n$ be an integral ideal of $\O_F$. A \textit{$\Lambda$-adic modular form $\mathcal{F}$ over $F$ of level $\n p$} is
      a set of elements of $\Lambda$
      $$
      \left\{\begin{array}{cc}
      C(\mathfrak{a},\mathcal{F}) & \mbox{for all nonzero integral ideals}\; \mathfrak{a}\; \mbox{of}\; \O_F \\
      C_{\lambda}(0,\mathcal{F}) & \mbox{for } \lambda=1,\ldots, h^+_F
      \end{array}\right \}
      $$
      with the property that for all but finitely many
      $(k,\zeta)\in \mathfrak{X}$, there is an adelic modular form $f$ of
      weight $k$ and level $\n p^r$ such that for each integral ideal $\mfr{a}$ of $F$, the normalized $\mfr{a}$-th Fourier coefficient satisfies $C(\mfr{a},f)=v_{k,\zeta}(C(\mathfrak{a},\mathcal{F}))$ and constant terms satisfy $C_{\lambda}(0,f)=v_{k,\zeta}(C_{\lambda}(0, \mathcal{F}))$ for all $\lambda=1,\ldots,h^+_F$.

\item A $\Lambda$-adic form is said to be \textit{a cusp form} if $v_{k,\zeta}(\mathcal{F})$ is a cusp form for almost all $(k,\zeta)\in \mfr{X}$.

\item Let $\chi$ be a narrow ray class character with modulus $\mfr{n}p$. We say that $\mathcal{F}$ is \textit{of character $\chi$} if $v_{k,\zeta}(\mathcal{F})$ has character $\chi\omega^{2-k}\rho_{\zeta}$ for almost all $(k,\zeta)\in \mfr{X}$.
\end{enumerate}
\end{defn}

We denote by $M(\n,\chi;\Lambda)$  and $S(\n,\chi;\Lambda)$ the space of $\Lambda$-adic modular forms
and the space of $\Lambda$-adic cusp forms of level $\n$ and
character $\chi$, respectively.

One can define $\Lambda$-adic modular forms and $\Lambda$-adic cusp forms with respect to the specialization $v'_{k,\zeta}$ in the same manner. Under this definition, we say that a $\Lambda$-adic modular form $\mathcal{F}$ is of character $\chi$ if $v'_{k,\zeta}(\mathcal{F})$ is of character $\chi\omega^{-k}\rho$. We denote by $M'(\n,\chi;\Lambda)$  and $S'(\n,\chi;\Lambda)$ the corresponding spaces of $\Lambda$-adic modular forms and $\Lambda$-adic cusp forms, respectively. 

The reason to mention different specializations is that both specializations are used in the literature. For example, in \cite{Hida2} and \cite{Hida3}, Hida used the specialization $v'_{k,\zeta}$, while Wiles used the specialization $v_{k,\zeta}$ in \cite{W2} and \cite{W3}. Indeed, these specializations are equivalent, which will be addressed in the following lemma.

\begin{lemma}\label{512}
We have a canonical isomorphism of $\Lambda$-modules
$$
M'(\n,\chi\omega^2;\Lambda)\isom M(\n,\chi;\Lambda);\; \mathcal{F}'(T) \mapsto \mathcal{F}'(u^{2}(1+T)-1).
$$
\end{lemma} 

\begin{proof}
The assertion is well-known. Given any $\mathcal{F}' \in M'(\mfr{n},\chi\omega^2;\Lambda)$, we set $\mathcal{F}(T)=\mathcal{F}'(u^{2}(1+T)-1)$. Then we have
$$
v_{k,\rho}(\mathcal{F}(T))=\mathcal{F}(\rho(u)u^{k-2}-1)=\mathcal{F}'(\rho(u) u^{k}-1)=v'_{k,\zeta}(\mathcal{F}'(T))
$$
which, by the definition of $M'(\mfr{n},\chi;\Lambda)$, is an adelic modular form of level $\mfr{n}p^r$ and character $\chi\omega^{2-k}\rho$ for almost all $(k,\zeta)\in \mfr{X}$. Therefore, $\mathcal{F}(T)$ is in  $M(\mfr{n},\chi;\Lambda)$, and clearly this provides a bijection.
\end{proof} 

\begin{remark}\label{513}
In this paper, we usually consider the space $M(\n,\chi;\Lambda)$. We will only use $M'(\n,\chi;\Lambda)$ in Section~\ref{sec:05}, where we will prove a control theorem for $M'(\n,\chi;\Lambda)$ for an arbitrary narrow ray class character $\chi$ with modulus $\n p$ and then deduce a control theorem for $M(\n,\chi;\Lambda)$ from Lemma~\ref{512}.
\end{remark}

In \cite[\S 1.2]{W2}, Wiles defined the Hecke actions on the space of $\Lambda$-adic modular forms, which commute with specialization map $v_{k,\zeta}$ and $v_{k,\zeta}'$. For details, we refer the reader to \textit{loc.~cit.}. Thus the Hida idempotent element $e$ acts on $M(\n,\chi;\Lambda)$ (resp.~$M'(\n,\chi;\Lambda)$) and preserves the subspace $S(\n,\chi;\Lambda)$ (resp.~$S'(\n,\chi;\Lambda)$). We define $M^{\ord}(\n,\chi;\Lambda)=e\cdot M(\n,\chi;\Lambda)$ and define $S^{\ord}(\n,\chi;\Lambda)$, ${M'}^{\ord}(\n,\chi;\Lambda)$, and ${S'}^{\ord}(\n,\chi;\Lambda)$  in the same manner. We denote by $\mathcal{H}^{\ord}(\mfr{n},\chi,\Lambda)\subset \End_{\Lambda}(M^{\ord}(\n,\chi;\Lambda))$ (resp.~$h^{\ord}(\mfr{n},\chi;\Lambda) \subset \End_{\Lambda}(S^{\ord}(\n,\chi;\Lambda))$ the Hecke algebra (resp.~cuspidal Hecke algebra) generated over $\Lambda$ by Hecke operators $T(\mfr{q})$, $S(\mfr{q})$ for all prime ideals $\mfr{q}$ not dividing $\mfr{n}p$ and $U(\mfr{p})$ for all prime ideals $\mfr{p}|p$.

Note that the isomorphism in Lemma~\ref{512} is Hecke-equivariant since specialization maps commute with Hecke operators and since for each $\mathcal{F}\in M(\n,\chi;\Lambda)$, if $\mathcal{F}'$ is the image of $\mathcal{F}$ under the isomorphism, we have $v_{k,\zeta}(\mathcal{F})=v'_{k,\zeta}(\mathcal{F}')$ for all $(k,\zeta)\in \mfr{X}$.
%%%%%%%%%%%%%%%%%%%%%%%%%%%%%%%%%%%%%%%%%%%%%%%%%%%%%%%%%%%%%%%%%%%%%%%%%%%%%%%%%%%%%%%%%%%%%%%%%%%%%%%
\subsection{$\Lambda$-adic Eisenstein series}\label{sec:42}
Eisenstein series provide interesting examples of $\Lambda$-adic modular forms. We recall their construction in this subsection. 

Let $\chi$ be an narrow ideal class character of conductor $\n$ or $\n p$. We assume that $\chi$ is even and is not of type $W$ in the sense of \cite{W3}, i.e., $F_{\chi}$ is not contained in $F_{\infty}$. Let $L_p(s,\chi)$ be the Deligne\textendash Ribet $p$-adic $L$-function (see \cite{DR} for the definition). It satisfies interpolation property
$$
L_p(1-k,\chi)=L(1-k,\chi\omega^{-k}) \prod_{\mfr{p}|p} (1-\chi\omega^{-k}(\mfr{p}) N(\mfr{p})^{k-1}).
$$
for positive integers $k$. Moreover, there exist relatively prime $G_{\chi}(T)$ and $H_{\chi}(T)$ in $\Lambda=\Zp[\chi][[T]]$ such that
$$
L_p(1-s,\chi)=G_{\chi}(u^s-1)/H_{\chi}(u^s-1),
$$
where $H_{\chi}(T)=1$ if $\chi$ is nontrivial; otherwise, $H_{\chi}(T)=T$. 

Let $\chi_1$ and $\chi_2$ be narrow ideal class characters of conductors $\n_1$ and $\n_2$, respectively, with associated signs $e_{1,\infty},e_{2,\infty}\in (\Z/2\Z)^d$ satisfying 
\begin{equation}\label{eq:sgn_cond}
e_{1,\infty}+e_{2,\infty}\equiv (0,\ldots,0) (\bmod\; 2\Z^d) 
\end{equation}
As the Teichm\"{u}ller character $\omega$ is totally odd, it follows from (\ref{eq:sgn_cond}) that the characters $\chi_1\omega^{2-k}$ and $\chi_2$ satisfy (\ref{eq:eisen_cond}). Therefore, the Eisenstein series $E_k(\chi_1\omega^{2-k},\chi_2)$ exists by Proposition~\ref{333} (or see Proposition~\ref{431}).

We now recall the definition of $\Lambda$-adic Eisenstein associated to $(\chi_1,\chi_2)$ following \cite[Proposition 1.3.1]{W2}.

\begin{prop}\label{lambda_adic_eisen}
Suppose $\chi_1$ and $\chi_2$ satisfy (\ref{eq:sgn_cond}) and the following properties
\begin{enumerate}
\item The character $\chi_1$ is nontrivial and $(\chi_1,\chi_2)\neq (\omega^{-2},\mathbbm{1})$.

\item We have $\mfr{n}_1\mfr{n}_2=\mfr{n}$ or $\mfr{n}p$ for some integral ideal $\mfr{n}$ prime to $p$.
\end{enumerate}
Then there exists a $\Lambda$-adic modular form $\mathcal{E}(\chi_1,\chi_2)\in M(\n,\chi_1\chi_2;\Lambda)$, called \textit{$\Lambda$-adic Eisenstein series}, satisfying
$v_{k,\zeta}(\mathcal{E}(\chi_1,\chi_2))=E_k(\chi_1\omega^{2-k}\rho_{\zeta},\chi_2)$. Moreover, its constant terms are defined as 
$$
C_{\lambda}(0,\mathcal{E}(\chi_1,\chi_2))=\delta(\chi_2)2^{-d} G_{\chi_1\chi_2^{-1}\omega^2}(u^2(T+1)-1),
$$
for $\lambda=1,\ldots, h^+_F$, where $\delta(\chi_2)=1$ if $\chi_2$ is a trivial character; otherwise, it is $0$. For an integral ideal $\mfr{m}$ of $\O_F$, its $\mfr{m}$-th Fourier coefficient is defined as
$$
C(\mfr{m},\mathcal{E}(\chi_1,\chi_2))=\sum_{\mfr{a}|\mfr{m},\gcd(p,\mfr{a})=1} \wt{\chi_1}(\mfr{a})\chi_2(\tfrac{\mfr{m}}{\mfr{a}}) N(\mfr{a}).
$$
Moreover, if $(\mfr{n}_2,p)=1$, then $\mathcal{E}(\chi_1,\chi_2)$ is in  $M^{\ord}(\mfr{n},\chi_1\chi_2;\Lambda)$.
\end{prop}

Let $p$ be an odd rational prime unramified in $F$. As in Proposition~\ref{432}, we write $\mfr{n}_i=\mfr{n}'_i\times \mfr{m}_i$ 
%with $\mfr{n}'_i=\prod_{\mfr{p}\nmid \gcd(\mfr{n}_1,\mfr{n}_2)} \mfr{p}^{e_{\mfr{p},i}}$ and $\mfr{m}_i= \prod_{\mfr{p}|\gcd(\mfr{n}_1,\mfr{n}_2)} \mfr{p}^{e_{\mfr{p},i}}$ 
for $i=1,2$. Also, we write the conductor of $\chi_1^{-1}\chi_2$ as $\mfr{n}_1'\mfr{n}_2'\mfr{m}$ for some integral ideal $\mfr{m}$ of $F$. For simplicity, we set 
\begin{equation}\label{eq:wh_G}
\wh{G}_{\chi_1\chi_2^{-1}}(T):=G_{\chi_1\chi_2^{-1}\omega^2}(u^2(T+1)-1).
\end{equation}

\begin{prop}\label{523}
Let $\wt{\mfr{n}_1}'=\lcm(\mfr{n}'_1,p)$ and let others notation and assumptions be as in Proposition~\ref{lambda_adic_eisen}. Assume further that $(\mfr{n}_2,p)=1$. Then the constant term of $\mathcal{E}(\chi_1,\chi_2)$ at the cusp associated to 
$x_{\lambda}^{-1}g=x_{\lambda}^{-1}\left(\begin{smallmatrix} a & b \\ c & d\end{smallmatrix}\right)$ is
\[
\begin{split}
C\times \wh{G}_{\chi_1\chi_2^{-1}}(T)\times \prod_{\mfr{q}|\mfr{n},\mfr{q}\nmid \cond(\chi_1\chi_2^{-1})} (1-\chi_1\chi_2^{-1}(\mfr{q})(1+T)^{-s(\mfr{q})}N(\mfr{q})^{-2})
\end{split}
\] 
if $g$ satisfies the following conditions
\begin{enumerate}

\item $\val_v(c_v)\geq \val_v(\wt{n_1}')$ for all $v|\wt{\mfr{n}_1}'$,

\item $\val_v(c_v)=0$ for all $v|\mfr{n}'_2$,

\item $\val_v(c_v)=e_{\mfr{p},1}$ for all $v|\gcd(\mfr{n}_1,\mfr{n}_2)$.
\end{enumerate}
Otherwise, it is $0$.
Here $C$ is a unit in $\Lambda$, and $s(\mfr{q})$ was defined in Section~\ref{sec:41}.
\end{prop}

\begin{proof}
To prove the assertion, we will compute the constant term of $v_k(\mathcal{E}(\chi_1,\chi_2))$ for all $k\geq 2$ with 
\begin{equation}\label{eq:cond_k}
\chi_1\omega^{2-k}|_{(\O_F/p\O_F)^{\times}}\neq \mathbbm{1}.
\end{equation}
Note that for all such $k$, the conductor of $\chi_1\omega^{2-k}$ is $\lcm(\mfr{n}_1,p)$ and by Proposition~\ref{lambda_adic_eisen}, one has $v_{k,1}(\mathcal{E}(\chi_1,\chi_2))=E_k(\chi_1\omega^{2-k},\chi_2)\in M_k^{\ord}(\mfr{n}p, \chi_1\omega^{2-k}\chi_2;\Lambda)$. If one of the above three conditions on $g$ does not hold, then the constant term of $E_k(\chi_1\omega^{2-k},\chi_2)$ is $0$ at the cusp associated to $x_{\lambda}^{-1}g$ by Proposition~\ref{432}. This yields that the constant term of $\mathcal{E}(\chi_1,\chi_2)$ at the cusp associated to $x_{\lambda}^{-1}g$ has infinitely many zeros, and hence, it has to be zero.

Now we assume that the above three conditions on $g$ hold. We claim that the constant term of $\mathcal{E}(\chi_1,\chi_2)$ at the cusp associated to $x_{\lambda}^{-1}g$ is
\[
\begin{split}
& C'\times (1+T)^{-s(\tfrac{\mfr{m}_1}{{\mfr{m}\mfr{n}'_2}})} \chi_2^{-1}(t_{\lambda}\mfr{D})  \wt{\chi_1}^{-1}(d_{\mfr{n}_1}) \chi_2(c_{\mfr{n}_2}\mfr{m}_1^{-1}) \times\\
& \wh{G}_{\chi_1\chi_2^{-1}}(T)\times \prod_{\mfr{q}|\mfr{n},\mfr{q}\nmid \cond(\chi_1\chi_2^{-1})} (1-\chi_1\chi_2^{-1}(\mfr{q})(1+T)^{-s(\mfr{q})}N(\mfr{q})^{-2})
\end{split}
\]
for some $p$-adic unit $C'$ in a finite cyclotomic extension of $\Q_p$. To see this, for $k\geq 2$ satisfying (\ref{eq:cond_k}), by Proposition~\ref{432}, the normalized constant term of $v_{k,1}(\mathcal{E}(\chi_1,\chi_2))=E_k(\chi_1\omega^{2-k},\chi_2)$ at the cusp associated to $x_{\lambda}^{-1}g$ is 
\[ 
\begin{split}
& C_1\times \omega^{k-2}(\mfr{n}_2') L(1-k, \chi_1^{-1}\omega^{2-k}\chi_2^{-1})\times \prod_{\mfr{q}|\mfr{n}, \mfr{q}\nmid \cond(\chi_1\chi_2^{-1})} (1-\chi_1\omega^{2-k}\chi_2^{-1}(\mfr{q}) N(\mfr{q})^{-k})\times \chi_2^{-1}(t_{\lambda}\mfr{D}) \times \\
& \chi_1\omega^{2-k}(d_{\mfr{n}_1}) \chi_2(c_{\mfr{n}_2}\mfr{m}_1^{-1}) \prod_{v|\mfr{n}_2} \varepsilon_v(2-k,\chi_{2,v}^{-1}\psi_v)^{-1} \times \prod_{v|\gcd(\mfr{n}_1,\mfr{n}_2)} M_{v,\tfrac{1-k}{2},e_{\mfr{p},1}}
\end{split}
\]
for some $p$-adic unit $C_1$ in a finite cyclotomic extension of $\Q_p$. Here $M_{v,\tfrac{1-k}{2},e_{\mfr{p},1}}$ was computed in Lemma~\ref{415}(3). Using Lemma~\ref{251}, one can simplify the above expression as follows
\[ 
\begin{split}
& C_2\times \omega^{k-2}(\tfrac{\mfr{m}_1}{{\mfr{m}\mfr{n}'_2}}) L(1-k, \chi_1^{-1}\omega^{2-k}\chi_2^{-1})\times \prod_{\mfr{q}|\mfr{n}, \mfr{q}\nmid \cond(\chi_1\chi_2^{-1})} (1-\chi_1\omega^{2-k}\chi_2^{-1}(\mfr{q}) N(\mfr{q})^{-k})\times \chi_2^{-1}(t_{\lambda}\mfr{D}) \times \\
& \chi_1\omega^{2-k}(d_{\mfr{n}_1}) \chi_2(c_{\mfr{n}_2}\mfr{m}_1^{-1}) \prod_{v|\mfr{n}_2} \varepsilon_v(2-k,\chi_{2,v}^{-1}\psi_v)^{-1} \times \prod_{v|\gcd(\mfr{n}_1,\mfr{n}_2)} \frac{\tau(\chi_{1,v})}{\tau(\chi_{1,v}^{-1}\chi_{2,v})}
\end{split}
\]
for some $p$-adic unit $C_2$ in a finite cyclotomic extension of $\Q_p$. It is easy to see that $v_{k,1}((1+T)^{-s(\tfrac{\mfr{m}_1}{{\mfr{m}\mfr{n}'_2}})})$ is a $p$-adic unit times $\omega^{k-2}(\tfrac{\mfr{m}_1}{{\mfr{m}\mfr{n}'_2}})$ and that $v_{k,1}(\wt{\chi_1}^{-1})(\mfr{a})$ is a $p$-adic unit times $\chi_1^{-1}\omega^{k-2}(\mfr{a})$ for all integral ideals $\mfr{a}$ prime to $\mfr{n}p$. Moreover, a direct computation yields that 
$$
v_{k,1}(1-\chi_1\chi_2^{-1}(\mfr{q})(1+T)^{-s(\mfr{q})}N(\mfr{q})^{-2})=1-\chi_1\omega^{2-k}\chi_2^{-1}(\mfr{q}) N(\mfr{q})^{-k}.
$$
To complete the proof, we claim that $\prod_{v|\mfr{n}_2} \varepsilon_v(2-k,\chi_{2,v}^{-1},\psi_v)$ and $\prod_{v|\gcd(\mfr{n}_1,\mfr{n}_2)} \tfrac{\tau(\chi_{1,v})}{\tau(\chi_{1,v}^{-1}\chi_{2,v})}$ are also $p$-adic units in a finite cyclotomic extension over $\Qp$. By the computation in \cite[p.~259]{Hida2} (see the discussion for the equation (4a) in \textit{loc.~cit.}), one can write 
$\varepsilon_v(2-k,\chi_{2,v}^{-1},\psi_v)$ as $\tau(\chi_{2,v})$ times a $p$-adic unit.
Then the claim follows from (\ref{eq:gauss_sum}) and the assumption that $\mfr{n}_2$ is prime to $p$.
\end{proof}
 
\begin{defn}\label{524}
Denote by $\mfr{P}$ the maximal ideal of $\O_{\infty}$, and let the notation be as above. Assume that $\n_1\n_2=\n p$ or $\n$.
\begin{enumerate}
\item The \textit{Eisenstein ideal $\mathcal{I}(\chi_1,\chi_2)$ associated with the pair of characters $(\chi_1,\chi_2)$} is defined as the kernel of the $\Lambda$-module homomorphism 
$$
\mathcal{H}^{\ord}(\n,\chi_1\chi_2;\Lambda)\map \Lambda;\; T\mapsto C(1,T\cdot \mathcal{E}(\chi_1,\chi_2)).
$$
We denote by $\mfr{M}(\chi_1,\chi_2):=(\mathcal{I}(\chi_1,\chi_2), \mfr{P}, T)$ the maximal ideal of $\mathcal{H}^{\ord}(\n,\chi_1\chi_2;\Lambda)$ containing $\mathcal{I}(\chi_1,\chi_2)$. We denote by $I(\chi_1,\chi_2)$ (resp.~$m(\chi_1,\chi_2)$) the image of $\mathcal{I}(\chi_1,\chi_2)$ (resp.~$\mfr{M}(\chi_1,\chi_2)$) in the cuspidal Hecke algebra $h^{\ord}(\n,\chi_1\chi_2;\Lambda)$.

\item We say that a pair of narrow ideal class characters $(\chi_1,\chi_2)$ is \textit{not exceptional} if the maximal ideal $\mfr{M}(\chi_1,\chi_2)$ does not contain any Eisenstein ideal other than $\mathcal{I}(\chi_1,\chi_2)$. 
\end{enumerate}
\end{defn}

\begin{prop}\label{525}
Let $(\chi_1,\chi_2)$ and $(\chi_1',\chi_2')$ be two pairs of narrow ray class characters of conductors $\n_i$ and $\n'_i$ for $i=1,2$, respectively. Assume that $\n_1\n_2=\n'_1\n'_2=\n$ or $\n p$. Then two Eisenstein sereis $\mathcal{E}(\chi_1,\chi_2)$ and $\mathcal{E}(\chi'_1,\chi'_2)$ are the same modulo $(\mfr{P}, T)$ if and only if 
$$
\begin{cases}
\chi_1\equiv\chi'_1 \mbox{ and } \chi_2\equiv\chi'_2 \bmod \mfr{P} ;\mbox{ or}\\
\chi_1\equiv\chi'_2\omega^{-1},\; \chi_2\equiv\chi'_1\omega,\mbox{ and } \chi_1\omega\chi_2^{-1}(\mfr{p})\equiv1 \bmod \mfr{P} \mbox{ for all } \mfr{p}|p.
\end{cases}
$$

If we assume further that $p\nmid \phi(N(\n)) h_F$, then $\mathcal{E}(\chi_1,\chi_2)$ and $\mathcal{E}(\chi'_1,\chi'_2)$ are the same modulo $(\mfr{P}, T)$ if and only if 
\begin{equation}
\begin{cases}
\chi_1=\chi'_1 \mbox{ and } \chi_2=\chi'_2;\; \mbox{or}\\
\chi_1=\chi'_2\omega^{-1},\; \chi_2=\chi'_1\omega,\mbox{ and } \chi_1\omega\chi_2^{-1}(\mfr{p})=1 \mbox{ for all } \mfr{p}|p.
\end{cases}
\end{equation}
\end{prop}

\begin{proof}
We follow the argument in \cite[Lemma 1.4.9]{Ohta3}. For any prime ideal $\mfr{q}$ not dividing $\n p$, we have
$$
\wt{\chi}_1(\mfr{q})N(\mfr{q})+\chi_2(\mfr{q})\equiv \wt{\chi'}_1(\mfr{q})N(\mfr{q})+\chi'_2(\mfr{q})\bmod (\mfr{P}, T).
$$
We obtain that
$$
\chi_1(\mfr{q})N(\mfr{q})+\chi_2(\mfr{q})\equiv \chi'_1(\mfr{q})N(\mfr{q})+\chi'_2(\mfr{q})\bmod \mfr{P}.
$$
Since $N(\mfr{q})\equiv \omega(\mfr{q}) \bmod \mfr{P}$, we have
$$
\chi_1(\mfr{q})\omega(\mfr{q})+\chi_2(\mfr{q})\equiv \chi'_1(\mfr{q})\omega(\mfr{q})+\chi'_2(\mfr{q})\bmod \mfr{P}.
$$
Thus by Artin's lemma on the linear independence of characters, we have
$$
\begin{cases}
\chi_1\omega\equiv \chi'_1\omega\mbox{ and } \chi_2\equiv \chi'_2 \bmod \mfr{P}, \mbox{ or }\\
\chi_1\omega\equiv \chi'_2 \mbox{ and } \chi_2\equiv \chi'_1\omega\bmod \mfr{P}. 
\end{cases}
$$
Since the narrow ray class number $h_F^+(\m)=|\Cl_F^+(\m)|$ divides $2^d\phi(N(\m))h_F$ for all integral ideals $\m$ of $\O_F$, the assumption that $p\nmid \phi(N(\n)) h_F$ implies that the the field $\Qp[\chi_i, \chi_i'\mid i=1,2]$ is unramified over $\Qp$.
Therefore, by the Teichm\"{u}ller lifting, we know that 
$$
\begin{cases}
\chi_1= \chi'_1\mbox{ and } \chi_2= \chi'_2, \mbox{ or }\\
\chi_1\omega= \chi'_2 \mbox{ and } \chi_2= \chi'_1\omega. 
\end{cases}
$$
To complete the proof, we recall that for each prime ideal $\mfr{p}$ dividing $p$, the $\mfr{p}$th Fourier coefficient of $\mathcal{E}(\chi_1,\chi_2)$ (resp.~$\mathcal{E}(\chi'_1,\chi'_2)$) is $\chi_2(\mfr{p})$ (resp.~$\chi'_2(\mfr{p})$). Hence, we have $\chi_1\omega\chi_2^{-1}(\mfr{p})\equiv 1 \bmod \mfr{P}$ and have $\chi_1\omega\chi_2^{-1}(\mfr{p})= 1$ if $p\nmid \phi(N(\n)) h_F$ for all $\mfr{p}|p$.
\end{proof}

The above proposition shows that a pair of characters $(\chi_1,\chi_2)$ is not exceptional if the following condition holds
\begin{equation}\label{eq:exception}
p\nmid N(\n)\phi(N(\n)) h_F \mbox{ and } \chi_1\omega\chi_2^{-1}(\mfr{p})\neq 1 \mbox{ for some }\mfr{p}|p.
\end{equation}
%%%%%%%%%%%%%%%%%%%%%%%%%%%%%%%%%%%%%%%%%%%%%%%%%%%%%%%%%%%%%%%%%%%%%%%%%%%%%%%%%%%%%%%%%%%%%%%%%%%%%%%
\section{Hilbert modular varieties and $p$-adic modular forms}\label{sec:05}
Throughout this section, we will write $\O=\O_F$ for simplicity and denote by $k$ a positive integer. Let $\mfr{n}$ be a nonzero integral ideal in $\O$. We set 
$$
\Gamma_1(\mfr{n})=\left\{\left(\begin{array}{cc} a & b\\ c & d\end{array}\right)\in \GL_2(\O_F)^+ \mid c \in \mfr{n}, d-1\in \mfr{n} \right\} \mbox{ and }
\Gamma_1^1(\mfr{n})=\left\{\left(\begin{array}{cc} a & b\\ c & d\end{array}\right)\in \Gamma_1(\mfr{n})\mid a-1\in \mfr{n} \right\}.
$$ 
In this section, we will first review moduli problems with different level structures and their compactifications. We then review the definition of $p$-adic modular forms. The main goal in this section is to prove a control theorem for $M^{\ord}(\mfr{n},\chi;\Lambda)$ (Corollary~\ref{655}) which is a key result in proving Theorem~\ref{711}. We are only able to prove such a theorem by using its geometric formulation (see (\ref{eq:geom_family})). This is the aim of the last section.
%%%%%%%%%%%%%%%%%%%%%%%%%%%%%%%%%%%%%%%%%%%%%%%%%%%%%%%%%%%%%%%%%%%%%%%%%%%%%%%%%%%%%%%%%%%%%%%%%%%%%
\subsection{Moduli problems with level structures}\label{sec:51}
In this subsection, we review moduli problems with different level structures. We refer the reader to \cite[\S 4.1.2]{Hida3} and \cite[Ch.~3, \S 6]{Gor} for more information. 

Recall that an abelian scheme $A$ with real multiplication (RM) by $\O$ over a base scheme $B$ is a proper smooth geometrically irreducible group scheme over $B$ together with an injection $\iota:\O\embed \End(A_{/B})$. We say that the abelian scheme $A$ satisfies the Rapoport condition if
\begin{equation}\tag{R}\label{R}
\Lie(A) \mbox{ is locally free } \O\otimes_{\Z} \O_{B}\mbox{-module of rank }1.
\end{equation}

\begin{defn}
A \textit{Hilbert-Blumenthal abelian variety (HBAV) $A$ over a scheme $B$} is an abelian scheme $A$ with RM by $\O$ over $B$ of relative dimension $d=[F:\Q]$ satisfying the condition (\ref{R}).  
\end{defn}

We will denote by $A^t$ the dual abelian scheme of $A$. Let $\mfr{c}$ be a fractional ideal of $F$. A $\mfr{c}$-polarization is an $\O$-linear isomorphism 
$\lambda:(\mathcal{M}_A,\mathcal{M}_A^+)\isom (\mfr{c},\mfr{c}^+)$ of sheaves in the \'{e}tale topology, where $\mathcal{M}_A=\Hom_{\O}(A,A^t)^{\sym}$. For the definition of $\Hom_{\O}(A,A^t)^{\sym}$, see \cite[Ch.~3, \S 6]{Gor}.    

For each integral ideal $\mfr{n}$ of $F$, \textit{a $\mu_{\mfr{n}}$-level structure (also, called by $\Gamma_1^1(\mfr{n})$-level structure) on a HBAV} $A_{/B}$ is an $\O$-linear closed immersion 
$$
\iota_{\mfr{n}}:\mu_{\mfr{n}}\otimes_{\Z} \mfr{D}^{-1}\embed A
$$ 
of group schemes over $B$. Here ${\mu_{\mfr{n}}}\otimes_{\Z} \mfr{D}^{-1}$ is the $\mfr{n}$-torsion points of $\bbG_m\otimes_{\Z}\mfr{D}^{-1}$. That is 
$$
(\mu_{\mfr{n}}\otimes_{\Z} \mfr{D}^{-1})(R)=\{x\in \bbG_m(R)\otimes_{\Z} \mfr{D}^{-1}\mid \mfr{n}\cdot x=0\}
$$ 
for $\O$-algebra $R$. A \textit{$\mu_{p^{\infty}}$-level structure} is a compatible sequence of $\mu_{p^n}$-level structures for $n\in \Z_{>0}$. Note that an abelian variety over a field of \ch $p$ with real multiplication by $\O$ and with $\mu_{p^n}$-level structure is ordinary in the sense that the connected component $A[p^r]^{\circ}$ of $A[p^r]$ is isomorphic to $\mu_{p^r}^d$ \'{e}tale locally.

Let $\mfr{n}$ and $\mfr{c}$ be, respectively, a fixed integral ideal and a fractional ideal of $F$ such that $\Gamma_1^1(\mfr{n})$ is neat, i.e., each test object $(A,\lambda,\iota_{\mfr{n}})$ of $\Gamma_1^1(\mfr{n})$-level structure does not have any nontrivial automorphism, where $\lambda$ is a $\mfr{c}$-polarization (for example, $\mfr{n}$ is generated by a positive integer $N\geq 4$). The functor assigning to a $\Z$-scheme $B$ the set of isomorphism classes of tuples $(A,\lambda,\iota_{\mfr{n}})_{/B}$ (resp.~$(A,(\O^{\times})^+\lambda,\iota_{\mfr{n}})_{/B}$) is representable by a geometrically connected, quasi-projective scheme $\mfr{M}(\mfr{c},\Gamma_1^1(\mfr{n}))$ (resp.~$\mfr{M}(\mfr{c},\Gamma_1(\mfr{n}))$) over $\Z$, which is smooth over $\Z[\frac{1}{N(\mfr{nD})}]$. Note that the coarse moduli schemes $\mfr{M}(\mfr{c},\Gamma_1^1(\mfr{n}))$ and $\mfr{M}(\mfr{c},\Gamma_1(\mfr{n}))$ exist for all integral ideals $\mfr{n}$.
%%%%%%%%%%%%%%%%%%%%%%%%%%%%%%%%%%%%%%%%%%%%%%%%%%%%%%%%%%%%%%%%%%%%%%%%%%%%%%%%%%%%%%%%%%%%%%%%%%%%%%%
\subsection{Geometric modular forms}\label{sec:52}
In this subsection, we review the definition of geometric modular forms of (parallel) weight following \cite[\S 4.1]{Hida3} and \cite[Ch.~5, \S 1]{Gor}. 

Let $\vartheta:(\mathcal{A}(\mfr{c}), \iota_{\mfr{n}})\map (\mfr{M}(\mfr{c},\Gamma_1(\mfr{n})),\iota_{\mfr{n}})$ be the universal abelian scheme with real multiplication by $\O$ with $\mu_{\mfr{n}}$-level structure (see \cite[\S 4.1]{DT} for more details). We denote by $\ul{\omega}=\det \vartheta_*\Omega_{\mathcal{A}(\mfr{c})/\mfr{M}(\mfr{c},\iota_{\mfr{n}})}$ the determinant of the pushforward of the sheaf of relative differentials on $\mathcal{A}(\mfr{c})$. For a $\Z[\frac{1}{N(\mfr{nD})}]$-algebra $R$, set $\mfr{M}(\mfr{c},\Gamma_1(\mfr{n}))_{/R}:=\mfr{M}(\mfr{c},\Gamma_1(\mfr{n}))\times_{\Spec(\Z[\frac{1}{N(\mfr{nD})}])} \Spec(R)$. 

\begin{defn}\label{621}
Let $R$ be a $\Z[\frac{1}{N(\mfr{nD})}]$-algebra. A \textit{$\mfr{c}$-Hilbert modular form $f$ over $R$ of level $\Gamma_1(\mfr{n})$  and weight $k$} is a global section of $\ul{\omega}^k$ on $\mfr{M}(\mfr{c},\Gamma_1(\mfr{n}))_{/R}$.
\end{defn}

By Definition~\ref{621}, we know that the space of $\mfr{c}$-Hilbert modular forms over $R$ of level $\Gamma_1(\mfr{n})$ and weight $k$ is $H^0(\mfr{M}(\mfr{c},\Gamma_1(\mfr{n}))_{/R},\ul{\omega}^k)$.
By the discussion in the previous section, we know that the HBAV $A$ over $\Spec(R)$ satisfies the Rapoport condition, i.e., $\ul{\omega}$ is a free $\O\otimes_{\Z} R$-module of rank $1$. A generator $\omega_0\in \ul{\omega}$ is called a non-vanishing differential. The following definition is equivalent to Definition~\ref{621} (see \cite[\S 1.2]{Kat}).

\begin{defn}
A \textit{$\mfr{c}$-Hilbert modular form $f$ over $R$ of level $\Gamma_1(\mfr{n})$ and weight $k$} is a rule
$$
(A,(\O^{\times})^+\lambda,\iota_{\mfr{n}},\omega_0)_{/\Spec(R)} \mapsto f(A,(\O^{\times})^+\lambda,\iota_{\mfr{n}},\omega_0)\in R
$$
satisfying the following properties:
\begin{enumerate}
\item for any $R$-algebra $R'$, one has
$$
f((A,(\O^{\times})^+\lambda,\iota_{\mfr{n}},\omega_0)\times_{\Spec(R)} \Spec(R'))=f(A,(\O^{\times})^+\lambda,\iota_{\mfr{n}},\omega_0)\otimes_R R',
$$

\item the value of $f$ on $(A,(\O^{\times})^+\lambda,\iota_{\mfr{n}},\omega_0)_{/\Spec(R)}$ only depends on its isomorphism class,

\item for $\alpha\in \O^{\times}$, we have
$$
f(A,(\O^{\times})^+\lambda,\iota_{\mfr{n}},\alpha^{-1}\omega_0)=N_{F/\Q}(\alpha)^k f(A,(\O^{\times})^+\lambda,\iota_{\mfr{n}},\omega_0).
$$
\end{enumerate}
We denote by $G_k(\mfr{c},\Gamma_1(\mfr{n});R)$ the space of modular forms over $R$ of level $\Gamma_1(\mfr{n})$ and weight $k$. 
\end{defn}

By the above two definitions, we obtain the equality 
\begin{equation}\label{eq:geom_form}
H^0(\mfr{M}(\mfr{c},\Gamma_1(\mfr{n}))_{/R},\ul{\omega}^k)=G_k(\mfr{c},\Gamma_1(\mfr{n});R)
\end{equation}
for all $\Z[\frac{1}{N(\mfr{nD})}]$-algebras $R$.
 
Let $\chi_0:(\O/\mfr{n}\O)^{\times}\map \C^{\times}$ be a character of finite order, and assume that $R$ contains all of the values of $\chi_0$. We say that $f\in G_k(\mfr{c},\Gamma_1(\mfr{n});R)$ is \textit{of type $\chi_0$} if 
$$
a\cdot f(A,(\O^{\times})^+\lambda,\iota_{\mfr{n}},\omega_0)=f(A,(\O^{\times})^+\lambda,a\iota_{\mfr{n}},\omega_0)=\chi_0(a)f,
$$
for all $a\in (\O/\mfr{n}\O)^{\times}$. We denote by $G_k(\mfr{c},\mfr{n},\chi_0;R)$ and $H^0(\mfr{M}(\mfr{c},\Gamma_1(\mfr{n}))_{/R},\ul{\omega}^k)^{(\chi_0)}$ the space of $\mfr{c}$-Hilbert modular forms of level $\Gamma_1(\mfr{n})$, type $\chi_0$, and weight $k$.

When $R=\C$, it is known \cite[\S 4.1.3]{Hida3} that there are canonical isomorphisms
$$
M_k(\Gamma_1(\mfr{c},\mfr{n});\C)\isom G_k(\mfr{c},\Gamma_1(\mfr{n});\C).
$$
Recall that the space $M_k(\Gamma_1(\mfr{c},\mfr{n});\C)$ was defined in Section 2. This isomorphism is obtained by the fact that for all $z\in \mathbf{H}^d$, one can construct a complex HBAV $A_z$, and all complex HBAVs are of this form. Here $\mathbf{H}$ is the complex upper half plane.

Now, we fix an odd rational prime $p$ unramified in $F$. Let $\mfr{n}$ be an integral ideal prime to $p$. 
Recall that Deligne\textendash Ribet \cite[\S 5]{DR} proved that $q$-expansion principle holds for modular forms of level $\Gamma_1^1(\mfr{n})$, and hence, it also holds for modular forms of level $\Gamma_1(\mfr{n})$. It asserts that the $q$-expansion of a modular form at the cusp $\infty$ determines the modular form, i.e., if all of the coefficients of $f$ are in a $\Zp$-algebra $R$, then $f\in G_k(\mfr{c},\Gamma_1(\mfr{n});R)$. Thus, we have the following isomorphisms:
$$
M_k(\Gamma_1(\mfr{c},\mfr{n});R)\isom G_k(\mfr{c},\Gamma_1(\mfr{n});R).
$$  

Let $\{t_1,\ldots, t_{h^+_F}\}$ be a fixed set of representatives of $\Cl_F^+$ such that $t_{\lambda}$ and $\mfr{n}p$ are coprime for all $\lambda=1,\ldots,,h^+_F$. We set 
\begin{equation}\label{eq:adelic_geom_form}
G_k(\mfr{n};R):=\bigoplus_{\lambda=1}^{h^+_F} G_k(t_{\lambda}\mfr{D},\Gamma_1(\mfr{n});R).
\end{equation}
The space $G_k(\mfr{n};R)$ coincides with the space $M_k(K_1(\mfr{n});R)$ defined in Section~\ref{sec:02}.

In Section~\ref{sec:02}, we reviewed the Hecke action on $M_k(K_1(\mfr{n});R)$. This action induces a Hecke action on $G_k(\mfr{n};R)$. One can also define the action geometrically (see \cite[\S 4.1.10]{Hida3} for the definition). Let $e^{\circ}=\lim_{n\to\infty} \prod_{\mfr{p}|p} T(\mfr{p})^{n!}$, and let $e=\lim_{n\to\infty} \prod_{\mfr{p}|p} U(\mfr{p})^{n!}$. We set 
$$
G_k^{\ord}(\mfr{n};R)=e^{\circ}\cdot G_k(\mfr{n};R)
\mbox{ and }
G_k^{\ord}(\mfr{n}p^r;R)=e\cdot G_k(\mfr{n}p^r;R)
$$
for $r\in \Z_{>0}$.
Then we have
\begin{equation}\label{eq:ord_forms}
G_k^{\ord}(\mfr{n}p^r;R)=M^{\ord}_k(K_1(\mfr{n}p^r);R)
\end{equation}
for all $r\in \Z_{\geq 0}$. 
%Note that when $r>0$, we only obtain the space of ordinary modular forms since all the test objects are ordinary modulo primes dividing $p$. We also define $G_k^{\ord}(\mfr{c},\Gamma;R)$ for $\Gamma=\Gamma(N)$, $\Gamma_1^1(\mfr{n})$, and $\Gamma_1(\mfr{n})$ in the same manner.
%%%%%%%%%%%%%%%%%%%%%%%%%%%%%%%%%%%%%%%%%%%%%%%%%%%%%%%%%%%%%%%%%%%%%%%%%%%%%%%%%%%%%%%%%%%%%%%%%%%%%%%%%
\subsection{Toroidal compactification and minimal compacification}\label{sec:53}
From now on, we fix an integral ideal $\mfr{n}$ and a fractional ideal $\mfr{c}$ such that $\mfr{M}(\mfr{c},\Gamma_1(\mfr{n}))$ is a fine moduli scheme. For simplicity, we write $\mfr{M}(\mfr{c},\Gamma_1(\mfr{n}))$ as $\mfr{M}$.
The existence of toroidal compactifications and minimal compactification of Hilbert modular varieties $\mfr{M}$ has been proved in \cite{Dim}. We refer the reader to \textit{loc.~cit.} for more details.  

To a smooth rational cone decomposition $\Sigma$ of $F^+$ (see \textit{loc.~cit.} for the definition), one can attach the toroidal compactification $M=M_{\Sigma}(\mfr{c},\Gamma_1(\mfr{n}))$, which is proper smooth scheme over $\Z[\frac{1}{N(\mfr{nD})}]$ containing $\mfr{M}$ as an open dense subscheme. The boundary $M-\mfr{M}$ is a divisor with normal crossing. Moreover, there is a tuple $(\mathcal{G},\lambda, \eta)$ over $M$, where $\pi:\mathcal{G}\map M$ is a semi-abelian scheme with $\O$-action, $\lambda:\mathcal{G}\map \mathcal{G}^t$ is a homomorphism such that the pullback of $\mfr{M}\subset M$ in $\mathcal{G}$ is $\mathcal{A}(\mfr{c})$, and $\eta$ is the corresponding level structure. We also denote by $\ul{\omega}=\det \pi_*\Omega_{\mathcal{G}/M}$ the determinant of the sheaf of relative differentials on $M$, which extends the sheaf of relative differentials on $\mfr{M}$. 
The Koecher's principle \cite[Theorem~7.1]{Dim} asserts that for each positive integer $k$, if $F\neq \Q$, one has
$$
H^0(M,\ul{\omega}^{k})=H^0(\mfr{M}, \ul{\omega}^{k}).
$$
Therefore, $H^0(M,\ul{\omega}^{k})$ is independent of the choice of the cone decomposition $\Sigma$ for all positive integers $k$.  

%%%%%%%%%%%%%%%%%%%%%%%%%%%%%%%%%%%%%%%%%%%%%%%%%%%%%%%%%%%%%%%%%%%%%%%%%%%%%%%%%%%%%%%%%%%%%%%%%%%%%%%%%
%\subsection{Minimal compactification}\label{sec:54}
%Let the notation be in the previous section. The minimal compactification of $\mfr{M}$ was established in \cite{Dim}. 

We put $\O_{\mathcal{M}}=\bigoplus_{k\geq 0} H^0(\mfr{M}, \ul{\omega}^{k})$. The minimal compactification of $\mfr{M}$ is given by $M^*:=\Proj(\O_{\mathcal{M}})$. It is projective, normal, and flat over $\Z[\frac{1}{N(\mfr{nD})}]$. On $\C$-points, it is obtained by adjoining one point at each cusp, i.e., $M^*(\C)=\mfr{M}(\C)\cup \{\mbox{cusps}\}$.  The invertible sheaf $\ul{\omega}$ on $\mfr{M}$ extends to an ample line bundle \cite[Lemma 2.1]{DW}, also denoted by $\ul{\omega}$, on $M^*$. 

We now fix an odd rational prime $p$ unramified in $F$ and relatively prime to $\mfr{n}$ and $\mfr{c}$. Then for $a\in \Z_{>0}$ big enough, we have
$$
H^0(M^*_{/W}, \ul{\omega}^{a(p-1)})\otimes_W \bbF=H^0(M^*_{/\bbF}, \ul{\omega}^{a(p-1)}),
$$
where $W$ is a $p$-adically complete DVR. One obtains a lifting $E\in H^0(M^*_{/W}, \ul{\omega}^{a(p-1)})$ of $\Ha^a$, where $\Ha\in H^0(M^*_{/\bbF}, \ul{\omega}^{(p-1)})$ is the Hasse invariant. Recall that the Hasse invariant $\Ha$ satisfies the property that 
\begin{equation}\label{eq:Hasse}
\Ha\equiv 1 \bmod p.
\end{equation}
See \cite[\S 4.1]{Hida3} for more details of $\Ha$ and $E$. Since the Hasse invariant is a nontrivial section on $\ul{\omega}^{p-1}$, the ordinary locus of $M^*$, which is denoted by $S^*=M^*[\frac{1}{E}]\subset M^*$, is defined by $S^*:=\Spec(\O_{\mathcal{M}}/(E-1))$.
% Moreover, it is affine and irreducible (see \textit{loc.~cit.}). 
We put $S=M[\frac{1}{E}]$ and $\mfr{S}=\mfr{M}[\frac{1}{E}]$. For $r\in \Z_{>0}$, one can view $E$ as an element in $H^0(\mfr{M}(\mfr{c}, \Gamma_1(\mfr{n}p^r))_{/W},\ul{\omega}^{a(p-1)})$ via the natural embedding 
$$
H^0(\mfr{M}_{/W},\ul{\omega}^{a(p-1)})\embed H^0(\mfr{M}(\mfr{c}, \Gamma_1(\mfr{n}p^r))_{/W},\ul{\omega}^{a(p-1)})
$$ 
induced by the natural forgetful morphism $\mfr{M}(\mfr{c}, \Gamma_1(\mfr{n}p^r))\to \mfr{M}$.
We set $\mfr{S}(\mfr{c},\Gamma_1(\mfr{n}p^r))=\mfr{M}(\mfr{c}, \Gamma_1(\mfr{n}p^r))[\frac{1}{E}]$. Note that $S$ and $\mfr{S}$ are not affine. 
%From the discussion in the previous section, one has
%\begin{equation}\label{ord_form}
%eH^0(\mfr{S}(\mfr{c}, \Gamma_1(\mfr{n}p^r)),\ul{\omega}_{/W}^{k})=G^{\ord}_k(\mfr{c},\Gamma_1(\mfr{n}p^r);W).
%\end{equation}

Recall that we have a canonical morphism $\pi: M\map M^*,$
which induces a canonical morphism $\pi:S\map S^*$ \cite[Theorem 8.6]{Dim}.
\begin{lemma}\label{54}
Let the notation be as above, and let $W_m=W/p^mW$. Suppose that $\mfr{n}$ is divisible by a positive integer $N\geq 3$. Then we have
$$
\pi_*(\ul{\omega}^k_{/W})\otimes_W W_m\isom\pi_*(\ul{\omega}^k_{/W}\otimes_W W_m)
$$
for $k\in \Z_{>1}$.
\end{lemma}

\begin{proof}
We follow the argument in \cite[p.~120]{Hida3}. Since $\pi$ is an isomorphism outside $S-\mfr{S}$, it suffices to show the assertion on the stalk at each cusp in $S^*$.  For each cusp $c\in S^*$ associated to integral ideals $\mfr{b}$ and $\mfr{b}'$ (see \cite[Definition~3.2]{Dim} for the definition), let $R$ be a $W[\frac{1}{N(\mfr{nD})},\zeta_c]$-algebra, where $\zeta_c$ is a $N_c$th root of unity for some positive integer $N_c$ prime to $p$. By \cite[\S 8]{Dim} and the proof of Proposition 3.3 in \textit{loc.~cit.}, the stalk of $\pi_*(\ul{\omega}^k_{/S})$ at $c$ is
$$
\wh{\pi_*(\ul{\omega}^k_{/R})}_c
= \left\{\sum_{\xi\in (\mfr{cbb}')^+\cup \{0\}} a(\xi)q^{\xi} \;\bigg|\; a(\xi)\in R, a(u^2\varepsilon\xi)=N_{F/\Q}(\varepsilon u^2)^{k/2} a(\xi) \mbox{ for all } (u,\varepsilon)\in \O_c\right\},
$$
where $\O_c=\{(u,\varepsilon)\in \O^{\times}\times (\O^{\times})^+\mid u-1\in \mfr{nbb'}^{-1}, u\varepsilon-1\in \mfr{bb'}^{-1} \}$ and $\q^{\xi}=e^{2\pi i\tr(\xi z)}$. When $u^2\varepsilon=1$, we have
$$
a(\xi)=N(\varepsilon u^2)^{k/2} a(\xi)=a(\xi).
$$
Moreover, when $\xi=0$, we have
$$
a(0)= N(\varepsilon u^2)^{k/2} a(0)=a(0),
$$
since $\varepsilon u^2\in (\O^{\times})^+$. From the argument in \cite[p.~120]{Hida3}, one sees that the above observation yields isomorphisms 
$$
\wh{\pi_*(\ul{\omega}^k_{/R})}_c\otimes_W W_m\isom \wh{\pi_*(\ul{\omega}^k_{/R}\otimes_W W_m)}_c
$$ 
for all cusps $c\in S^*$, and hence, the assertion follows.
\end{proof}

\begin{cor}\label{53}
Let the notation and the assumptions be as in Lemma~\ref{54}. Then we have
$$
H^0(S_{/W},\ul{\omega}^k)\otimes_W W_m\isom H^0(S_{/W},\ul{\omega}^k\otimes_W W_m).
$$
\end{cor}

\begin{proof}
The assertion is (Hp1) in \cite[p.~120]{Hida3}. From the discussion in  \textit{loc.~cit.}, one sees that the assertion follows from Lemma~\ref{54}.
\end{proof}

%%%%%%%%%%%%%%%%%%%%%%%%%%%%%%%%%%%%%%%%%%%%%%%%%%%%%%%%%%%%%%%%%%%%%%%%%%%%%%%%%%%%%%%%%%%%%%%%%%%%%%%
\subsection{$p$-adic modular forms}\label{sec:55}
In this section, we review the definition of $p$-adic modular forms following \cite[\S 4.1]{Hida3} and prove a control theorem for later use. For simplicity, we set $\Gamma=\Gamma_1(\mfr{n})$ and set $\Gamma^r=\Gamma\cap \Gamma_1(p^r)$.

We first review the definition of Igusa tower following \cite[\S 4.1.6]{Hida3}. Let $p$ be an odd rational prime unramified in $F$, and let $\mfr{n}$ and $\mfr{c}$ be as in the previous subsection. Let $W$ be a $p$-adically complete DVR, and set $W_m:=W/p^mW$. The \textit{Hilbert modular Igusa tower} $T_{m,n}(\mfr{c},\Gamma_1(\mfr{n}))_{/W_m}$ is the moduli stack over $W_m$ that parameterizes isomorphism classes of tuples $(A, (\O^{\times})^+\lambda, \iota_{\mfr{n}}, \iota_{p^n})_{/B}$ over a $W_m$-scheme $B$, where
\begin{itemize}
\item $A\map B$ is a HBAV,

\item $\lambda:A\map A^t$ is a $\mfr{c}$-polarization,

\item $\iota_{\mfr{n}}$ and $\iota_{p^n}$ are respectively $\mu_{\mfr{n}}$ and $\mu_{p^n}$-level structure.
\end{itemize} 
The Igusa tower $T_{m,n}(\mfr{c},\Gamma_1(\mfr{n}))$ is an \'{e}tale covering over $S\otimes_W W_m$ with Galois group $(\O/p\O)^{\times}$.

Following \textit{loc.~cit.}, we define 
$$
V_{m,n}(\mfr{c},\Gamma):=H^0(T_{m,n}(\mfr{c},\Gamma)_{/W_m},\O_{T_{m,n}}),\;
V_{m,\infty}(\mfr{c},\Gamma):=\bigcup_n V_{m,n}(\mfr{c},\Gamma),
$$
$$
V(\mfr{c},\Gamma):=\varprojlim_m V_{m,\infty}(\mfr{c},\Gamma), \mbox{ and }
\mathcal{V}(\mfr{c},\Gamma):=\varinjlim_m V_{m,\infty}(\mfr{c},\Gamma).
$$ 
Here the projective limit is with respect to the natural isomorphisms $V_{m+1,\infty}/p^m V_{m+1,\infty}\isom V_{m,\infty}$ for all $m\in \Z_{>0}$, and the direct limit is with respect to the morphisms induced by multiplication by $p$. The space $V(\mfr{c},\Gamma)$ is the space of  $p$-adic $\mfr{c}$-Hilbert modular forms of level $\Gamma$.  We put $\Lambda=W[[1+p\Zp]]\isom W[[T]]$. Then $V(\mfr{c},\Gamma)$ is a $W[[\Zp^{\times}]]$-module and in particular, a $\Lambda$-module.

\begin{lemma}\label{base_change_prop}
There is an isomorphism
$$
H^0(\mfr{S}(\mfr{c}, \Gamma_1(\mfr{n}p^r))_{/W},\ul{\omega}^k)\otimes_W W_m\isom H^0(\mfr{S}(\mfr{c},\Gamma_1(\mfr{n}p^r))_{/W_m},\ul{\omega}^k\otimes_W W_m)
$$
for all $r\in \Z_{>0}$.
\end{lemma}

\begin{proof}
 Recall that $\mfr{S}(\mfr{c}, \Gamma_1(\mfr{n}p^r))_{/W}=\mfr{M}(\mfr{c},\Gamma_1(\mfr{n}p^r))_{/W}[\tfrac{1}{E}]$, where $\mfr{M}(\mfr{c},\Gamma_1(\mfr{n}p^r))_{/W}$ is the moduli stack that parameterizes isomorphism classes of $\mfr{c}$-polarized HBAV together with $\Gamma_1(\mfr{n}p^r)$-level structure. This corresponds to the Igusa scheme in \cite[\S 4.1]{Hsieh}. The assertion follows from the same proof as in Lemma~4.2 of \textit{loc.~cit.}.
\end{proof}

Let $\theta:(\O\otimes \Zp)^{\times}\map W^{\times}$ be a character of finite order. We say that a form is \textit{of parallel weight $k\in \Zp$ and character $\theta$} if for any $\alpha\in (\O\otimes \Zp)^{\times}$, we have
$$
\alpha^* f(A,(\O^{\times})^+\lambda,\iota_{\mfr{n}},\iota_{p^{\infty}}):=f(A, (\O^{\times})^+ \lambda,\eta,\iota_{p^{\infty}}\circ \alpha^{-1})=\theta(\alpha) \<N(\alpha)\>^k f,
$$ 
where $N:(\O\otimes \Zp)^{\times} \map \Zp^{\times}$ is the norm map and $\<\;\>:\Zp^{\times}\map 1+p\Zp$ is the projection map.  
%We denote by $\mathcal{V}(\mfr{c},\Gamma)[k,\theta]$ the space of $p$-adic modular forms of level $\Gamma$, weight $k$, and character $\theta$. 
We denote by $H^0(T_{m,n}(\mfr{c},\Gamma)_{/W_m},\O_{T_{m,n}})[k]$ the subspace of $H^0(T_{m,n}(\mfr{c},\Gamma)_{/W_m},\O_{T_{m,n}})$ consisting of all elements $f$ satisfying
$$
\alpha^* f(A,\lambda,\eta,\iota_{p^n}):=f(A,\lambda,\eta,\iota_{p^n}\circ \alpha^{-1})=\<N(\alpha)\>^k f.
$$
From the discussion in \cite[\S 1.10]{Kat}, one has a canonical isomorphism
\begin{equation}\label{eq:key_isom}
H^0(T_{m,n}(\mfr{c},\Gamma)_{/W_m},\O_{T_{m,n}})\isom  \bigoplus_k H^0(\mfr{S}(\mfr{c},\Gamma_1(\mfr{n}p^n))_{/W_m},\ul{\omega}^k).
\end{equation}

We denote by $\mathcal{V}(\mfr{c},\Gamma)[k,\theta]$ (resp.~$\mathcal{V}(\mfr{c},\Gamma)[k]$ when $\theta=\mathbbm{1}$) the space  consisting of elements $v\in \mathcal{V}(\mfr{c},\Gamma)$ such that $\alpha\cdot v=\theta(\alpha)\<N(\alpha)\>^k v$ (resp.~$\alpha\cdot v=\<N(\alpha)\>^k v$) for all $\alpha\in (\O\otimes\Zp)^{\times}$. Similarly, for a Hecke character $\chi$ with modulus $\mfr{n}p$, we denote by $\mathcal{V}(\mfr{c},\Gamma)[\chi]$ the subspace of $\mathcal{V}(\mfr{c},\Gamma)$ on which the group $\Gamma_0(\mfr{n}p)$ acts via $\chi$. For a Hecke character $\psi$ with modulus $\mfr{n}p^n$, we defined the space $H^0(T_{m,n}(\mfr{c},\Gamma)_{/W_m},\O_{T_{m,n}})[k,\psi]$ of $H^0(T_{m,n}(\mfr{c},\Gamma)_{/W_m},\O_{T_{m,n}})[k]$ via the same manner.

One can define the Hecke action on $\mathcal{V}(\mfr{c},\Gamma)$ (see \cite[\S 4.1.10]{Hida3} for the definition). Recall that $e$ and $e^{\circ}$ are the idempotent elements attached to $U(p)$ and $T(p)$, respectively. %Note that we have $T(p)\equiv U(p)\bmod p$ and hence, $e\equiv e^{\circ} \bmod p$. 
We write $\mathcal{V}^{\ord}(\mfr{c},\Gamma)=e\cdot \mathcal{V}(\mfr{c},\Gamma)$. Let $V^{\ord}(\mfr{c},\Gamma)=\Hom(\mathcal{V}^{\ord}(\mfr{c},\Gamma), \Qp/\Zp)$ be the Pontryagin dual of $\mathcal{V}^{\ord}(\mfr{c},\Gamma)$, and let $V^{\ord}(\mfr{c},\Gamma,\chi)$ be the Pontryagin dual of $\mathcal{V}^{\ord}(\mfr{c},\Gamma)[\chi]$ for Hecke characters $\chi$ with modulus $\mfr{n}p$.

The following theorem is called the vertical control theorem in \cite[\S 4.1.8]{Hida3} in which Hida only proved the assertion for the space of cusp forms of level $\Gamma_1(\mfr{n}p)$. By using Corollary~\ref{53}, we are able to prove a theorem for the space of modular forms of parallel weight and level $\Gamma_1(\mfr{n}p^r)$.  

\begin{thm}\label{56}
Let notation be as above. Suppose that $\mfr{c}$ is prime to $\mfr{n}p$. 
\begin{enumerate}
%\item The Pontryagin dual $\Hom_W(V^{\ord}(\mfr{c},\Gamma),W)$ of $\mathcal{V}(\mfr{c},\Gamma)$ is a projective $W[[\Zp^{\times}]]$-module of finite type. 
\item If $k\geq 2$, then we have $eH^0(\mfr{S}(\mfr{c},\Gamma^r),\ul{\omega}^k) \otimes \Qp/\Zp \isom eH^0(\mfr{M}(\mfr{c},\Gamma^r),\ul{\omega}^k)\otimes \Qp/\Zp$ for $r\in \Z_{>0}$. \\
Moreover, (when $r=0$) we have $ e^{\circ}H^0(S(\mfr{c},\Gamma),\ul{\omega}^k) \otimes \Qp/\Zp \isom e^{\circ}H^0(M(\mfr{c},\Gamma),\ul{\omega}^k)\otimes \Qp/\Zp$.
%In particular, we have $eH^0(\mfr{S}(\mfr{c}, \Gamma_0^1),\ul{\omega}^k)\otimes \Qp/\Zp\isom e^{\circ}G_k(\mfr{c},\Gamma),\ul{\omega}^k)\otimes \Qp/\Zp$.

\item If $k\geq 3$, we have $\mathcal{V}^{\ord}(\mfr{c},\Gamma)[k]\isom e^{\circ}G_k(\mfr{c},\Gamma;W)\otimes \Qp/\Zp$.

\item If $k\geq 3$, we have $V^{\ord}(\mfr{c},\Gamma)\otimes_{\Lambda,k} W \isom \Hom_W(e^{\circ}G_k(\mfr{c},\Gamma;W),W)$.

\item The space $V^{\ord}(\mfr{c},\Gamma)$ is a free $\Lambda$-module of finite rank.
 
\end{enumerate}
\end{thm}

\begin{proof}
We follow the argument of \cite[Theorem 4.10]{Hida3}. 

\begin{enumerate}
\item As the argument for $eH^0(\mfr{S}(\mfr{c},\Gamma^r),\ul{\omega}^k) \otimes \Qp/\Zp$ and for $ e^{\circ}H^0(S(\mfr{c},\Gamma),\ul{\omega}^k) \otimes \Qp/\Zp$ are the same, we will only deal with the former case. Suppose that $r>0$. We write $\mfr{S}=\mfr{S}(\mfr{c},\Gamma^r)$ and $\mfr{M}=\mfr{M}(\mfr{c},\Gamma^r)$ for simplicity. Since $\mfr{S}$ is an open subscheme of $\mfr{M}$, we have an embedding
$$
\epsilon_m:e H^0(\mfr{M}_{/W}, \ul{\omega}^k)\otimes W_m\embed eH^0(\mfr{S}_{/W}, \ul{\omega}^k)\otimes W_m
$$
for $m\in \Z_{> 0}$. We claim that $\epsilon_m$ is an isomorphism. Given any $\ol{f}\in eH^0(\mfr{S}_{/W}, \ul{\omega}^k)\otimes W_m$, let $f\in eH^0(\mfr{S}_{/W}, \ul{\omega}^k)$ be such that $f\equiv \ol{f} \bmod p^m$. Since $\mfr{S}=\mfr{M}[\frac{1}{E}]$, we have $H^0(\mfr{S}_{/W}, \ul{\omega}^k)=\varinjlim_n H^0(\mfr{M}_{/W},\ul{\omega}^{k+na(p-1)})/E^n$ (see \S 3.3.2 in \textit{loc.~cit.}). Thus $E^m f$ belongs to  $H^0(\mfr{M}_{/W}, \ul{\omega}^{k+ma(p-1)})$ for some $m\in \Z_{>0}$, and we have $e(E^mf)\in eH^0(\mfr{M}_{/W},\ul{\omega}^{k+ma(p-1)})$.  Moreover, we have
$$
e(E^m f)\equiv E^m(ef)=E^m f \bmod p^m.
$$
Let $K$ be the quotient field of $W$. Since the dimension of $eH^0(\mfr{M}(\mfr{c},\Gamma^r)_{/W},\ul{\omega}^k)\otimes K$ is bounded independent of $k$ for $k\geq 2$ (see the proof of Theorem~4.9 in \textit{loc.~cit.}), we have an isomorphism $e H^0(\mfr{M},\ul{\omega}^k)\otimes K \isom e H^0(\mfr{M}, \ul{\omega}^{k+ma(p-1)})\otimes K$ induced by multiplying by $E^m$. Therefore, there exists $g\in e H^0(\mfr{M}_{/W}, \ul{\omega}^k)$ and $l\in \Z_{\geq 0}$ such that $p^l\cdot e(E^mf)=E^mg$. Furthermore, we have
$$
p^l e(E^mf)\equiv p^l E^mf\equiv E^mg \bmod p^{l+m}.
$$
Since $E^m\equiv 1 \bmod p$ (see (\ref{eq:Hasse})), we have $p^l f\equiv g\bmod p^{l+m}$, and hence, $g\in p^l e H^0(\mfr{M}_{/W}, \ul{\omega}^k)$. We know that $f\equiv p^{-l}g \bmod p^m$. Thus $\epsilon_m$ is an isomorphism for all $m\in \Z_{>0}$. Since injective limit is an exact functor, the assertion follows.

\item 
If $\mfr{n}$ is divisible by a positive integer $N\geq 3$ such that $\mfr{M}(\mfr{c},\Gamma_1(\mfr{n}))$ is a fine moduli scheme, then for $k\geq 3$, we have
\[
\begin{split}
\mathcal{V}^{\ord}(\mfr{c},\Gamma)[k]
&= \varinjlim_m \varinjlim_n eH^0(T_{m,n}, \O_{T_{m,n}})[k]
=^{(i)}\varinjlim_m eH^0(T_{m,1}, \O_{T_{m,1}})[k]\\
&\isom^{(ii)}eH^0(S(\mfr{c}, \Gamma\cap \Gamma_0(p)),\ul{\omega}^k\otimes \Qp/\Zp)
\isom^{(iii)} e^{\circ} H^0(S(\mfr{c},\Gamma),\ul{\omega}^k\otimes \Qp/\Zp)\\
&\isom^{(iv)}e^{\circ}H^0(S(\mfr{c},\Gamma),\ul{\omega}^k)\otimes \Qp/\Zp \isom^{(v)}e^{\circ}G_k(\mfr{c},\Gamma;W)\otimes \Qp/\Zp.
\end{split}
\]
Note that the equality (i) follows from the fact that the Hecke operator $U(p)$ sends each modular form of level $\Gamma_1(p^n)$ to a modular form of level $\Gamma_1(p^{n-1})$ for all $n\geq 2$ (see p.~121 in \textit{loc.~cit.}). The isomorphism (ii) follows from (\ref{eq:key_isom}). The isomorphism (iii) is obtained by the fact that every $p$-ordinary modular form of level $\Gamma\cap\Gamma_0(p)$ and weight $k\geq 3$ is old at $p$. The isomorphism (iv) follows from Corollary~\ref{53}. The isomorphism (v) follows from part (1) of Theorem~\ref{56} and (\ref{eq:geom_form}).

When $\mfr{n}$ is not divisible by any positive integer $N\geq 3$ or $\mfr{M}(\mfr{c},\Gamma_1(\mfr{n}))$ is not a fine moduli, the above computation works except for the equality (ii). To prove this equality holds, we choose a prime number $l$ prime to $p$ such that $p\nmid l-1$ and $\mfr{M}(\mfr{c},\Gamma_1(\mfr{n}l))$ is a fine moduli scheme. Thus, Corollary~\ref{53} holds for $\Gamma_1(\mfr{n}l)$, and hence holds for $\Gamma_1(\mfr{n})\cap \Gamma_0(l)$ since $p$ does not divide $l-1$. Moreover, one has an injective homomorphism 
$$
G_k(\Gamma_1(\mfr{n});\Zp)\embed G_k(\Gamma_1(\mfr{n})\cap\Gamma_0(l);\Zp)
$$
induced by $l$-stabilization. Therefore, the above equality (ii) holds, and hence, the assertion follows.

\item  Since $V^{\ord}(\mfr{c},\Gamma)$ is the Pontryagin dual of $\mathcal{V}^{\ord}(\mfr{c},\Gamma)$, we have
\[
\begin{split}
V^{\ord}(\mfr{c},\Gamma) \otimes_{\Lambda,k} W
&=V^{\ord}(\mfr{c},\Gamma)/(T-u^k+1)V^{\ord}
\isom \Hom(\mathcal{V}^{\ord}(\mfr{c},\Gamma)[k],\Qp/\Zp)\\
&\isom \Hom(e^{\circ}G_k(\mfr{c},\Gamma;W)\otimes \Qp/\Zp,\Qp/\Zp)
\isom\Hom(e^{\circ}G_k(\mfr{c},\Gamma;W),W).
\end{split}
\]
Note that the second isomorphism follows from the Theorem~\ref{56}(2).

\item As $G_k(\mfr{c},\Gamma;W)$ is a free $W$-module of finite rank, by Theorem~\ref{56}(3), the last assertion follows from the well-known lemma, Lemma~\ref{652}. 
\end{enumerate}
\end{proof}

\begin{lemma}\label{652}
Let $M$ be a finitely generated torsion-free $\Lambda$-module. If $M/PM$ is free for infinitely many height $1$ prime ideals $P\subset\Lambda$, then $M$ is free $\Lambda$-module of finite rank.
\end{lemma}

Recall that $\mu_{p^{\infty}}$ is the group of all $p$-power roots of unity. Also, recall that $\rho=\rho_{\zeta}$ is the character associated with the $p^{r-1}$th root of unity $\zeta$ defined in Section~\ref{sec:41}. The next lemma is a key result to prove the control theorem (Corollary~\ref{655}).

\begin{lemma}\label{58}
Let $\chi$ be a narrow ray class character with modulus $\mfr{n}p$. Assume that $W$ contains the values of $\chi$ and $\mu_{p^{\infty}}$ and that $\mfr{c}$ is prime to $\mfr{n}p$. If $k\geq 2$ and $\zeta\in \mu_{p^{\infty}}$, then we have 
$$
V^{\ord}(\mfr{c},\Gamma_1(\mfr{n}),\chi)\otimes_{W[[\Zp^{\times}]],k,\rho_{\zeta}} W \isom \Hom_W(G_k^{\ord}(\mfr{c},\mfr{n}p^r, \chi\omega^{-k}\rho_{\zeta};W),W).
$$
\end{lemma}

\begin{proof}
For simplicity, we will write $\rho=\rho_{\zeta}$. We follow the argument as in \cite[Corollary 4.23]{Hsieh}. By the same trick as in the proof of Theorem~\ref{56}(2), we may assume that the integral ideal $\mfr{n}$ is divisible by a positive integer $N\geq 3$ and $\mfr{M}(\mfr{c},\Gamma_1(\mfr{n}))$ is a fine moduli scheme. For simplicity, we write $V^{\ord}(\mfr{c},\Gamma_1(\mfr{n}),\chi)$ as $V^{\ord,\chi}$ and write $\mathcal{V}^{\ord}(\mfr{c},\Gamma_1(\mfr{n}))$ as $\mathcal{V}^{\ord}$. Let $K$ be the quotient field of $W$. Since the conductor of the character $\chi\omega^{-k}\rho$ is $\mfr{n}p^r$, for $k\geq 2$, we have
$$
\mathcal{V}^{\ord}[k,\chi\omega^{-k}\rho] 
= \varinjlim_m \varinjlim_n eH^0(T_{m,n},\O_{T_{m,n}})[k,\chi\omega^{-k}\rho]
= \varinjlim_m eH^0(\mfr{S}(\mfr{c},\Gamma_1(\mfr{n}p^r))_{/W_m},\ul{\omega}^k)^{(\chi\omega^{-k}\rho)},
$$
where the last term is the subspace of $eH^0(\mfr{S}(\mfr{c},\Gamma_1(\mfr{n}p^r))_{/W_m},\ul{\omega}^k)$ on which the group $\Gamma_0(\mfr{n}p^r)$ acts via $\chi\omega^{-k}\rho$. Let $G_m=eH^0(\mfr{S}(\mfr{c},\Gamma_1(\mfr{n}p^r))_{/W_m},\ul{\omega}^k)^{(\chi\omega^{-k}\rho)}$, and let $G=eH^0(\mfr{M}(\mfr{c},\Gamma_1(\mfr{n}p^r))_{/W},\ul{\omega}^k)^{(\chi\omega^{-k}\rho)}$. Let $C_m$ be the cokernel of the embedding $G \otimes W_m\embed G_m$, which exists as $\mfr{S}(\mfr{c},\Gamma_1(\mfr{n}p^r))_{/W}$ is an open subscheme of $\mfr{M}(\mfr{c},\Gamma_1(\mfr{n}p^r))_{/W}$. Taking the injective limit, we obtain a short exact sequence
$$
0\map G\otimes \Qp/\Zp\map \varinjlim G_m\map \varinjlim C_m \map 0.
$$
We set $C=\varinjlim C_m$. Since $V^{\ord}(\mfr{c},\Gamma_1(\mfr{n}))$ is the Pontryagin dual of $\mathcal{V}^{\ord}(\mfr{c},\Gamma_1(\mfr{n}))$, by taking the Pontryagin dual of the above short exact sequence, we obtain
$$
0\map C^*\map V^{\ord,\chi}/(T-\rho(u)u^k+1)V^{\ord,\chi}\map \Hom_W(G,W)\map 0.
$$
We claim that $C^*=0$, which implies the assertion since $G=G_k^{\ord}(\mfr{c},\mfr{n}p^r, \chi\omega^{-k}\rho_{\zeta};W)$ by (\ref{eq:geom_form}). Since $V^{\ord}$ is free $\Lambda$-module of finite rank by Theorem \ref{56}(4), $V^{\ord,\chi}$ is also free $\Lambda$-module of finite rank, and hence, $C^*$ is a torsion-free $W$-module. To show $C^*=0$, it suffices to show that $C^*$ is a torsion $W$-module. Indeed, we will show that $|(\Z/p^r\Z)^{\times}| G_m\subset G\otimes W_m$ for all $m\in \Z_{>0}$. Given any $f_m\in G_m$, by Lemma~\ref{base_change_prop}, there exists $f\in G$ such that $f\equiv f_m \bmod p^m$. Let $e_{\rho}=\sum_{\gamma\in (\Z/p^r\Z)^{\times}} \rho(\gamma)\cdot \gamma^{-1}$. Since $f\in G$, we have $\gamma\cdot f=\rho(\gamma)f$ for all $\gamma\in (\Z/p^r\Z)^{\times}$, and hence, we have 
$$
 |(\Z/p^r\Z)^{\times}|f_m \equiv |(\Z/p^r\Z)^{\times}|f =e_{\rho}f \in M\otimes W_m.\qedhere
$$ 
\end{proof}

Next, we define families of $p$-adic modular forms following \cite[\S 3.3.4]{Hida3}. For $\mu\in t_{\lambda}^+\cup\{0\}$, let $$c_{\lambda}(\mu):\mathcal{V}^{\ord}(t_{\lambda}\mfr{D},\Gamma_1(\mfr{n}))\map \Qp/\Zp$$ be the linear map associating to $f$ its $e^{2\pi i \tr(\mu z)}$-coefficient. Then $c_{\lambda}(\mu)\in V^{\ord}(t_{\lambda}\mfr{D},\Gamma_1(\mfr{n}))$ for all $\mu$. Let $\chi$ be a narrow ray class character with modulus $\mfr{n}p$, and let 
\begin{equation}\label{eq:geom_family}
G^{\ord}(\mfr{n}, \chi;\Lambda)=\bigoplus_{\lambda=1}^{h_F^+} \Hom_{\Lambda}(V^{\ord}(t_{\lambda}\mfr{D}, \Gamma_1(\mfr{n}),\chi),\Lambda).
\end{equation}
To each $\mathcal{F}\in G^{\ord}(\mfr{n},\chi;\Lambda)$, we associate its Fourier coefficients 
$$
\left\{\begin{array}{cc}
      C(\mathfrak{a},\mathcal{F}) & \mbox{for all nonzero integral ideals}\; \mathfrak{a}\; \mbox{of}\; \O_F \\
      C_{\lambda}(0,\mathcal{F}) & \lambda=1,\ldots, h^+_F
      \end{array}\right \}.
$$
Here for an integral ideal $\mfr{a}$, we have $C(\mfr{a},\mathcal{F})=N(t_{\lambda}\mfr{D})^{-k/2}\cdot \mathcal{F}(c_{\lambda}(\mu))$ form some $\mu \in t_{\lambda}\mfr{D}$ satisfying $\mfr{a}=(\mu)(t_{\lambda}\mfr{D})^{-1}$, and similarly, we have $C_{\lambda}(0,\mathcal{F})=N(t_{\lambda}\mfr{D})^{-k/2}\cdot \mathcal{F}(c_{\lambda}(0))$. The following theorem shows that $G^{\ord}(\mfr{n}, \chi;\Lambda)$ and ${M'}^{\ord}(\mfr{n},\chi;\Lambda)$ are isomorphic. Recall that the space ${M'}^{\ord}(\mfr{n},\chi;\Lambda)$ of $\Lambda$-adic modular forms is defined in Section 4.

\begin{thm}\label{59}
Let the notation be as above. Suppose that the assumption in Lemma~\ref{58} holds. Then we have an isomorphism of $\Lambda$-modules
$$
G^{\ord}(\mfr{n}, \chi;\Lambda)\isom {M'}^{\ord}(\mfr{n},\chi;\Lambda).
$$
In particular, ${M'}^{\ord}(\mfr{n},\chi;\Lambda)$ is a free $\Lambda$-module of finite rank.
\end{thm}

\begin{proof}
We follow the argument as in \cite[Theorem 4.25]{Hsieh}. By the $q$-expansion principle, we have a natural embedding
$$
G^{\ord}(\mfr{n}, \chi;\Lambda)\embed {M'}^{\ord}(\mfr{n},\chi;\Lambda).
$$
Recall that we denote by $Q(\Lambda)$ (resp.~$K$) the fraction field of $\Lambda$ (resp.~$W$). It follows from the proof of \cite[Theorem 1.2.2]{W2} that $\dim_{Q(\Lambda)}{M'}^{\ord}(\mfr{n},\chi;\Lambda)\otimes_{\Lambda} Q(\Lambda)$ is less than or equal to $\rank_W M^{\ord}_k(\mfr{n}p,\chi\omega^{-k},W)$. By Lemma~\ref{58} and (\ref{eq:geom_family}), we have
\[
\begin{split}
G^{\ord}(\mfr{n}, \chi;\Lambda)/(T-u^k+1)
\isom & \bigoplus_{\lambda=1}^{h_F^+} \Hom_W(\Hom_W(G_k^{\ord}(t_{\lambda}\mfr{D},\mfr{n}p, \chi\omega^{-k};W),W),W)\\
\isom & \bigoplus_{\lambda=1}^{h_F^+} G_k^{\ord}(t_{\lambda}\mfr{D},\mfr{n}p, \chi\omega^{-k};W).
\end{split}
\]
By (\ref{eq:adelic_geom_form}), one can identify the space $\bigoplus_{\lambda=1}^{h_F^+} G_k^{\ord}(t_{\lambda}\mfr{D},\mfr{n}p, \chi\omega^{-k};W)$ with the space $G_k^{\ord}(\mfr{n}p,\chi\omega^{-k},W)$ and hence, one has $\dim_K G_k^{\ord}(\mfr{n}p,\chi\omega^{-k},K)=\dim_{Q(\Lambda)} G^{\ord}(\mfr{n},\chi;\Lambda)\otimes_{\Lambda} Q(\Lambda)$. Moreover, from the discussion in Section~\ref{sec:52}, we know that $\dim_K M_k^{\ord}(\mfr{n}p,\chi\omega^{-k},K)=\dim_K G_k^{\ord}(\mfr{n}p,\chi\omega^{-k},K)$. Therefore, we obtain an isomorphism
$$
G^{\ord}(\mfr{n}, \chi;\Lambda)\otimes_{\Lambda} Q(\Lambda) \isom {M'}^{\ord}(\mfr{n},\chi;\Lambda)\otimes_{\Lambda} Q(\Lambda).
$$
Let $\{\mathcal{F}_1,\ldots, \mathcal{F}_s\}\subset {M'}^{\ord}(\mfr{n}, \chi;\Lambda)$ be a basis of ${M'}^{\ord}(\mfr{n}, \chi;\Lambda)\otimes_{\Lambda} Q(\Lambda)$. Thus, for every element $\mathcal{F}\in {M'}^{\ord}(\mfr{n}, \chi;\Lambda)$, one has $\mathcal{F}=\sum_{i=1}^s x_i \mathcal{F}_i$ for some $x_i\in Q(\Lambda)$. For any $s$ positive integral ideals $\mfr{a}_1,\ldots,\mfr{a}_s$, we have an equation of matrices $$AX=B$$ for $A=(C(\mfr{a}_i,\mathcal{F}_j))$, $X=(x_1,\ldots,x_s)^t$, and $B=(C(\mfr{a}_i,\mathcal{F})^t)$. Here $A^t$ is the transpose of a matrix $A$. Since $\{\mathcal{F}_1,\ldots, \mathcal{F}_s\}\subset {M'}^{\ord}(\mfr{n}, \chi;\Lambda)$ is a basis of ${M'}^{\ord}(\mfr{n},\chi;\Lambda)\otimes_{\Lambda}Q(\Lambda)$, one can pick integral ideals $\mfr{a}_1,\ldots,\mfr{a}_s$ of $F$ such that $a=\det A \neq 0\in \Lambda$. By multiplying the adjugate matrix of $A$ on both sides, we see that $a\mathcal{F}\in \Lambda\cdot \mathcal{F}_1+\cdots+\Lambda\cdot \mathcal{F}_s$. Therefore, we have $a{M'}^{\ord}(\mfr{n},\chi;\Lambda)\subset \mathcal{F}_1+\cdots+\Lambda\cdot \mathcal{F}_s$. In particular, ${M'}^{\ord}(\mfr{n},\chi;\Lambda)$ is a finitely generated $\Lambda$-module since $\Lambda$ is Noetherian. By the same argument as in \cite[p.~210]{Hida2}, we know that ${M'}^{\ord}(\mfr{n},\chi;\Lambda)/(T-u^{k}+1)$ is a free $W$-module of finite rank for almost all $k\in \Z_{\geq 2}$, and hence ${M'}^{\ord}(\mfr{n},\chi;\Lambda)$ is a free $\Lambda$-module of finite rank by Lemma~\ref{652}.

We set $N={M'}^{\ord}(\mfr{n},\chi;\Lambda)/G^{\ord}(\mfr{n}, \chi;\Lambda)$. Then $N$ is a torsion $\Lambda$-module. To prove the assertion, we will show that $N=0$ by showing that $N$ is a flat $\Lambda$-module. Let $\kappa$ be the residue field of $\Lambda$. Since ${M'}^{\ord}(\mfr{n},\chi;\Lambda)$ is a free $\Lambda$-module, we obtain the exact sequence
$$
0\map \Tor^1(N,\kappa)\map G^{\ord}(\mfr{n},\chi;\Lambda)\otimes \kappa
\xrightarrow{\iota} {M'}^{\ord}(\mfr{n},\chi;\Lambda)\otimes \kappa\map N\otimes \kappa\map 0.
$$
By the $q$-expansion principle again, $\iota$ is injection, and hence, $\Tor^1(N,\kappa)=0$. It follows that $N$ is a flat $\Lambda$-module since $\Lambda$ is a local Noetherian ring.
\end{proof}

\begin{cor}[Control theorem]\label{655}
Let the notation and the assumptions be as in Lemma~\ref{58}. Then for each integer $k\geq 2$, we have an isomorphism
$$
M^{\ord}(\mfr{n},\chi;\Lambda)/(T-\rho(u)u^{k-2}+1)\isom M_k^{\ord}(\mfr{n}p^r, \chi\omega^{2-k}\rho;W).
$$
Moreover, $M^{\ord}(\mfr{n},\chi;\Lambda)$ is free $\Lambda$-module of finite rank.
\end{cor}

\begin{proof}
We follow the argument in \cite[Theorem 3.8]{Hida3}. 
By Lemma~\ref{512}, Lemma~\ref{58} and Theorem~\ref{59}, we have a series of isomorphisms
\[
\begin{split}
M^{\ord}(\mfr{n},\chi;\Lambda)/(T-\rho(u)u^{k-2}+1)\isom
&{M'}^{\ord}(\mfr{n},\chi\omega^2;\Lambda)/(T-\rho(u)u^{k}+1)\mbox{ (Lemma~\ref{512})}\\
\isom &  G^{\ord}(\mfr{n},\chi\omega^2;\Lambda)/(T-\rho(u)u^{k}+1) \mbox{ (Theorem~\ref{59})}\\
\isom &\left(\bigoplus_{\lambda=1}^{h^+_F}\Hom_{\Lambda}(V^{\ord}(t_{\lambda}\mfr{D}, \mfr{n},\chi\omega^2),\Lambda)\right)\otimes_{W[[\Zp^{\times}]],k,\rho} W \mbox{ (by (\ref{eq:geom_family}))}\\
\isom & \bigoplus_{\lambda=1}^{h^+_F}\left(\Hom_{\Lambda}(V^{\ord}_{\Lambda}(t_{\lambda}\mfr{D}, \mfr{n},\chi\omega^2),\Lambda)\otimes_{W[[\Zp^{\times}]],k,\rho} W\right)\\
\isom & \bigoplus_{\lambda=1}^{h^+_F}\Hom_W(V^{\ord}_{\Lambda}(t_{\lambda}\mfr{D}, \mfr{n},\chi\omega^2)\otimes_{W[[\Zp^{\times}]],k, \rho} W,W)\\
\isom & \bigoplus_{\lambda=1}^{h^+_F}\Hom_W(\Hom_W(G^{\ord}_k(t_{\lambda}\mfr{D}, \mfr{n}p^r,\chi\omega^{2-k}\rho ;W),W),W) \mbox{ (Lemma~\ref{58})} \\
\isom & M^{\ord}_k(\mfr{n}p^r,\chi\omega^{2-k}\rho ;W) \mbox{ by (\ref{eq:ord_forms})}.
\end{split}
\]
The second assertion follows from Lemma~\ref{652}.
\end{proof}

%%%%%%%%%%%%%%%%%%%%%%%%%%%%%%%%%%%%%%%%%%%%%%%%%%%%%%%%%%%
\section{Main results}\label{sec:06}
Let $p$ be an odd rational prime unramified in $F$, and let $\mfr{n}$ be an integral ideal prime to $p$.  Let $\O_{\infty}$ be as in Section~\ref{sec:04}, and let $\Lambda_{\infty}=\O_{\infty}[[T]]$. For simplicity, we write $M_{\Lambda_{\infty}}=M^{\ord}(K_1(\mfr{n});\Lambda_{\infty})$ and write $S_{\Lambda_{\infty}}$ in the same manner. We denote by $\Lambda_{\infty}[C_{\mfr{n}p}^*]$ the free abelian group generated by $C_{\mfr{n}p}^*$ over $\Lambda_{\infty}$, where $C_{\mfr{n}p}^*$ was defined in Section~\ref{sec:24}. Moreover, it follows from Theorem~\ref{ordcusp} that $e\cdot \Lambda_{\infty}[C_{\mfr{n}p}^*]$ is a free $\Lambda$-module.

%%%%%%%%%%%%%%%%%%%%%%%%%%%%%%%%%%%%%%%%%%%%%%%%%%%%%%%%%%
\subsection{Proof of main results: Part $1$}\label{sec:61}
Define a map  $C_0: M(K_1(\mfr{n});\Lambda_{\infty})  \rightarrow \Lambda_{\infty}[C_{\n p}^*]$ by sending $\mathcal{F} \in M(K_1(\mfr{n});\Lambda_{\infty})$ to $\sum_{\gamma\in C_{\mfr{n}p}} C_{\gamma}(0,\mathcal{F})\cdot I_{[\gamma]}$, where $C_{\gamma}(0,\mathcal{F})$ is the constant term of $\mathcal{F}$ at the cusp $\gamma$. It follows from Weierstrass preparation theorem that every element in $\Lambda_{\infty}$ only has finitely many zeros in $\C_p$. Hence, it follows from Definition~\ref{511}(2) that a $\Lambda$-adic modular form $\mathcal{F}\in M(K_1(\mfr{n});\Lambda_{\infty})$ is a cusp form if and only if $C_0(\mathcal{F})=0$. Therefore, we obtain a left exact sequence of $\Lambda$-modules
$$
0\map S(K_1(\mfr{n});\Lambda_{\infty})\map M(K_1(\mfr{n});\Lambda_{\infty}) \xrightarrow{C_0} \Lambda_{\infty}[C_{\n p}^*].
$$
Since ordinary subspaces of the the above spaces are the largest algebra direct summand on which $U(\mfr{p})$ acts as a unit for all $\mfr{p}|p$ (see \cite[\S 7.2]{Hida2}), taking the Hida's idempotent element $e$ is an exact  functor, and hence, we obtain the following exact sequence of flat $\Lambda$-modules 
%by taking ordinary projection at the above sequence 
$$
0\map S_{\Lambda_{\infty}}\map M_{\Lambda_{\infty}} \xrightarrow{C_0}e\cdot \Lambda_{\infty}[C_{\n p}^*].
$$
The following theorem shows that the map $C_0$ is surjective on the space $M_{\Lambda_{\infty}}$.

\begin{thm}\label{711}
Let the notation be as above. Then the map $C_0:M_{\Lambda_{\infty}} \rightarrow e\cdot \Lambda_{\infty}[C_{\n p}^*]$ is surjective. In particular, we have a short exact sequence of flat $\Lambda$-modules
$$
0\map S_{\Lambda_{\infty}}\map M_{\Lambda_{\infty}} \xrightarrow{C_0}e\cdot \Lambda_{\infty}[C_{\n p}^*]\map 0
$$
\end{thm}

\begin{proof}
We follow the argument of \cite[Theorem 4.26]{Hsieh}. For a free $\Lambda_{\infty}$-module $V$ and for a character $\psi$ of $G_F$, we will denote by $V^{(\psi)}$ the $\psi$-eigenspace of $V$. We have $M_{\Lambda_{\infty}}=\oplus_{\chi} M_{\Lambda_{\infty}}^{(\chi)}$ and $e\cdot \Lambda_{\infty}[C_{\n p}^*]=\oplus_{\chi} e\cdot \Lambda_{\infty}[C_{\n p}^*]^{(\chi)}$, where the direct sums run through all narrow ray class characters with modulus $p$. To show that $C_0$ is surjective, it suffices to show that the induced map $C_0: M_{\Lambda_{\infty}}^{(\chi)}\map e\cdot \Lambda_{\infty}[C_{\mfr{n}p}^*]^{(\chi)}$ is surjective for all characters $\chi$.
 
Let $\mfr{P}$ be the maximal ideal of $\O_{\infty}$, and let $\bbF=\O_{\infty}/\mfr{P}$ be its residue field. By Nakayama's lemma, it suffices to show that the map
$$
M_{\Lambda_{\infty}}^{(\chi)}/(\mfr{P},T) \xrightarrow{C_0}e\cdot \Lambda_{\infty}[C_{\n p}^*]^{(\chi)}/(\mfr{P},T)
$$
is surjective. By Theorem~\ref{56} and Corollary~\ref{655}, for $k\geq 3$ such that $\chi\omega^{2-k}$ is a trivial character, we have  
$$
M_{\Lambda_{\infty}}^{(\chi)}/(\mfr{P},T)\isom M_k^{\ord}(\mfr{n};\O_{\infty})\otimes_{\O_{\infty}} \bbF\isom eH^0(S^*_{/\O_{\infty}},\pi_*(\ul{\omega}^k))\otimes_{\O_{\infty}} \bbF.
$$
Moreover, we have
$$
e\cdot \Lambda_{\infty}[C_{\n p}^*]^{(\chi)}/(\mfr{P},T) \isom e\cdot \O_{\infty}[C_{\mfr{n}}^*]\otimes_{\O_{\infty}} \bbF =eH^0(\partial S^*_{/\O_{\infty}}, \pi_*(\ul{\omega}^k)) \otimes_{\O_{\infty}} \bbF.
$$
Since $S^*$ is affine and $\pi_*(\ul{\omega}^k)$ is an invertible sheaf, the map 
$$
C_0: eH^0(S^*_{/\O_{\infty}},\pi_*(\ul{\omega}^k)) \otimes_{\O_{\infty}} \bbF\map eH^0(\partial S^*_{/\O_{\infty}}, \pi_*(\ul{\omega}^k)) \otimes_{\O_{\infty}} \bbF
$$
is surjective. This shows that $C_0$ is surjective on each $\chi$-component.
\end{proof}

We now fix two primitive narrow ray class characters $\chi_1$ and $\chi_2$ of conductors $\n_1$ and $\n_2$, respectively. We assume that $\chi_1$ is not a trivial character and $(\chi_1,\chi_2)\neq (\omega^2,\mathbbm{1})$. In addition, we assume that $\n_1\n_2=\n$ or $\n p$ and $\n_2$ is prime to $p$.
Let $\Lambda=\Zp[\chi_1,\chi_2][[T]]$. It is known that $\Lambda_{\infty}$ is a faithfully flat $\Lambda$-module \cite[Lemma 2.1.1]{Ohta3}. Since the short exact sequence in Lemma~\ref{711} can be obtained by tensoring with $\Lambda_{\infty}$ over $\Lambda$, we obtain a short exact sequence of $\Lambda$-modules
\begin{equation}\label{FundSeq}
0\map S^{\ord}(K_1(\mfr{n});\Lambda)\map M^{\ord}(K_1(\mfr{n});\Lambda) \xrightarrow{C_0}e\cdot \Lambda[C_{\n p}^*]\map 0.
\end{equation}

Recall that we denote by $\mfr{M}(\chi_1,\chi_2)$ the maximal ideal of $\mathcal{H}^{\ord}(\mfr{n},\chi_1\chi_2;\Lambda)$ containing the Eisenstein ideal $\mathcal{I}(\chi_1,\chi_2)$ (see Definition~\ref{524}).

\begin{lemma}\label{712}
Assume that (\ref{eq:exception}) holds. Then $e\cdot \Lambda[C_{\n p}^*]_{\mfr{M}}$ is a free $\Lambda$-module of rank one. Moreover, the Hecke operator $U({\p})$ acts on $e\cdot \Lambda_{\infty}[C_{\n p}^*]_{\mfr{M}(\chi_1,\chi_2)}$ by multiplication by $\wt{\chi}_1(\p)N(\p)+\chi_2(\mfr{p})$ for all prime ideals $\p$ and $S(\q)$ acts on it by multiplication $\wt{\chi_1}\chi_2(\mfr{q})$ for all prime ideals $\q$ not dividing $\n p$.
\end{lemma}

\begin{proof}
We follow the argument in \cite[Proposition 3.1.2]{Ohta}. We will write $\mfr{M}=\mfr{M}(\chi_1,\chi_2)$ in this proof for simplicity. Note that $e\cdot\Lambda[C_{\n p}^*]_{\mfr{M}}$ is a free $\Lambda$-module, since $e\cdot\Lambda[C_{\n p}^*]$ is a free $\Lambda$-module (Theorem~\ref{ordcusp}). Suppose the rank of $e\cdot\Lambda[C_{\n p}^*]_{\mfr{M}}$ is greater or equal to two. Then $e\cdot \Lambda[C_{\n p}^*]_{\mfr{M}} \otimes_{\Lambda} Q(\Lambda)$ is a $Q(\Lambda)$-vector space of dimension greater than one. Then there exists an Eisenstein series $\mathcal{E}(\theta,\psi)\in M_{\Lambda}$ with $\theta\psi=\chi_1\chi_2$ and $(\theta,\psi)\neq (\chi_1,\chi_2)$. By Proposition~\ref{525}, we know that $\chi_1\omega\chi_2^{-1}(\mfr{p})=1$ for all prime ideals $\mfr{p}|p$, which contradicts to the assumption of (\ref{eq:exception}). Therefore, $e\cdot\Lambda[C_{\n p}^*]_{\mfr{M}}$ is free of rank one.
The second assertion follows from the fact that the map $C_0$ is Hecke-equivariant (Lemma~\ref{HeckeEquiv}). 
\end{proof}

When (\ref{eq:exception}) holds, it follows from the above lemma that the first short exact sequence in (\ref{eq:1}) is obtained by taking localization on (\ref{FundSeq}) at the maximal ideal $\mfr{M}(\chi_1,\chi_2)$ of $\mathcal{H}^{\ord}(\mfr{n},\chi_1\chi_2;\Lambda)$. In general, if (\ref{eq:exception}) does not hold, then $e\cdot \Lambda[C_{\n p}^*]_{\mfr{M}(\chi_1,\chi_2)}$ is a free $\Lambda$-module of rank greater than $1$. This means that the space $M^{\ord}(K_1(\mfr{n});\Lambda)_{\mfr{M}(\chi_1,\chi_2)}$ contains Eisenstein series other than $\mathcal{E}(\chi_1,\chi_2)$. We will define a subspace $\wt{M}_{\Lambda,(\chi_1,\chi_2)}$ of $M^{\ord}(K_1(\mfr{n});\Lambda)$ that does not contain any Eisenstein series except for $\mathcal{E}(\chi_1,\chi_2)$ as follows. From now on, we do not assume the condition (\ref{eq:exception}). Set $$A(\chi_1,\chi_2):=\prod_{\substack{\mfr{q}|\mfr{n}\\\mfr{q}\nmid \cond(\chi_1\chi_2^{-1})}}(1-\chi_1\chi_2^{-1}(\mfr{q})(1+T)^{-s(\mfr{q})}N(\mfr{q})^{-2})
\wh{G}_{\chi_1\chi_2^{-1}}(T),$$ 
where $s(\mfr{q})\in \Zp$ was defined in Section~\ref{sec:41} and $\wh{G}_{\chi_1\chi_2^{-1}}(T)$ was defined by (\ref{eq:wh_G}). By Proposition~\ref{523}, one sees that $A(\chi_1,\chi_2)$ is the common factor of the constant terms of $\mathcal{E}(\chi_1,\chi_2)$ at different cusps. It follows from the surjectivity of $C_0$ that there exists $\mathcal{G}_0:=\mathcal{G}_0(\chi_1,\chi_2)\in M^{\ord}(K_1(\mfr{n});\Lambda)$ such that $C_0(\mathcal{E}(\chi_1,\chi_2))=A(\chi_1,\chi_2) C_0(\mathcal{G}_0)$. The subspace of $e\cdot \Lambda[C_{\n p}^*]$ generated by $C_0(\mathcal{G}_0)$ is a free $\Lambda$-module of rank $1$, say $\Lambda\cdot c_{\infty}$ for some $c_{\infty}\in e\cdot \Lambda[C_{\n p}^*]$ such that $C_0(\mathcal{G}_0)=c_{\infty}$. We denote by $\wt{M}_{\Lambda}:=\wt{M}_{\Lambda,(\chi_1,\chi_2)}$ the preimage $C_0^{-1}(\Lambda\cdot c_{\infty})$ of $\Lambda\cdot c_{\infty}$ and set $\wt{S}_{\Lambda}= S^{\ord}(K_1(\mfr{n});\Lambda)$. Then, we obtain a short exact sequence of flat $\Lambda$-modules
\begin{equation}\label{eq:25}
0\map \wt{S}_{\Lambda} \map  \wt{M}_{\Lambda} \xrightarrow{C_0} \Lambda\cdot c_{\infty}\map 0
\end{equation}
Let $K$ be a field extension of $\Qp$. Since for all $k\in \Z_{>0}$, the space $M_k(K_1(\mfr{n}p);K)$ is a direct sum of the space of cusp forms and the space generated by Eisenstein series, we know that over $Q(\Lambda)$, the space of $\Lambda$-adic modular forms is a direct sum of the space of $\Lambda$-adic cusp forms and the space generated by $\Lambda$-adic Eisenstein series. From this and the definition of $\wt{M}_{\Lambda}$, we have 
\begin{equation}\label{eq:splitting}
\wt{M}_{\Lambda}\otimes_{\Lambda} Q(\Lambda)=(\wt{S}_{\Lambda}\otimes_{\Lambda} Q(\Lambda))\oplus (\mathcal{E}_{\Lambda}\otimes_{\Lambda} Q(\Lambda)),
\end{equation}
where $\mathcal{E}_{\Lambda}$ is the $\Lambda$-module generated by $\mathcal{E}(\chi_1,\chi_2)$.
Let 
$$
s:\Lambda\cdot c_\infty \otimes_{\Lambda} Q(\Lambda) \map \wt{M}_{\Lambda}\otimes_{\Lambda} Q(\Lambda)
$$ 
be the unique, up to a scalar, Hecke-equivariant splitting map whose image in $\wt{M}_{\Lambda}\otimes_{\Lambda} Q(\Lambda)$ is $\mathcal{E}_{\Lambda}\otimes_{\Lambda} Q(\Lambda)$ (for example, one can take $s$ sending $c_{\infty}$ to $\mathcal{E}(\chi_1,\chi_2)/A(\chi_1,\chi_2)$).  Note that this map is unique up to a scalar, since both of the domain and the range of $s$ are $1$-dimensional $Q(\Lambda)$-vector space. Also, note that $\mathcal{G}_0$ may not be in the image of $s$ as the $\mathcal{E}(\chi_1,\chi_2)/A(\chi_1,\chi_2)$ may not be in $\wt{M}_{\Lambda}$. Then, we obtain an exact sequence
\begin{equation}\label{eq:Lambda-splitting}
0\leftarrow \wt{S}_{\Lambda}\otimes_{\Lambda}Q(\Lambda) \xleftarrow{t}  \wt{M}_{\Lambda}\otimes_{\Lambda}Q(\Lambda) \xleftarrow{s} \Lambda\cdot c_{\infty}\otimes_{\Lambda} Q(\Lambda) \leftarrow 0
\end{equation}

\begin{thm}\label{713}
Let the notation be as above. Then the congruence module attached to the data (\ref{eq:25}) and (\ref{eq:Lambda-splitting}) is $\Lambda/(A(\chi_1,\chi_2))$. 
\end{thm}

\begin{proof}
From the definition of the splitting map $s$, we have 
$$
s(\Lambda\cdot c_\infty \otimes_{\Lambda} Q(\Lambda))\cap \wt{M}_{\Lambda}= \mathcal{E}(\chi_1,\chi_2)\Lambda,
$$
and hence, the congruence module is isomorphic to $\Lambda/C_0(\mathcal{E}(\chi_1,\chi_2))\Lambda=\Lambda/A(\chi_1,\chi_2)$.
\end{proof}

By the same argument as in Theorem~\ref{713}, one can show that the congruence module attached to the first short exact sequence of (\ref{eq:1}) is also $\Lambda/(A(\chi_1,\chi_2))$. This proves the first part of Theorem~\ref{11}.

%%%%%%%%%%%%%%%%%%%%%%%%%%%%%%%%%%%%%%%%%%%%%%%%%%%%%%%%%%%%%%%%%%%%%%%%%%%%%%%%%%%
\subsection{Proof of main results: part $2$}\label{sec:62}
Let the notation be as in the previous subsection.
%From now on, we assume that $A(\chi_1,\chi_2)$ is not a unit in $\Lambda$ since otherwise, it is not an interesting case as explained in the introduction.
%there is no non-trivial congruence between $\Lambda$-adic cusp forms and $\mathcal{E}(\chi_1,\chi_2)$, which is not an interesting case.
Let $\mathcal{H}\subset \End_{\Lambda}(\wt{M}_{\Lambda})$ and $h\subset \End_{\Lambda}(\wt{S}_{\Lambda})$ be the $\Lambda$-algebras generated by all Hecke operators. The goal of this subsection is to prove Theorem~\ref{12}. 
The first step is to construct a nice basis of $\wt{M}_{\Lambda}$. Since $\wt{S}_{\Lambda}$ is a finitely generated free $\Lambda$-module, we may let $\{\mathcal{F}_{1},\ldots,\mathcal{F}_{m}\}$ be a basis of $\wt{S}_{\Lambda}$ over $\Lambda$. Since $\Hom_{\Lambda}(\wt{S}_{\Lambda},\Lambda)\isom h$ \cite[\S 3]{Hida1}, there exists a $\Lambda$-basis $\{h_1,\ldots, h_m\}$ such that 
$$
C(1,h_j\cdot \mathcal{F}_i)=
\begin{cases}
1&\mbox{if } i\neq j\\
0&\mbox{otherwise}.
\end{cases}
$$
For each $i$, let $H_i\in \mathcal{H}$ map to $h_i$ via the natural projection $\mathcal{H}\surj h$. 
Let 
$$
\mathcal{F}_0=\mathcal{G}_0-\sum_{i=1}^m C(1, H_i\cdot \mathcal{F})\mathcal{F}_i\in \wt{M}_{\Lambda},
$$
where $\mathcal{G}_0$ was defined in the previous subsection satisfying $A(\chi_1,\chi_2)C_0(\mathcal{G}_0)=C_0(\mathcal{E}(\chi_1,\chi_2))$.
Then, we have 
\begin{equation}\label{eq:const_of_F0}
C_{\lambda}(0,\mathcal{F}_0)=C_{\lambda}(0,\mathcal{G}_0)= u
\end{equation}
for some $u\in \Lambda^{\times}$ and for all $\lambda=1,\ldots,h_F^+$. By the definition of $\wt{M}_{\Lambda}$ and $\wt{S}_{\Lambda}$,  the $\Lambda$-rank of $\wt{S}_{\Lambda}$ is one less than the $\Lambda$-rank of $\wt{M}_{\Lambda}$. Thus, it follows that the set $\{\mathcal{F}_0,\ldots, \mathcal{F}_m\}$ is a $\Lambda$-basis of $\wt{M}_{\Lambda}$. Before we move to the second step, we make some observations on $\mathcal{F}_0$.

\begin{prop}\label{721}
Let the notation be as above.
\begin{enumerate}
\item We have $C_0(A(\chi_1,\chi_2)\mathcal{F}_0-\mathcal{E}(\chi_1,\chi_2))=0$. 

\item We can write $\mathcal{F}_0=\frac{\mathcal{E}(\chi_1,\chi_2)-\mathcal{F}_S}{A(\chi_1,\chi_2)}$ for some $\mathcal{F}_S\in \wt{S}_{\Lambda}$.

\item Modulo $\wt{S}_{\Lambda}$, $\mathcal{F}_0$ is an eigenform whose eigenvalues are the same as those of $\mathcal{E}(\chi_1,\chi_2)$.

\item We have $C(1,H_i\cdot \mathcal{F}_0)=0$ for all $i=1,\ldots, m$.
\end{enumerate}
\end{prop}

\begin{proof}
(1) follows from the construction of $\mathcal{F}_0$. (2) and (3) follow from (1) and the exactness of (\ref{eq:25}). For (4), we have
\[
\begin{split}
C(1, H_i\cdot \mathcal{F}_0)
= C(1,H_i\cdot \mathcal{F})-\sum_{j=1}^m C(1, H_j\cdot \mathcal{F}) C(1, H_i\cdot \mathcal{F}_j)
= C(1,H_i\cdot \mathcal{F})-C(1, H_i\cdot \mathcal{F})
= 0
\end{split}
\]
for all $i=1,\ldots,m$.
\end{proof}

We now fix a cusp form $\mathcal{F}_S$ satisfying Proposition~\ref{721}(2). It is congruent to the Eisenstein series $\mathcal{E}(\chi_1,\chi_2)$ modulo $A(\chi_1,\chi_2)$. We define the surjective $\Lambda$-module homomorphism in Theorem~\ref{12} as
$$
\Psi:h/I\map \Lambda/A(\chi_1,\chi_2);\; H\mapsto C(1,H\cdot \mathcal{F}_S).
$$
Set
$$
\wt{M}'_{\Lambda}=\{\mathcal{F}\in \wt{M}_{\Lambda}\otimes_{\Lambda}Q(\Lambda)\mid C(\mfr{a},\mathcal{F})\in \Lambda \mbox{ for all nonzero integral ideals }\mfr{a}\}.
$$
The following theorem describes equivalent statements of $\Psi$ being an isomorphism when $\chi_2=\mathbbm{1}$.

\begin{thm}\label{723}
Let the notation be as above. Assume that $\chi_2=\mathbbm{1}$. Then the following are equivalent.
\begin{enumerate}
\item There exists an element $H\in \mathcal{H}$ such that $C(1,H\cdot \mathcal{F}_0)\in \Lambda^{\times}$.

\item There exists an element $H\in \mathcal{H}$ such that $C(1, H\cdot \mathcal{F})=C_{\lambda}(0,\mathcal{F})$ for all $\mathcal{F}\in \wt{M}_{\Lambda}$ and for all $\lambda=1,\ldots, h^+_F$.

\item The homomorphism $\Psi: h/I\surj \Lambda/\wh{G}_{\chi_1}(T)$ is an isomorphism.

\item We have $\wt{M}_{\Lambda}=\wt{M}'_{\Lambda}$. In particular, we have an isomorphism of $\Lambda$-modules $\Hom_{\Lambda}(\wt{M}_{\Lambda},\Lambda)\isom \mathcal{H}$.
\end{enumerate}
\end{thm}

\begin{proof}
First, we prove $(1)\Rightarrow (2)$. Let $H\in \mathcal{H}$ satisfy $C(1, H\cdot \mathcal{F}_0)\in \Lambda^{\times}$. We define
$$
H_0=\frac{u}{C(1,H\cdot \mathcal{F}_0)}\left(H-\sum_{i=1}^{m} C(1, H\cdot \mathcal{F}_i) H_i\right)\in \mathcal{H}.
$$
We claim that $C(1,H_0\cdot F)=C_{\lambda}(0,F)$ for all $F\in \wt{M}_{\Lambda}$ and $\lambda=1,\ldots, h^+_F$. 
Since the set $\{\mathcal{F}_0,\ldots, \mathcal{F}_m\}$ is a $\Lambda$-basis of $\wt{M}_{\Lambda}$, it suffices to show that $C(1,H_0\cdot \mathcal{F}_j)=0$ for all $j=1,\ldots, m$ and $C(1, H_0\cdot \mathcal{F}_0)=C_{\lambda}(0,\mathcal{F}_0)$. For $1\leq j\leq m$, we have 
\[
\begin{split}
C(1, H_0\cdot \mathcal{F}_j)
&=\frac{u}{C(1, H\cdot \mathcal{F}_0)} \left(C(1,H\cdot \mathcal{F}_j)-\sum_{i=1}^m C(1, H\cdot \mathcal{F}_i) C(1,H_i\cdot \mathcal{F}_j)\right)\\
&=\frac{u}{C(1, H\cdot \mathcal{F}_0)}(C(1, H\cdot \mathcal{F}_j)-C(1, H\cdot \mathcal{F}_j))=0.
\end{split}
\]
Moreover, we have
\[
\begin{split}
C(1,H_0\cdot \mathcal{F}_0)
= \frac{u}{C(1, H\cdot \mathcal{F}_0)}\left(C(1, H\cdot \mathcal{F}_0)-\sum_{i=1}^m C(1, H\cdot \mathcal{F}_i) C(1,H_i\cdot \mathcal{F}_0)\right)
=^{(*)} u
= C_{\lambda}(0,\mathcal{F}_0)
\end{split}
\]
for all $\lambda=1,\ldots,h_F^+$. Note that the equality $(*)$ is obtained by Proposition~\ref{721}(4) and the last equality is obtained by (\ref{eq:const_of_F0}). Thus the statement (2) holds.

Next, we show that $(2)\Rightarrow (3)$. Let $H\in \mathcal{H}$ satisfy $C(1, H\cdot \mathcal{F})=C_{\lambda}(0,\mathcal{F})$ for all $\lambda=1,\ldots,h_F^+$ and for all $\mathcal{F}\in \wt{M}_{\Lambda}$. We define
$$
H_0=u^{-1}H\in \mathcal{H}.
$$
Then, one has $C(1,H_0\cdot \mathcal{S})=0$ for all $\mathcal{S}\in \wt{S}_{\Lambda}$. Moreover, one has 
$$
C(1, H_0\cdot \mathcal{F}_0)= u^{-1}C(1,H\cdot \mathcal{F}_0)=u^{-1}C_{\lambda}(0,\mathcal{F}_0)=1.
$$
Recall that the $\Lambda$-homomorphism 
$$
\Psi:h\map \Lambda/(\wh{G}_{\chi_1}(T));\; T\mapsto C(1,T\cdot \mathcal{F}_S)
$$
is surjective. Therefore, it suffices to show that $\ker\Psi=I$. It is clear that $I\subset \ker \Psi$. It remains to show $\ker\Psi\subset I$. Now given any $T_S\in \ker\Psi\subset h$, we let $H'\in \mathcal{H}$ be any lifting of $T_S$ via $\mathcal{H}\surj h$. We define
$$
H''=H'-\frac{C(1,H'\cdot \mathcal{E}(\chi_1,\mathbbm{1}))}{C(1,H_0\cdot \mathcal{E}(\chi_1,\mathbbm{1}))}H_0.
$$
We claim that $H''$ kills $\mathcal{E}(\chi_1,\mathbbm{1})$. To see this, we note that
$$
C(1,H''\cdot \mathcal{E}(\chi_1,\mathbbm{1}))=
C(1,H'\cdot \mathcal{E}(\chi_1,\mathbbm{1})) -\frac{C(1,H'\cdot \mathcal{E}(\chi_1,\mathbbm{1}))}{C(1,H_0\cdot \mathcal{E}(\chi_1,\mathbbm{1}))}C(1, H_0\cdot \mathcal{E}(\chi_1,\mathbbm{1}))=0.
$$
Since $H''$ projects to $T_S$, we obtain that $T_S\in I$. Thus the statement (3) holds.

Next, we show that $(3)\Rightarrow (4)$. It is clear that $\wt{M}_{\Lambda}\subset \wt{M}'_{\Lambda}$. We have to show that $\wt{M}'_{\Lambda}\subset \wt{M}_{\Lambda}$. Given any $\mathcal{F}\in M'_{\Lambda}$, we write $\mathcal{F}=\frac{P}{Q}\mathcal{E}(\chi_1,\mathbbm{1})+\frac{S}{R}f$ for some $f\in \wt{S}_{\Lambda}$ and $P,Q,S,R\in \Lambda$ with $(P,Q)=1$, $(S,R)=1$ and $Q,T\neq 0$. To show $\mathcal{F}\in \wt{M}_{\Lambda}$, it suffices to show that $C_{\lambda}(0,\mathcal{F})\in \Lambda$ for all $\lambda=1,\ldots, h^+_F$. Moreover, we know that $C_{\lambda}(0,\mathcal{F})=\frac{P}{Q}C_{\lambda}(0,\mathcal{E}(\chi,\mathbbm{1}))=\frac{P}{Q}\wh{G}_{\chi_1}(T)$ for all $\lambda=1,\ldots, h_F^+$, so it is enough to show that $Q$ divides $\wh{G}_{\chi_1}(T)$. We set $\mathcal{F}':=Q \cdot \frac{S}{R}f=Q\mathcal{F}-P\mathcal{E}(\chi_1,\mathbbm{1})\in \wt{S}_{\Lambda}$. Then $\mathcal{F}'$ has the same eigenvalues as those of $\mathcal{E}(\chi_1,\mathbbm{1})$ modulo $Q$. We obtain a surjective homomorphism of $\Lambda$-modules $h/I\surj P\cdot \Lambda/Q\isom \Lambda/Q$ defined by $H\mapsto C(1,H\cdot \mathcal{F}')$. Since the statement (3) holds, we have $\Lambda/\wh{G}_{\chi_1}(T)\isom h/I\surj \Lambda/Q$, which implies that  $Q$ divides $\wh{G}_{\chi_1}(T)$. 

Finally, we show that $(4)\Rightarrow (1)$. Since $\mathcal{H}\isom \Hom_{\Lambda}(\wt{M}_{\Lambda},\Lambda)$, there exist $H_0,\ldots, H_m$ in $\mathcal{H}$ satisfying $C(1, H_i\cdot \mathcal{F}_j)=\delta_{i,j}$ for $i,j=0,\ldots,m$, where $\delta_{i,j}$ is $1$ if $i=j$ and is $0$ if $i\neq j$. In particular, $C(1, H_0\cdot \mathcal{F}_0)=1\in \Lambda^{\times}$, and hence (1) holds.
\end{proof}

To complete the proof of Theorem~\ref{12}, we prove the following lemma.

\begin{lemma}\label{722}
Let the notation and assumption be as in Theorem~\ref{723}. Then there exists an element $H\in \mathcal{H}$ such that $C(1, H\cdot \mathcal{F}_0)\in \Lambda$ is a unit.
\end{lemma}

\begin{proof}
Suppose that $C(1,H\cdot \mathcal{F}_0)$ is not a unit in $\Lambda$ for all $H\in \mathcal{H}$. Note that $\Lambda$ is a local ring with maximal ideal $(T,\varpi)$, where $\varpi$ is a uniformizer of $\Zp[\chi_1]$. Also, note that $C(\mfr{a}, \mathcal{F}_0)=C(1,T({\mfr{a}})\cdot \mathcal{F}_0)\in (T,\varpi)$ for all integral $\mfr{a}$ of $\O_F$ since $\Lambda$ is a local ring and $C(1,T({\mfr{a}})\cdot \mathcal{F}_0)$ is not a unit. Moreover, since $C(\mfr{p},\mathcal{E}(\chi_1,\mathbbm{1}))=1$ for all $\p|p$, by Proposition~\ref{721}(2), one obtains that $U(\p)\cdot \mathcal{F}_0=\mathcal{F}_0+\mathcal{F}_{T{\p}}$ for some $\mathcal{F}_{U_{\p}}\in \wt{S}_{\Lambda}$. Thus for all integral ideals $\mfr{a}$ and for all prime ideals $\p|p$, we have
$$
C(1,T({\mfr{a}})U({\p})\cdot \mathcal{F}_0)=C(\mfr{a},U({\p})\cdot \mathcal{F}_0)=C(\mfr{a},\mathcal{F}_0)+C(\mfr{a}, \mathcal{F}_{U_{\p}}).
$$
Therefore, we know that $C(\mfr{a}, \mathcal{F}_{U_{\p}})\in (T,\varpi)$ for all integral ideals $\mfr{a}$ and prime ideals $\p|p$.

Let $f_2=v_{2,1}(\mathcal{F}_0)\in M_2^{\ord}(\n p,\chi_1;\Zp[\chi_1])$, and let $f_0=(C_{\lambda}(0, v_{2,1}(\mathcal{F}_0)))_{\lambda=1}^{h_F^+}$. By the construction of $\mathcal{F}_0$ and the assumption on its coefficients, we know that 
$$
f_2\equiv f_0\bmod \varpi\mbox{ and } S(\mfr{q})f_2\equiv \chi_1(\mfr{q})f_2\bmod \varpi,
$$
for all prime ideals $\mfr{q}$ coprime to $\mfr{n}p$. Thus for any prime ideal $\mfr{q}$ coprime to $\mfr{n}p$, we have
\[
\begin{split}
(\chi_1(\mfr{q})-1)f_0&= \chi_1(\mfr{q})f_0-f_0
\equiv \chi_1(\mfr{q})f_2-f_0%\bmod \varpi
%\equiv \chi_1(\mfr{q})f_2-f_0%\bmod \varpi
\equiv S(\mfr{q})(f_2-f_0)\equiv 0\bmod \varpi.
\end{split}
\]
Furthermore, since $\chi_1$ is not a trivial character, we can choose a prime ideal $\mfr{q}$ such that $\chi_1(\mfr{q})-1$ is not congruent to zero modulo $\varpi$. Hence, there exists an element $H\in \mathcal{H}$ such that $C(1,H\cdot \mathcal{F}_0)\in \Lambda^{\times}$.
\end{proof}

\begin{cor}\label{68}
The $\Lambda$-module homomorphism $\Psi:h/I\to \Lambda/(A(\chi_1,\chi_2))$ is an isomorphism. In particular, one has $\Hom_{\Lambda}(\wt{M}_{\Lambda},\Lambda)\isom \mathcal{H}$.
\end{cor}

\begin{proof}
When $\chi_2$ is a trivial character, the assertion follows from Theorem~\ref{723} and Lemma~\ref{722}.
When $\chi_2$ is not a trivial character, Theorem~\ref{723}(4) holds automatically since $C_{\lambda}(0,\mathcal{F})=0$ for all $\mathcal{F}\in \wt{M}_{\Lambda}$. Note that in the proof of Theorem~\ref{723}, the idea to prove $(2)\Rightarrow (3)$ is to construct $H_0\in \mathcal{H}$ such that $C(1, H_0\cdot \mathcal{F}_0)=1$, which automatically exists if Theorem~\ref{723}(4) holds. Thus, by the same argument as in Theorem~\ref{723}, we again see that $\Psi$ is an isomorphism if $\chi_2$ is nontrivial.
\end{proof}

Before we move on, we recall the following result of Ohta.

\begin{lemma} \cite[Lemma 1.1.4]{Ohta}\label{731}
Let $R$ be an integral domain with quotient field $Q(R)$, and let
$$
0\map A \xrightarrow{i} B\xrightarrow{p}  C \map 0.
$$
be a short exact sequence of flat $R$-modules. Assume that we are given splitting maps after tensoring with $Q(R)$
over $R$, i.e., we have
$$
0 \leftarrow A\otimes_R Q(R) \xleftarrow{t}  B\otimes_R Q(R) \xleftarrow{s} C\otimes_R Q(R) \leftarrow 0.
$$
Then we have an isomorphism of $R$-modules
$$
\mathcal{C}_s:=C/p(B\cap s(C\otimes 1))\isom t(B\otimes 1)/A.
$$
\end{lemma}

In the remainder of this paper, we discuss two applications of Theorem~\ref{723}. We first compute the congruence module attached to 
\begin{equation}\label{eq:cong2}
0\map \mathcal{I}\map \mathcal{H}\map \Lambda\map 0,
\end{equation}
where the surjecive homomorphism $\mathcal{H}\map \Lambda$ is defined by $T\mapsto C(1,T\cdot \mathcal{E}(\chi_1,\chi_2))$. Let $\mathcal{E}_{\Lambda}$ be as in the proof of Theorem~\ref{713}. For simplicity, we will write $\wt{M}_{Q(\Lambda)}=\wt{M}_{\Lambda}\otimes_{\Lambda} Q(\Lambda)$ and write $\mathcal{E}_{Q(\Lambda)}$ and $\wt{S}_{Q(\Lambda)}$ in the same manner. By (\ref{eq:splitting}), we have 
\begin{equation}\label{eq:73}
\mathcal{H}\otimes_{\Lambda} Q(\Lambda)=\Hom_{Q(\Lambda)}(\wt{M}_{Q(\Lambda)}, Q(\Lambda))=\Hom_{Q(\Lambda)}(\wt{S}_Q(\Lambda), Q(\Lambda))\oplus \Hom_{Q(\Lambda)}(\mathcal{E}_{Q(\Lambda)}, Q(\Lambda)).
\end{equation}
We consider the splitting $s:Q(\Lambda)\to\mathcal{H}\otimes_{\Lambda} Q(\Lambda)$ whose image is  $\Hom_{\Lambda}(\mathcal{E}_{\Lambda}, Q(\Lambda))\isom Q(\Lambda)$ (for example, one can take $s$ sending $1\in Q(\Lambda)$ to an element in $\Hom_{Q(\Lambda)}(\mathcal{E}_{Q(\Lambda)},Q(\Lambda))$ sending $\mathcal{E}(\chi_1,\chi_2)/A(\chi_1,\chi_2)$ to $1$). Then, we have 
\begin{equation}\label{eq:hekce_cong_module}
0\leftarrow \mathcal{I}\otimes_{\Lambda} Q(\Lambda) \xleftarrow{t}  \mathcal{H}\otimes_{\Lambda} Q(\Lambda) \xleftarrow{s} \Lambda\otimes_{\Lambda} Q(\Lambda) \leftarrow 0.
\end{equation}

\begin{cor}\label{main_thm2}
Let the notation be as above. Then the congruence module attached to the data, (\ref{eq:cong2}) and (\ref{eq:hekce_cong_module}),
is $\Lambda/A(\chi_1,\chi_2)$ 
\end{cor}

\begin{proof}
By Corollary~\ref{68}, it suffices to show that the congruence module associated to these data is $h/I$. It follows from (\ref{eq:73}) that we have
$$
\mathcal{I} \otimes_{\Lambda} Q(\Lambda)\isom \Hom_{Q(\Lambda)}(\wt{S}_{Q(\Lambda)}\otimes Q(\Lambda))\isom h\otimes_{\Lambda} Q(\Lambda).
$$
Moreover, the image of $\mathcal{H}$ in $h\otimes_{\Lambda} Q(\Lambda)$ is identified with $h\subset h\otimes_{\Lambda} Q(\Lambda)$, and the image of $\mathcal{I}$ in $h$ is identified with $I$. Thus, by Lemma~\ref{731}, we see that the congruence module is $t(\mathcal{H})/I=h/I$.
\end{proof}

\begin{cor}\label{610}
The $\Lambda$-module $\Ann_{\mathcal{H}}(\mathcal{I})$ is free of rank $1$.
\end{cor}

\begin{proof}
Note that one has $\mathcal{I}\cap \ker(\mathcal{H}\to h)=0$ as an operator in $\mathcal{H}$ that acts trivially on the space generated by Eisenstein series and acts trivially on the space of cusp forms has to be $0$. By Theorem~\ref{723}, one has the following commutative diagram:
$$
\xymatrix{
  &  0     \ar[r]\ar[d]   & \ker(\mathcal{H}\to h)\ar[r]\ar[d]  & (A(\chi_1,\chi_2))\ar[r]\ar[d] & 0 \\
0\ar[r] &\mathcal{I}\ar[r]\ar[d]   &  \mathcal{H} \ar[r]\ar[d]  &   \Lambda\ar[r]\ar[d]          & 0 \\
0\ar[r] & I        \ar[r]\ar[d]    &h      \ar[r]\ar[d]         & h/I\isom \Lambda/(A(\chi_1,\chi_2))   \ar[r]\ar[d]& 0\\
& 0 & 0 & 0 &,
}
$$
which yields an isomorphsim of $\Lambda$-modules $\Ann_{\mathcal{H}}(\mathcal{I})\isom \ker(\mathcal{H}\to h)\isom \Lambda\cdot A(\chi_1,\chi_2)$. Thus the assertion follows.
\end{proof}

%%%%%%%%%%%%%%%%%%%%%%%%%%%%%%%%%%%%%%%%%%%%%%%%%%%%%%%%%%%%%%%%%%%%%%%%%%%
%\newpage


\begin{thebibliography}{Xyz12}

\bibitem[BDP]{BDP} A. Betina, M. Dimitrov, and A. Pozzi, On the failure of Gorensteinness at weight $1$ Eisenstein points of the eigencurve, arXiv:1804.00648.
 
\bibitem[Bum]{B} D. Bump, Automorphic forms and representations, Cambridge Studies in Advanced Mathematics 55, Cambridge University Press, Cambridge, 1997.

\bibitem[Cas]{Cas} W. Casselman, On Some Results of Atkin and Lehner, Math. Ann. \textbf{201} (1973), 301-314.

\bibitem[CF]{Cas-Fro} J. W. S. Cassels and A. Frohlich, Algebraic Number Theory, New York, Academic Press., 1968.

\bibitem[Cha]{cha} C.-L. Chai, Arithmetic minimal compactification of the Hilbert-Blumenthal moduli space, Appendix to \cite{W3}, Ann. of Math. 
\textbf{131} (1990), 541-554.

%\bibitem[Col]{Col} P. Colmez, R\'{e}sidu en s = 1 des fonctions zeta p-adiques, Invent. Math. 91(1988), 371-389.

%\bibitem[Dpp]{DDP} S. Dasgupta, H. Darmon, and R. Pollack, Hilbert modular forms and the Gross-Stark conjecture, Ann of Math. 174 (2011), 439-484.

\bibitem[DR]{DR} P. Deligne and K. A. Ribet, Values of abelian $L$-functions at negative integers over totally real fields, Invent. Math. \textbf{59} (1980), 227-286.

\bibitem[DS]{DS} F. Diamond and J. Shurman, A first course in modular forms, Grad. Text. Math. 228, Springer-Verlag, Springer-Verlag, New York 2005.


\bibitem[Dim]{Dim} M. Dimitrov, Compactifications arithm\'{e}tiques des vari\'{e}t\'{e}s de Hilbert et forms modulaires de Hilbert pour $\Gamma_1(\mfr{c},\mfr{n})$, Geometric Aspects of Dwork Theory (A. Adolphson, F. Baldassarri, P. Berthelot, N. Katz and F. Loeser, eds.), Walter de Gruyter, Berlin (2004), 525-551.

%\bibitem[Dim2]{Dim2} M. Dimitrov, On Ihara's lemma for Hilbert modular varieties, Compos. Math. 145 (2009), 1114-1146.

\bibitem[DT]{DT} M. Dimitrov and J. Tilouine, Vari\'{e}t\'{e}s et forms modulaires de Hilbert arith\'{e}tiques pour $\Gamma_1(\mfr{c},\mfr{n})$. Geometric Aspects of Dwork Theory (A. Adolphson, F. Baldassarri, P. Berthelot, N. Katz and F. Loeser, eds.), Walter de Gruyter, Berlin (2004), 553-609.

\bibitem[DW]{DW}  M. Dimitrov and G. Wiese. Unramifiedness of Galois representations attached to
weight one Hilbert modular eigenforms mod p. J. Inst. Math. Jussieu (2018), 1-26. 

\bibitem[Eme]{E} M. Emerton, The Eisenstein ideal in Hida's ordinary Hecke algebra, Int. Math. Res. Not. \textbf{15} (1999), 793-802.

\bibitem[Gor]{Gor} E. Z. Goren, Lectures on Hilbert modular varieties and modular forms. CRM Monograph Series, 2001.

\bibitem[HP]{HP} G Harder and R. Pink, Modular konstruierte unverzweigte abelsche $p$-Erweiterungen von $\Q(\zeta_p)$ und die Struktur ihrer Galoisgruppen, Math. Nachr. \textbf{159} (1992), 83-99.

\bibitem[Hid1]{Hida1} H. Hida, On $p$-adic Hecke algebras for $\GL_2$ over totally real fields, Annals of Math. \textbf{128} (1988), 295-384.

\bibitem[Hid2]{Hida2} H. Hida, Elementary theory of $L$-functions and Eisenstein series, London Math. Soc. Stud. Texts, vol. 26, Cambridge Univ. Press, 1993.

\bibitem[Hid3]{Hida3} H. Hida, $p$-adic automorphic forms on Shimura varieties, Springer Monographs in Mathematics, Springer-Verlag, New York, 2004.

\bibitem[Hsi]{Hsieh} M.-L. Hsieh, Eisenstein congruence on unitary groups and Iwasawa main conjectures
for CM fields, J. Amer. Math. Soc. \textbf{27-3} (2014), 753-862.

\bibitem[Kat]{Kat} N. Katz, $p$-adic $L$-function for CM fields, Invent. Math. \textbf{49} (1978), 199-297.

%\bibitem[M]{Miy} T. Miyake, Modular forms, Springer-Verlag, New York, 1989, translated from the Japannese by Yoshitaka Maeda.

\bibitem[Laf]{Matt} M. Lafferty, Eichler-Shimura cohomology groups and the Iwasawa main conjecture, Ph.D. thesis (2015).
              
\bibitem[MW]{MW} B. Mazur and A. Wiles, Class fields of abelian extension of $\Q$, Invent. Math. \textbf{76} (1984), 179-330.

\bibitem[Neu]{Neu} J. Neukirch and N. Schappacher, Algebra number theory, vol. 9, Springer Berlin, 1999.

%\bibitem[Ohta1]{Ohta1} M. Ohta, On the $p$-adic Eichler-Shimura isomorphism for $\Lambda$-adic cusps froms, J. Reine Angew. Math. 463 (1995), 49-98.

\bibitem[Oht1]{Ohta2} M. Ohta, Ordinary $p$-adic \'{e}tale cohomology groups attached to towers of elliptic modular curves, Compos. Math. \textbf{115} (1999), 241-301.

\bibitem[Oht2]{Ohta3} M. Ohta,  Ordinary $p$-adic \'{e}tale cohomology groups attached to towers of elliptic modular curves II, Math. Ann. \textbf{318} (2000), 557-583.

\bibitem[Oht3]{Ohta} M. Ohta, Congruence modules related to Eisenstein series, Ann. Sci. \'{E}cole Norm Sup.~(4) \textbf{36} (2003), 225-269.

\bibitem[Oza]{Oz} T. Ozawa, Constant terms of Eisenstein series over a totally real field, 	Int. J. Number Theory (2017), 309-324.
             
\bibitem[Rap]{Rap} M. Rapoport, Compactifications de l'espace de modules de Hilbert-Blumenthal, Compos. Math. \textbf{36} (1978), 255-335.

\bibitem[Shi]{Sh} G. Shimura, The special values of the zeta functions associated with Hilbert modular forms, Duke Math. J. \textbf{45} (1978), 637-679.

\bibitem[Sch]{Sch} R. Schmidt, Some remarks on local newforms for $\GL(2)$, J. Ramanujan Math. Soc. \textbf{17} (2002), no. 2, 115-147.

\bibitem[Sha]{Sha} R. Sharifi, A reciprocity map and the two-variable $p$-adic $L$-function, Ann. of Math. \textbf{173} (2011), 251-300.

\bibitem[Tat]{Tate} J. Tate, Number theoretic background, Automorphic forms, representations and L-functions (Proc. Sympos. Pure Math., Oregon State Univ., Corvallis, Ore., 1977), Part 2, 3–26. 

\bibitem[Ven]{Ven} K. Ventullo, On the rank one abelian Gross-Stark conjecture,  Comment. Math. Helv. \textbf{90} (2015), 939-963. 

\bibitem[Wan]{Jun} J. Wang, On the Jacobi sums mod $P^n$, J. Number Theory \textbf{39} (1991), 50-64.

%\bibitem[Wa]{Wa} L. Washington, Introduction to cyclotomic fields, Second edition, Grad. Texts. Math. 83, Springer-Verlag, New York, 1997.

\bibitem[Wil1]{W1} A. Wiles, On $p$-adic representations for totally real fields, Ann. of Math. \textbf{123} (1986), 407-456.

\bibitem[Wil2]{W2} A. Wiles, On ordinary $\lambda$-adic representations associated to modular forms, Invent. Math. \textbf{94} (1988), 539-573.

\bibitem[Wil3]{W3} A. Wiles, The Iwasawa conjecture for totally real fields, Ann. of Math \textbf{131} (1990), 493-540.
\end{thebibliography}
\end{document}